\documentclass[hidelinks,onefignum,onetabnum]{siamart220329}

\usepackage{lipsum}
\usepackage{amsfonts}
\usepackage{graphicx}
\usepackage{epstopdf}
\ifpdf
  \DeclareGraphicsExtensions{.eps,.pdf,.png,.jpg}
\else
  \DeclareGraphicsExtensions{.eps}
\fi

\newsiamremark{remark}{Remark}
\newsiamremark{hypothesis}{Hypothesis}
\crefname{hypothesis}{Hypothesis}{Hypotheses}
\newsiamthm{claim}{Claim}

\headers{SOC-MartNet for HJB equations and SOCPs}{Wei Cai, Shuixin Fang, and Tao Zhou}

\title{SOC-MartNet: A Martingale Neural Network for the Hamilton-Jacobi-Bellman  Equation without Explicit $\inf_{u \in U} H$ in Stochastic Optimal Controls
\thanks{This work of SF and TZ is supported by the NSF of China (under grant 12288201) and the Youth Innovation Promotion Association (CAS). Date. July 8, 2024. {\it Previous version of this manuscript is published on arxiv arXiv:2405.03169, May 6, 2024}}
}

\author{Wei Cai\thanks{Department of Mathematics, Southern Methodist University, Dallas, TX 75275, USA.
Corresponding author.
  (\email{cai@smu.edu}).}
\and Shuixin Fang\thanks{Institute of Computational Mathematics and Scientific/Engineering Computing,
Academy of Mathematics and Systems Science, Chinese Academy of Sciences, Beijing, 100190, P. R. China.
  (\email{sxfang@amss.ac.cn}).}
\and Tao Zhou\thanks{Institute of Computational Mathematics and Scientific/Engineering Computing,
Academy of Mathematics and Systems Science, Chinese Academy of Sciences, Beijing, 100190, P. R. China. Corresponding author.
  (\email{tzhou@lsec.cc.ac.cn}).}
}

\usepackage{amsopn}

\usepackage{cleveref}
\usepackage[sort]{cite}

\usepackage{booktabs}
\usepackage{multirow}

\usepackage{algpseudocode}
\algrenewcommand\algorithmicrequire{\textbf{Input:}}
\algrenewcommand\algorithmicensure{\textbf{Output:}}

\newcommand{\R}{\mathbb{R}}
\newcommand{\di}{\,\mathrm{d}}

\newcommand{\M}{\mathcal{M}}

\newcommand{\E}[1]{\mathbb{E}\left[#1\right]}

\newcommand{\br}[1]{\left(#1\right)}
\newcommand{\abs}[1]{\left\vert#1\right\vert}

\newcommand{\bbr}[1]{\left\{#1\right\}}

\newcommand{\vbr}[2]{#1(#2 #1)}
\newcommand{\rbr}[1]{\left[#1\right]}

\newcommand{\argmin}{\mathop{\arg\min}}

\newcommand{\B}[1]{\boldsymbol{#1}}
\newcommand{\sptext}[1]{\;\; \text{#1} \;\;}

\usepackage{graphicx}
\usepackage{float}
\usepackage[caption=false]{subfig}
\usepackage[justification=centering]{caption}

\usepackage{pgfplots}
\pgfplotsset{width=10cm,compat=1.9}

\newcommand{\pathhjbfigd}{CodeAndFig/SOCMN_HJB3te1e-2_d1000/outputs/}
\newcommand{\pathhjbfige}{CodeAndFig/SOCMN_HJB3te5e-2to1_d1000/outputs/}
\newcommand{\pathhjbfigf}{CodeAndFig/SOCMN_HJB3_d50_100/outputs/}
\newcommand{\pathhjbfigh}{CodeAndFig/SOCMN_HJBv3_d2e3_te1/outputs/}
\newcommand{\pathhjbfigi}{CodeAndFig/SOCMN_HJBv3_d2e3_te0.5/outputs/}
\newcommand{\pathhjbfigj}{CodeAndFig/SOCMN_HJBv3_d2e3_te0.1/outputs/}
\newcommand{\pathhjbfigk}{CodeAndFig/SOCMN_HJBv3_d2e3_te0.01/outputs/}

\newcommand{\pathhjbfig}{CodeAndFig/Non-deg_HJBv3_t0te/outputs/}
\newcommand{\pathreva}{CodeAndFig_Revision/Count/outputs/}
\newcommand{\pathrevb}{CodeAndFig_Revision/HJB2b/outputs/}
\newcommand{\pathrevc}{CodeAndFig_Revision/HJB2c/outputs/}
\newcommand{\pathrevd}{CodeAndFig_Revision/HJBst/}
\newcommand{\pathreve}{CodeAndFig_Revision/HJBeps/outputs/}
\newcommand{\pathrevf}{CodeAndFig_Revision/HJBpert/}

\newcommand{\cb}[1]{{\color{black} #1}}

\ifpdf
\hypersetup{
  pdftitle={SOC-MartNet},
  pdfauthor={Wei Cai, Shuixin Fang, and Tao Zhou}
}
\fi

\begin{document}

\maketitle

\begin{abstract}
    In this paper, we propose a martingale-based neural network, SOC-MartNet, for solving high-dimensional Hamilton-Jacobi-Bellman (HJB) equations where no explicit expression is needed for the infimum of the Hamiltonian, $\inf_{u \in U} H(t,x,u, z,p)$, and stochastic optimal control problems (SOCPs) with controls on both drift and volatility.
    We reformulate the HJB equations for the value function by training two neural networks, one for the value function and one for the optimal control with the help of two stochastic processes - a Hamiltonian process and a cost process.
    The control and  value networks are trained such that the associated Hamiltonian process is minimized to satisfy the minimum principle of a feedback SOCP, and the cost process becomes a martingale, thus, ensuring the value function network as the solution to the corresponding HJB equation.
    Moreover, to enforce the martingale property for the cost process, we employ an adversarial network and construct a loss function characterizing the projection property of the conditional expectation condition of the martingale.
    Numerical results show that the proposed SOC-MartNet is effective and efficient for solving HJB-type equations and SOCPs with a dimension up to \cb{$10,000$ in a small number of iteration steps (less than $6000$)} of training.
\end{abstract}



\section{Introduction}

In this paper, we consider the numerical solution of high-dimensional Hamilton-Jacobi-Bellman (HJB)-type equations and their applications to stochastic optimal control problems (SOCPs).
The considered HJB-type equation is given in the form of
\begin{equation}\label{eq_HJBPDE}
    \partial_t v(t, x) + \mathcal{L} v(t, x)
    + \inf_{\kappa \in U} H\br{t, x, \kappa, \partial_x v(t, x), \partial_{xx}^2 v(t, x)}  = 0
\end{equation}
for $(t, x) \in [0, T) \times \R^d$
with $\partial_x=\nabla_x$ and $\partial_{xx}^2=\nabla_x \nabla_x^{\top}$ as the gradient and Hessian operator, respectively, and a terminal condition
\begin{equation}\label{eq_term}
    v(T, x) = g(x), \quad x \in \R^d,
\end{equation}
where  $\mathcal{L}$ is a differential operator given by
\begin{equation} \label{generator}
    \mathcal{L} := \mu^{\top}(t, x) \partial_x + \frac{1}{2} \operatorname{Tr}\bbr{\sigma \sigma^{\top}\br{t, x} \partial_{xx}^2}
\end{equation}
for some given functions $\mu: [0, T] \times \R^d \to \R^d$ and $\sigma: [0, T] \times \R^d \to \R^{d\times q}$;
$H(t,x,\kappa, z,p)$ is the Hamiltonian as a mapping  $(t, x,\kappa, z,p) \in [0, T] \times \R^d \times  U \times \R^d \times \R^{d\times d} \to \R$, and $U \subset \R^m$ .
The HJB-type equation~\eqref{eq_HJBPDE} is general and covers common HJB equations  in SOCPs; see the discussions in section~\ref{sec_amart}.
In addition to this equation, this paper also explores the application of the proposed method to common semi-linear parabolic equations without boundary conditions.

The HJB equation is a fundamental partial differential equation (PDE) in the field of optimal control theory \cite{kushner2013,Fleming1975Deterministic,Yong1999Stochastic,Pham2009Continuous,carmona2016}.
In the typical framework of dynamic programming \cite{Bellman1957Dynamic,Fleming1975Deterministic}, the optimal feedback control is identified by the verification technique,
which involves minimizing a Hamiltonian depending on the derivatives of a value function \cite[p. 278]{Yong1999Stochastic}.
On this account, the HJB equation, which governs this value function, stands as a cornerstone of dynamic programming.
The well-posedness of HJB equations has been firmly  established with the theory of viscosity solutions; see, e.g., \cite{Lions1982Generalized,Ishii1989uniqueness,Crandall1992User}.
But solving the HJB equation is still challenging due to its non-smoothness and high dimensionality.

The wide application of HJB equations has spurred extensive research on efficient numerical methods.
Conventional approaches include the Galerkin method \cite{Smears2014Discontinuous},
the finite volume method \cite{osher1991shu,Richardson2006Numerical},
the monotone approximation scheme \cite{Barles2002On},
the patchy dynamic programming \cite{Cacace2012patchy}, etc.
These methods generally suffer from the curse of dimensionality (CoD) \cite{Bellman1957Dynamic}, that is, the computation complexity increases exponentially with the dimension of the HJB equations.
In \cite{Kharroubi2014numerical,Kharroubi2015Discrete}, the HJB equation is solved through the associated BSDE deduced from a Feynman-Kac representation \cite{Kharroubi2015Feynman}, \cb{but the resolution of BSDE relies on least square regression on a set of basis functions, 
where the CoD arises since the number of required basis functions explodes with the dimension.}
There are also literature leveraging dimension reduction techniques, e.g., \cite{Tensor2021Tensor,Mitigating2017Kang,Kalise2018Polynomial},
but these techniques depend on the dimension reducibility of the problem.

In recent years, deep learning has emerged as a promising tool to overcome the CoD, leading to a growing body of deep learning methods for solving high dimensional PDEs, e.g., \cite{weinan2017deep,han2018solving,Zhang2022FBSDE,Zang2020Weak,hure2020deep,Gao2023Failure,Guo2022Monte,Raissi2019Physics,hu2024sdgd}.
While demonstrably effective for usual high-dimensional PDEs, these methods encounter new challenges when applied to the HJB-type equation~\eqref{eq_HJBPDE}.
One of the challenges stems from the inherent infimum operator
$\inf_{\kappa \in U}$ imposed on the Hamiltonian $H$ in the HJB equation.
Directly minimizing the Hamiltonian for every time-space point $(t, x)$ is computationally expensive, in fact, a CoD problem itself for high dimensional control spaces.

To avoid this issue, the works in \cite{Darbon2021some,Darbon2020Overcoming,Darbon2023Neural,hu2024sdgd} focus on Hamilton-Jacobi equations where $\inf_{\kappa \in U} H$ is explicitly known.
The work \cite[section 3.4]{Nakamura2021Adaptive} considers specific optimal control problems such that $\inf_{\kappa \in U} H$ admits an analytic expression.
There are also research resorting to deep neural networks (DNNs).
For example, \cite[section 3.2]{Ji2022Solving} introduces a DNN to learn the feedback control $u(t, x)$
such that $u(t, x)$ becomes a stationary point of $H$, i.e., $\partial_{\kappa} H\vert_{\kappa = u(t, x)} = 0$,
whereas certain conditions on $U$ and $H$ are needed to ensure the stationary point is a minimizer of $\kappa \mapsto H$.
The paper \cite{Zhou2021Actor} considers static HJB-type PDEs,
where solving $\inf_{\kappa \in U} H$ is avoided by reformulating the problem into a SOCP solved by reinforcement learning.
\cb{In \cite{Li2024neural}, the authors propose a neural network approach for high-dimensional SOCPs, using the stochastic Pontryagin maximum principle to guide a forward SDE for state-space exploration for handling problems with complex dynamics. 
}
In addition, there are works on numerical methods for SOCPs by not explicitly solving the HJB equation, e.g., \cite{Hure2021Deep,Bachouch2022Deep,Jiequn2016Deep,Zhao2017High,Fu2020Highly,Gong2017efficient}.
At this time, developing new efficient numerical methods for  high-dimensional HJB equations still remains an actively researched topic.

In this paper, we propose a new approach for solving the high-dimensional HJB-type equation \eqref{eq_HJBPDE}.
In our approach, the control and value functions of the problem are approximated by DNNs.
The HJB equation is encoded into a Hamiltonian process and a cost process both depending on the control network and the value network.
The value and the control networks are trained by minimizing a functional of the Hamiltonian process while ensuring the cost process to be a martingale, which guarantees the value function is the solution to the HJB equation.
The martingale property is  enforced by an adversarial learning, whose loss function is constructed by characterizing the projection property of conditional expectations.
The proposed method, named SOC-MartNet, is able to solve high dimensional stochastic optimal control problems, by using the the approach of a martingale formulation, originally developed in the DeepMartNet for boundary value and eigenvalue problems of high dimensional PDEs \cite{cai2023deepmartnet,cai2023deepmartnet2}.
Our numerical experiments show that the proposed SOC-MartNet is highly effective and efficient for solving equations with dimension up to $2000$ in a small number of iterations of training.

In the SOC-MartNet method,
the task of finding $\inf_{\kappa \in U} H$ for each $(t, x)$ is accomplished by training an optimal control network to minimize a functional of the Hamiltonian process, and  avoiding  the need of evaluating explicitly the infimum in the HJB equation.
Moreover, the training algorithm enjoys parallel efficiency as it is free of time-directed iterations during gradient computation.
This feature is much different from existing deep-learning probabilistic methods for PDEs.
The SOC-MartNet also demonstrates broad applicability, effectively handling high-dimensional HJB equations and parabolic equations as well as SOCPs.

The remainder of this paper is organized as follows.
In section~\ref{sec_dpp}, we briefly review the main ideas in dynamic programming and minimum principle for solving SOCPs as well as the proposed computational approach.
In section~\ref{sec_amart}, we propose the SOC-MartNet and its algorithm for general non-degenerated HJB equations and parabolic equations.
Numerical results are presented in section~\ref{sec_tests}.
A conclusion and plan for future work are given in section~\ref{sec_conclu}.

\section{Dynamic programming and minimum principle, and computational approach}\label{sec_dpp}

We consider a filtered complete probability space $(\Omega, \mathcal{F}, \mathbb{F}^B, \mathbb{P})$ with $\mathbb{F}^B := (\mathcal F_t)_{0\leq t \leq T}$ as the natural filtration of the standard $q$-dimensional Brownian motion $B = (B_t)_{0\le t \leq T}$, and $T \in (0, \infty)$ a deterministic terminal time.
Let $\mathcal{U}_{\mathrm{ad}}$ be the set of admissible feedback control functions defined by
\begin{equation}\label{eq_defUad}
    \mathcal{U}_{\mathrm{ad}} := \bbr{u: [0, T] \times \R^d \to U \big\vert\; \text{$u$ is Borel measurable}} \sptext{with} U \subset \R^m.
\end{equation}
For any $u \in \mathcal{U}_{\mathrm{ad}}$, the controlled state process $X^u$ is governed by the following stochastic differential equation (SDE):
\begin{equation}\label{eq_stateq}
    X_t^u = x_0 + \int_0^t \bar{\mu}\br{s, X_s^u, u(s, X_s^u)} \di s + \int_0^t \bar{\sigma}\br{s, X_s^u, u(s, X_s^u)} \di B_s, \quad t \in [0, T]
\end{equation}
with $x_0 \in \R^d$,
where $\bar{\mu}: [0, T] \times \R^d \times U \to \R^d$ and $\bar{\sigma}: [0, T] \times \R^d \times U \to \R^d$ are the controlled drift coefficient and controlled diffusion coefficient, respectively, and the stochastic integral with respect to $B_s$ is of It\^o type.
The cost functional of $u$ is given by
\begin{equation}\label{eq_defcost}
    J(u) := \E{\int_0^T c\br{s, X_s^u, u(s, X_s^u)} \di s + g(X_T^u)},
\end{equation}
where $c: [0, T] \times \R^d \times U \to \R$ and $g: \R^d \to \R$ characterize the running cost and terminal cost, respectively.
Our main focus is the following SOCP:
\begin{equation}\label{eq_wscop}
    \text{Find}\;\; u^* \in \mathcal{U}_{\mathrm{ad}} \sptext{such that} J\br{u^*} = \inf_{u \in \mathcal{U}_{\mathrm{ad}}} J\br{u}.
\end{equation}

To carry out the approach of dynamic programming, we define the value function $v$ by $v(t, x) := \inf_{u \in \mathcal{U}_{\mathrm{ad}}} J\br{t, x, u}$ with
\begin{equation*}
    J\br{t, x, u} := \E{\int_t^T c\br{s, X_s^u, u(s, X_s^u)} \di s + g(X_T^u) \Big\vert X_t^u = x}
\end{equation*}
for $(t, x) \in [0, T] \times \R^d$.
Under certain conditions (see, e.g., \cite[Theorem 4.3.1 and Remark 4.3.4]{Pham2009Continuous}), the value function $v$ is the viscosity solution to the following fully nonlinear HJB equation
\begin{equation}\label{eq2_hjb}
    \partial_t v(t, x) + \inf_{\kappa \in U} {H}\br{t, x, \kappa, \partial_x v(t, x), \partial_{xx}^2 v(t, x)} = 0, \quad (t, x) \in [0, T) \times \R^d
\end{equation}
with the terminal condition $v(T, x) = g(x)$, $x \in \R^d$, and the Hamiltonian ${H}$ given by
\begin{equation}\label{eq_defbarH}
    {H}(t, x, \kappa, z, p) := \frac{1}{2} \operatorname{Tr}\br{p \bar{\sigma}\bar{\sigma}^{\top}(t, x, \kappa)}  + z^{\top} \bar{\mu}(t, x, \kappa) + c(t, x, \kappa)
\end{equation}
for $(t, x, \kappa, z, p) \in [0, T]  \times \R^d \times U \times \R^d \times \R^{d\times d}$.

Under the regularity condition $v \in C^{1, 2}$, i.e., $v$ is once and twice continuously differentiable with respect to $t \in [0, T]$ and $x \in \R^d$, respectively,
the classical verification theorem (see, e.g., \cite{carmona2016} or \cite[p. 268, Theorem 5.1]{Yong1999Stochastic}) reveals the optimal feedback control as
\begin{equation}\label{eq0_uXargmin0}
    u^*(t, X_t^*) \in \argmin_{\kappa \in U} {H}\br{t, X_t^*, \kappa, \partial_x v(t, X_t^*), \partial_{xx}^2 v(t, X_t^*)}, \quad t \in [0, T],
\end{equation}
where $X^* := X^{u^*}$ is the controlled diffusion defined in \eqref{eq_stateq} corresponding to the optimal control $u^*$. The above equation \eqref{eq0_uXargmin0} implies that for $t\in [0,T]$
\begin{equation}\label{eq0_uXargmin}
\begin{aligned}
    &H\br{t, X_t^*, u^\ast(t,X_t^\ast), \partial_x v(t, X_t^*), \partial_{xx}^2 v(t, X_t^*)} \\
    =\;& \inf_{\kappa \in U} H\br{t, X_t^*, \kappa, \partial_x v(t, X_t^*),  \partial_{xx}^2 v(t, X_t^*)}.
\end{aligned}
\end{equation}

Therefore, by the minimum principle \eqref{eq0_uXargmin}, to find the optimal feedback control, it is sufficient to ensure
\begin{equation}\label{eq_uXargmin0}
    u^*(t, x) \in \argmin_{\kappa \in U} {H}\br{t, x, \kappa, \partial_x v(t, x), \partial_{xx}^2 v(t, x)}, \quad x \in \Gamma_t, \quad t \in [0, T],
\end{equation}
or for all $t \in [0, T]$ and $x \in \Gamma_t$,
\begin{equation} \label{eq_uXargmin}
\begin{aligned}
    &H\br{t, x, u^\ast(t,x), \partial_x v(t, x), \partial_{xx}^2 v(t, x)}\\
    =\;& \inf_{\kappa \in U} {H}\br{t, x, \kappa, \partial_x v(t, x),  \partial_{xx}^2 v(t, x)}
\end{aligned}
\end{equation}
for some state set $\Gamma_t \supset \Gamma(X_t^*)$,
where $\Gamma(X_t^*)$ denotes the support set of the probability density function of $X_t^*$.

\bigskip

\noindent {\bf Computational approach.} On the basis of \eqref{eq_uXargmin}, the key step for solving the SOCP~\eqref{eq_wscop} is then to find the value function $v$ from the HJB equation~\eqref{eq2_hjb},  but,  the evaluation of the $\inf_{\kappa \in U} {H}$   in the the HJB equation  \eqref{eq_wscop}, thus finding $v(t,x)$, suffers from the CoD for a high dimensional control space. However, if the optimal control $u^\ast$ is known, then from the  minimum principle \eqref{eq_uXargmin}, the HJB equation is reduced to,  without the inf operation,
\begin{equation}\label{eq2_hjb1}
    \partial_t v(t, x) +  {H}\br{t, x, u^\ast(t,x), \partial_x v(t, x), \partial_{xx}^2 v(t, x)} = 0, \quad (t, x) \in [0, T) \times \Gamma_t.
\end{equation}
This argument suggests that we should consider a computational approach, which produces approximations simultaneously to both the value function $v$ and the optimal control $u^\ast$ represented by separate neural networks, to ensure that \eqref{eq2_hjb1} will hold once the training of the networks leads to convergence. Therefore, before the convergence of the network for the optimal control, the value network is in fact an approximation to a version of the HJB equation \eqref{eq2_hjb1} where the optimal control $u^\ast$ is replaced with its neural network approximation. 
Similarly, the infimum of the Hamiltonian in \eqref{eq_uXargmin} is taken with the value function replaced by its network approximation. It is expected that once both networks have converged, the correct forms of HJB equation and the infimum of Hamiltonian will have been enforced to yield both the optimal control and value function. This  will be the approach of the proposed numerical method.

\section{Proposed method}\label{sec_amart}

Throughout this section, we assume that the HJB-type equation~\eqref{eq_HJBPDE} 
admits a classical solution $v \in C^{1,2}$, 
\cb{where the regularity of $v$ is needed for the martingale formulation. 
According to standard PDE theory (see, e.g., \cite[Section 6.2, Theorem 5]{Krylov1987Nonlinear})}, 
this regularity follows from the non-degeneracy condition of \eqref{eq_HJBPDE}, i.e., $\sigma \sigma^{\top}(t, x)$ in \eqref{generator} is uniformly positive definite on $[0, T] \times \R^d$. 
Below, we show that \eqref{eq_HJBPDE} covers many useful cases, including:

\begin{itemize}
    \item {\bf SOCP~\eqref{eq_wscop} without volatility control.} If $\bar{\sigma}(t, x, \kappa) = \bar{\sigma}(t, x)$, the HJB equation \eqref{eq2_hjb} degenerates into a special case of \eqref{eq_HJBPDE} with
        \begin{equation*}
            \mathcal{L} = \frac{1}{2} \operatorname{Tr}\bbr{\bar{\sigma} \bar{\sigma}^{\top}\br{t, x} \partial_{xx}^2}, \quad H(t, x, \kappa,  z, p) = z^{\top} \bar{\mu}(t, x, \kappa) + c(t, x, \kappa)
        \end{equation*}
        for $(t, x, \kappa, z, p) \in [0, T]  \times \R^d  \times U \times \R^d \times \R^{d\times d}$.

    \item \cb{{\bf Controlled volatility $\bar{\sigma}$ with an uncontolled part.}
    If the $\bar{\sigma}$ admits a decomposition as $\bar{\sigma}(t, x, \kappa) = \bar{\sigma}_0(t, x) + \bar{\sigma}_1(t, x, \kappa)$
    for some $\R^{d\times q}$-valued functions $\bar{\sigma}_0$ and $\bar{\sigma}_1$, 
    then the HJB equation~\eqref{eq2_hjb} becomes a special case of \eqref{eq_HJBPDE} with
    \begin{equation*}
        \begin{aligned}
        &\mathcal{L} = \frac{1}{2} \operatorname{Tr}\bbr{\bar{\sigma}_0 \bar{\sigma}_0^{\top}\br{t, x} \partial_{xx}^2}, \\
        &H(t, x, \kappa, z, p) = \frac{1}{2} \operatorname{Tr}\rbr{p \Sigma(t, x, \kappa)}  + z^{\top} \bar{\mu}(t, x, \kappa) + c(t, x, \kappa),
        \\
        &\Sigma(t, x, \kappa) :=\bar{\sigma}_0(t, x) \bar{\sigma}_1^{\top}(t, x, \kappa) + \bar{\sigma}_1(t, x, \kappa) \bar{\sigma}_0^{\top}(t, x) +  \bar{\sigma}_1(t, x, \kappa) \bar{\sigma}_1^{\top}(t, x, \kappa)
        \end{aligned}
    \end{equation*}
    for $(t, x, \kappa, z, p) \in [0, T]  \times \R^d \times  U \times \R^d \times \R^{d\times d}$.}

    \item {\bf Knowledge of some preliminary approximation $u_0$ for the optimal control $u^*$.} In this case, the HJB equation~\eqref{eq2_hjb} can be rewritten into \eqref{eq_HJBPDE} with
          \begin{equation*}
              \mathcal{L} = \mathcal{L}^{u_0}, \quad H\br{t, x, \kappa,  \partial_x v(t, x), \partial_{xx}^2 v(t, x)} = \br{\mathcal{L}^{\kappa} - \mathcal{L}^{u_0}} v (t, x) + c(t, x, \kappa),
          \end{equation*}
          where
          \begin{equation*}
              \mathcal{L}^{\kappa} := \bar{\mu}^{\top}(t, x, \kappa) \partial_x + \frac{1}{2} \operatorname{Tr}\bbr{\bar{\sigma} \bar{\sigma}^{\top}\br{t, x, \kappa} \partial_{xx}^2} \sptext{for} \kappa \in U
          \end{equation*}
          with the convention $\mathcal{L}^{u} v(t, x) := \mathcal{L}^{u(t, x)} v(t, x)$.
    \item {\bf Explicit form of the optimal control in the Hamiltonian.}
        If the following function $\bar{H}$ is explicitly known:
        \begin{equation*}
            \bar{H}(t, x, z, p) := \inf_{u \in U} H\br{t, x, u,  z, p}, \quad (t, x, z, p) \in [0, T] \times \R^d  \times \R^{d\times d},
        \end{equation*}
        then \eqref{eq_HJBPDE} degenerates into a parabolic equation as
        \begin{equation}\label{eq_hjbbarH}
            \partial_t v(t, x) + \mathcal{L} v(t, x) + \bar{H}\br{t, x, \partial_x v(t, x), \partial_{xx}^2 v(t, x)}  = 0.
        \end{equation}
\end{itemize}
From the above discussions, the SOC-MartNet designed for the HJB-type equation \eqref{eq_HJBPDE} is applicable for the SOCP~\eqref{eq_wscop} with non-degenerated diffusion 
and the parabolic equation~\eqref{eq_hjbbarH} without boundary conditions.

\subsection{Martingale formulation for HJB-type equations}\label{sec_martformu}

Let $X: [0, T] \times \Omega \to \R^d$ be a uncontrolled diffusion process associated with the operator $\mathcal{L}$, i.e.,
\begin{equation}\label{eq_SDE}
    X_t = X_0 + \int_0^t \mu(s, X_s) \di s + \int_0^t \sigma(s, X_s) \di B_s, \quad t \in [0, T].
\end{equation}

\smallskip
\noindent{\bf Hamiltonian process $H_t^{u, v}$ and cost process $\M_t^{u, v}$ for $t \in [0, T]$.} Using the optimal control $u \in \mathcal{U}_{\mathrm{ad}}$ and the value function $v$, we define two processes -
a Hamiltonian process $H_t^{u, v}$, which will be used for enforcing the minimum principle \eqref{eq_uXargmin},
\begin{equation}
    H_t^{u, v} =H^{u, v}(t,X_t):= H\br{t, X_t, u(t, X_t),  \partial_x v(t, X_t), \partial_{xx}^2 v(t, X_t)}, \label{eq_defHt}
\end{equation}
and, a cost process $\M_t^{u, v}$, whose being a martingale implies that the value function $v(t, x)$ satisfies the HJB equation \eqref{eq2_hjb1},
\begin{equation}\label{eq_defMt}
    \M_t^{u, v}=\M^{u, v}(t,X_t) := v(t, X_t) + \int_0^t H_s^{u, v} \di s.
\end{equation}
Specifically, we aim at finding a set of sufficient conditions on $\M^{u, v}$ and $H^{u, v}$, under which $v$  satisfies the HJB-type equation~\eqref{eq_HJBPDE}, and $u$ is the optimal feedback control in sense of \eqref{eq_uXargmin}, as mentioned in the comment after \eqref{eq2_hjb1}.

Recalling \eqref{eq_uXargmin}, we assume that the uncontrolled diffusion $X_t$ can explore the whole support set of $X_t^*$ (see Remark~\ref{rmk_suppXt}), i.e.,
\begin{equation}\label{eq_inclusionCond}
   \Gamma_t= \Gamma(X_t)  \supset \Gamma(X_t^*), \quad t \in [0, T].
\end{equation}
Then, to establish \eqref{eq_uXargmin}, it is sufficient to consider the condition
\begin{equation}\label{eq_HtinfH}
    H_t^{u, v}= H^{u, v}(t,X_t) = \inf_{\kappa \in U} H\br{t, X_t, \kappa, \partial_x v(t, X_t), \partial_{xx}^2 v(t, X_t)}, \;\; t \in [0, T].
\end{equation}

\begin{remark} ({\bf State spaces of controlled and uncontrolled diffusion}) \label{rmk_suppXt}
    For the SOCP~\eqref{eq_wscop}, we introduce the following two strategies to meet the inclusion condition \eqref{eq_inclusionCond}:
    \begin{itemize}
        \item We randomly take the starting point $X_0$ of the uncontrolled diffusion $X$ such that the distribution of $X_0$ covers a neighborhood of $X_0^* = x_0$ in \eqref{eq_stateq}, e.g., $X_0 \sim N(x_0, r I_d)$ with $r > 0$ a hyper parameter.

        \item The uncontrolled diffusion can be taken as $X^{u_0}$ given by \eqref{eq_stateq} with $u_0$ as an initial approximation of $u^*$.
        Then, we turn to solve the following equation equivalent to \eqref{eq2_hjb}:
        \begin{equation*}
            \br{\partial_t + \mathcal{L}^{u_0}} v(t, x) + \inf_{\kappa \in U} \bbr{(\mathcal{L}^{\kappa} - \mathcal{L}^{u_0}) v(t, x) + c(t, x, \kappa)} = 0
        \end{equation*}
        for $(t, x) \in [0, T) \times \R^d$,
        for which the Hamiltonian process in \eqref{eq_defHt} is given by
        \begin{equation*}
            H_t^{u, v} = \br{\mathcal{L}^{u} - \mathcal{L}^{u_0}} v(t, X_t^{u_0}) + c\br{t, X_t^{u_0}, u(t, X_t^{u_0})}, \quad t \in [0, T].
        \end{equation*}
    \end{itemize}
\end{remark}

The condition \eqref{eq_HtinfH} will be enforced as part of the training for the value and control neural network, but, in order to avoid point-wise application of this condition in a time-space $(t,x)$ region, which is a source of CoD issue, we will reformulate the condition into an integral form in the following lemma, which can be sampled in the fashion of a Monte Carlo method for an empirical loss function.

\begin{lemma}\label{lemm_infH}
    Let $v$ be any Borel measurable function from $[0, T] \times \R^d $ to $\R$ such that
    \begin{equation}\label{eq_intEabsinfH}
        \int_0^T \E{\abs{\inf_{\kappa \in U} H\br{t, X_t, \kappa,  \partial_x v(t, X_t), \partial_{xx}^2 v(t, X_t)}}} \di t < \infty.
    \end{equation}
    Assume the optimal feedback control exists under $v$, i.e., the following equation~\eqref{eq_Htuv} admits a solution $u \in \mathcal{U}_{\mathrm{ad}}$:
    \begin{equation}\label{eq_Htuv}
        H_t^{u, v} = \inf_{\kappa \in U} H\br{t, X_t, \kappa,  \partial_x v(t, X_t), \partial_{xx}^2 v(t, X_t)}
    \end{equation}
    for $(t, \omega) \in [0, T] \times \Omega$, a.e.-$\di t \times \mathbb{P}$.
    Then, an optimal control $u$ for ~\eqref{eq_Htuv} can be found from the following minimization problem:
    finding a $u \in \mathcal{U}_{\mathrm{ad}}$ such that
    \begin{equation}\label{eq_Mt_ut2}
        \int_0^T \E{H_t^{u, v}}\di t = \inf_{\bar{u} \in \mathcal{U}_{\mathrm{ad}}} \int_0^T \E{H_t^{\bar{u}, v}}\di t.
    \end{equation}
\end{lemma}

\begin{proof}
    The existence of solution $u$ for the minimization problem \eqref{eq_Mt_ut2} is evident by taking $\int_0^T \E{\,\cdot\,} \di t$ on both sides of \eqref{eq_Htuv}.

   Next, recalling the definition in \eqref{eq_HtinfH}, we assume that there is a $\bar u \in \mathcal{U}_{\mathrm{ad}}$  as an optimal feedback control satisfying \eqref{eq_Htuv}, i.e.,
    \begin{equation}\label{eq_Hu1v}
        H_t^{\bar u, v} = \inf_{\kappa \in U} H\br{t, X_t, \kappa,  \partial_x v(t, X_t), \partial_{xx}^2 v(t, X_t)}
    \end{equation}
    for $(t, \omega) \in [0, T] \times \Omega$, a.e.-$\di t \times \mathbb{P}$.
    Then, if $u \in \mathcal{U}_{\mathrm{ad}}$ satisfying \eqref{eq_Mt_ut2}, then by its definition, we have that $\int_{0}^T \E{H_t^{u, v}} \di t \leq \int_{0}^T \E{H_t^{\bar u, v}} \di t$.
    Inserting \eqref{eq_Hu1v} into the right-side of this inequality, and further using \eqref{eq_intEabsinfH}, we obtain
    $$\int_{0}^T \E{\varepsilon_t} \di t \leq 0,$$ where
    \begin{equation*}
        \varepsilon_t = \varepsilon(t,X_t) := H_t^{u, v} - \inf_{\kappa \in U} H\br{t, X_t, \kappa,  \partial_x v(t, X_t), \partial_{xx}^2 v(t, X_t)}.
    \end{equation*}
    On the other hand, the definition of $\varepsilon(t,X_t)$ implies that $\varepsilon_t $ is a nonnegative function of $(t,X_t)$,  namely, $\varepsilon_t \geq 0$ for any $t \in [0, T]$.
    Thus it follows from \cite[Theorem 4.8 (i)]{Achim2020Probability} that
    \begin{equation*}
        \varepsilon_t = 0 \sptext{for} (t, \omega) \in [0, T] \times \Omega, \;\; \text{a.e.-}\di t \times \mathbb{P},
    \end{equation*}
    which implies an optimal control $u$ for \eqref{eq_Htuv}.
\end{proof}

\begin{remark}
    ({\bf Integral form of minimum principle and avoiding the CoD problem in $(t,x)$ variables}) \Cref{lemm_infH} allows us to replace the minimum principle in \eqref{eq_Htuv} in a point-wise $(t, X_t)$ sense by an integral version of the minimum principle defined in \eqref{eq_Mt_ut2}.
    The latter offers significant computational benefits, as the minimization is for a functional (defined by a double integral) on the admissible control set $\mathcal{U}_{\mathrm{ad}}$ only, instead of the Hamiltonian on the control space $U$ for each possible time-state $(t, X_t)$.
    More importantly, the high dimensional double integral in $(t, x)$ can be carried out in in the space-time domain in a Monte Carlo fashion, thus avoiding the CoD problem from the high dimension of $X_t$ and allowing independent sampling of $t$ and $x$ variables for an efficient parallel implementation.
\end{remark}

Utilizing the condition~\eqref{eq_Htuv}, we can simplify the HJB-type equation~\eqref{eq_HJBPDE} into
\begin{equation}\label{eq_ptLeqH}
    \br{\partial_t + \mathcal{L}} v(t, X_t) = - H_t^{u, v}(t,X_t)
\end{equation}
for $t \in [0, T]$.
Next, we will show that \eqref{eq_ptLeqH} can be fulfilled by enforcing the cost process $\M^{u, v}$ in \eqref{eq_defMt} to be a martingale with the value function $v$ and the optimal control $u$.
The following lemma presents the details.

\begin{lemma}\label{lemm_mart}
    For any $(u, v) \in \mathcal{U}_{\mathrm{ad}} \times C^{1, 2}$ satisfying
    \begin{equation}\label{eq_cond_mart_hjb}
        \int_0^T \E{\abs{\partial_x v \sigma \br{t, X_t}}^2} \di t < \infty, \;\;\; \int_0^T\E{\abs{H_t^{u, v}}^2} \di t < \infty, \;\;\; \E{\abs{v(T, X_T)}^2} < \infty,
    \end{equation}
    the equation \eqref{eq_ptLeqH} holds for $(t, \omega) \in [0, T] \times \Omega$ a.e.-$\di t \times \mathbb{P}$ if and only if $M_t^{u, v}(t,X_t)$ is a $\mathbb{F}^B$-martingale, i.e.,
    \begin{equation}\label{eq_Mt_ut}
        \M_t^{u, v} = \E{\M_T^{u, v} \vert \mathcal{F}_t}, \quad t \in [0, T].
    \end{equation}
\end{lemma}

\begin{proof}

    For $v \in C^{1, 2}$, the It\^o formula implies that for $t \in [0, T]$,
    \begin{equation}\label{eq_v_decom2}
        v(t, X_t) = v(0, X_0) + \int_0^t \br{\partial_t + \mathcal{L}} v(s, X_s) \di s + \int_0^t \partial_x v \sigma \br{s, X_s} \di B_s.
    \end{equation}

    $\eqref{eq_ptLeqH} \Rightarrow \eqref{eq_Mt_ut}$:
    Inserting \eqref{eq_ptLeqH} into the above equation and further using the definition of \eqref{eq_defMt}, we obtain
    \begin{equation*}
        \M_t^{u, v} = v(t, X_t) + \int_0^t H^{u, v}(s,X_s) \di s = v(0, X_0) + \int_0^t \partial_x v \sigma \br{s, X_s} \di B_s, \quad t \in [0, T],
    \end{equation*}
    Then $M_t^{u, v}(t,X_t)$ is a martingale from the above equation combined with the first condition in \eqref{eq_cond_mart_hjb}, thus $\eqref{eq_Mt_ut}$ holds.

    $\eqref{eq_Mt_ut} \Rightarrow \eqref{eq_ptLeqH}$:
    Recalling again the definition in \eqref{eq_defMt}, the last two conditions in \eqref{eq_cond_mart_hjb} implies that $\mathbb{E}[\abs{\M_T^{u, v}}^2] < \infty$.
    Then by the martingale representation theorem \cite[Theorem 1.2.9]{Pham2009Continuous}, there exists a $\mathbb{F}^B$-adapted process $Z: [0, T] \times \Omega \to \R^q$ such that
    \begin{equation*}
        \E{\int_0^T \abs{Z_s}^2 \di s} < +\infty, \quad \M_t^{u, v} = \M_0^{u, v} + \int_0^t Z_s \di B_s, \quad t \in [0, T].
    \end{equation*}
    By using the definition $ \M_t^{u, v}$ in \eqref{eq_defMt} above, we then have
    \begin{equation}\label{eq_v_decom1}
        v(t, X_t) = v(0, X_0) - \int_0^t H^{u, v}(s,X_s) \di s + \int_0^t Z_s \di B_s.
    \end{equation}
    Combining \eqref{eq_v_decom2} and \eqref{eq_v_decom1}, we have that
    \begin{equation*}
        Q_t := \int_0^t \bbr{\br{\partial_t + \mathcal{L}} v(s, X_s) + H^{u, v}(s,X_s)} \di s = \int_0^t \bbr{\partial_x v \sigma \br{s, X_s} - Z_s} \di B_s
    \end{equation*}
    for $t \in [0, T]$,
    which means that $Q$ is a $\mathbb{F}^B$-adapted process with finite variation,
    and is also a continuous martingale with $Q_0 = 0$.
    Thus it follows from \cite[Proposition 1.1.9]{Pham2009Continuous} that $Q_t = 0$ for any $t \in [0, T]$, a.s., which validates \eqref{eq_ptLeqH}.
\end{proof}

\Cref{lemm_infH,lemm_mart} directly lead to the following theorem,
which gives the martingale formulation~\eqref{eq_martformu} for the HJB-type equation~\eqref{eq_HJBPDE}.

\begin{theorem}\label{thm_martform}
    Assume $(u, v) \in \mathcal{U}_{\mathrm{ad}} \times C^{1, 2}$ satisfies \eqref{eq_intEabsinfH} and \eqref{eq_cond_mart_hjb}.
    Then, an optimal feedback control $u$ satisfying \eqref{eq_Htuv}, and the value function $v$ satisfying \eqref{eq_HJBPDE} for $t \in [0, T]$ and $x \in \Gamma(X_t)$, can be found by the following two conditions,
    \begin{equation}\label{eq_martformu}
        \int_0^T \E{H_t^{u, v}}\di t = \inf_{\bar{u} \in \mathcal{U}_{\mathrm{ad}}} \int_0^T \E{H_t^{\bar{u}, v}}\di t, \quad \M_t^{u, v} = \E{\M_T^{u, v} \vert \mathcal{F}_t}, \quad t \in [0, T],
    \end{equation} 
    where $M_t^{u, v}$ and $H_t^{u, v}$ are given by \eqref{eq_defHt} and \eqref{eq_defMt}, respectively.
\end{theorem}

\begin{remark}\label{rmk_ord1}
    \cb{
    A unique feature of the martingale formulation~\eqref{eq_martformu} is that it works with expectations and conditional expectations associated with the diffusion process $X$, rather than its pathwise properties. 
    These expectations can be approximated by applying the Euler-Maruyama scheme \cite[Proposition 5.11.1]{Kloeden1992Numerical} to \eqref{eq_SDE}, which achieves a first-order convergence rate in the weak sense (see the numerical results in \cref{sec_testcr}). 
    In contrast, related works such as \cite{weinan2017deep,han2018solving,raissi2018forwardbackward,hure2020deep,Zhou2021Actor,Zhang2022FBSDE,Ji2022Solving} leverage the pathwise properties of $X$. 
    Consequently, their method depend on the strong convergence rate of the Euler-Maruyama scheme, which is of order $1/2$, as analyzed in \cite{han2018convergence,hure2020deep,Germain2022Approximation} and demonstrated numerically in \cite{Zhang2022FBSDE}.
    }
\end{remark}

Now based on \cref{thm_martform},
the key issue is to fulfill the minimum condition and the martingale condition in \eqref{eq_martformu}, to be achieved by the SOC-MartNet algorithm.

\subsection{SOC-MartNet via adversarial learning for control/value functions}

Under some Lipschitz continuity and boundedness conditions on $\mu$ and $\sigma$, the uncontrolled diffusion $X$ given by \eqref{eq_SDE} is a Markov process relative to $\mathbb{F}^B$ \cite[Theorem 17.2.3]{Samuel2015Stochastic},
which, recalling \eqref{eq_defMt}, implies that
\begin{equation*}
    \E{\M_T^{u, v} \vert \mathcal{F}_t} =  \E{\M_T^{u, v} \vert X_t}, \quad t \in [0, T].
\end{equation*}
Thus the martingale condition in \eqref{eq_martformu} can be simplified into
\begin{equation}\label{eq_martMtXt}
    \M_t^{u, v} = \E{\M_T^{u, v} \vert X_t}, \quad t \in [0, T].
\end{equation}
To avoid computing the conditional expectation $\E{\,\cdot\,\vert X_t}$ as in the original DeepMartNet \cite{cai2023deepmartnet,cai2023deepmartnet2}, we modify the martingale condition in \eqref{eq_martMtXt} into
\begin{equation}\label{eq_suprho}
    \sup_{\rho \in \mathcal{T}} \abs{\int_0^{T-\Delta t} \E{\rho(t, X_t) \br{\M_{t+\Delta t}^{u, v} - \M_t^{u, v}}} \di t}^2 = 0,
\end{equation}
where $\mathcal{T}$ denotes the set of test functions, defined by
\begin{equation*}
    \mathcal{T} := \bbr{\rho: [0, T] \times \R^d \to \R \big\vert\; \text{$\rho$ is smooth and bounded}},
\end{equation*}
and $\Delta t \in (0, T)$ is the time step size.
\cb{Here $\rho$ in $\mathcal{T}$ is a general test function.} For sufficiently small $\Delta t$, the condition \eqref{eq_suprho} ensures the martingale condition in \eqref{eq_martMtXt}.
Actually, by the property of conditional expectations, it holds that
\begin{equation}\label{eq_rhoDelMt}
    \E{\rho(t, X_t)(\M_{t+\Delta t}^{u, v} - \M_t^{u, v})} = \mathbb{E}\Big[\rho(t, X_t) \E{(\M_{t+\Delta t}^{u, v} - \M_t^{u, v}) \vert X_t}\Big]
\end{equation}
for $t \in [0, T - \Delta t]$.
Inserting \eqref{eq_rhoDelMt} into \eqref{eq_suprho}, we have that
\begin{equation*}
    \int_0^{T-\Delta t} \mathbb{E}\Big[\rho(t, X_t) \E{(\M_{t+\Delta t}^{u, v} - \M_t^{u, v}) \vert X_t}\Big] \di t = 0 \sptext{for all} \rho \in \mathcal{T},
\end{equation*}
where $\E{(\M_{t+\Delta t}^{u, v} - \M_t^{u, v}) \vert X_t}$ is a deterministic and Borel measurable function of $(t, X_t)$ \cite[Corollary 1.97]{Achim2020Probability}, and thus,
\begin{equation}\label{eq_martform}
    \E{(\M_{t+\Delta t}^{u, v} -  \M_t^{u, v}) \vert X_t} = 0, \;\;\; (t, \omega) \in [0, T - \Delta t] \times \Omega, \;\;\; \text{a.e.-}\di t \times \mathbb{P}.
\end{equation}
The above conditions imply that $\M^{u, v}$ satisfies the martingale condition in \eqref{eq_martMtXt} approximately for sufficiently small $\Delta t$.

A unique feature of \eqref{eq_suprho} lies in its natural connection to adversarial learning \cite{Zang2020Weak}, based on which, we can fulfill the conditions \eqref{eq_term} and \eqref{eq_martformu} by
\begin{equation}\label{eq_conopt}
    (u, v) = \lim_{\lambda \rightarrow +\infty} \argmin_{\br{\bar{u}, \bar{v}} \in \mathcal{U}_{\mathrm{ad}} \times \mathcal{V}} \bbr{\sup_{\rho \in \mathcal{T} } \mathbb{L}(\bar{u}, \bar{v}, \rho, \lambda)}
\end{equation}
where \cb{$u \in \mathcal{U}_{\mathrm{ad}}$} given in \eqref{eq_defUad}, and \cb{$v \in \mathcal{V}$} the set of candidate value functions satisfying \eqref{eq_term}, i.e.,
\begin{equation}\label{eq_defV}
    \mathcal{V} := \bbr{v: [0, T] \times \R^d \to \R  \big\vert\; v \in C^{1, 2}, \; v(T, x) = g(x),\; \forall x \in \R^d},
\end{equation}
and $\mathbb{L}$ is the augmented Lagrangian defined by
\begin{equation}\label{eq_defLag}
    \mathbb{L}(u, v, \rho, \lambda) := \int_0^T \E{H_t^{u, v}}\di t + \lambda \abs{\int_0^{T-\Delta t} \E{\rho(t, X_t) \br{\M_{t+\Delta t}^{u, v} - \M_t^{u, v}}} \di t}^2
\end{equation}
with $\lambda$ the multiplier being sufficiently large.

For adversarial learning, we replace the functions $u$, $v$ and $\rho$ by the control network $u_{\alpha}: [0, T] \times \R^d \to U$, the value network $v_{\theta}: [0, T] \times \R^d \to \R$ and the adversarial network $\rho_{\eta}: [0, T] \times \R^d \to \R^{r}$ parameterized by $\alpha$, $\theta$ and $\eta$, respectively.
Since the range of $u_{\alpha}$ should be restricted in the control space $U$, if $U = [a, b] := \prod_{i=1}^m [a_i, b_i]$ with $a_i, b_i$ the $i$-th elements of $a, b \in \R^m$,
the structure of $u_{\alpha}$ can be
\begin{equation}\label{eq_defualp}
    u_{\alpha}(t, x) = a + \frac{b - a}{6}\mathrm{ReLU6}(\psi_{\alpha}(t, x)),  \quad (t, x) \in [0, T] \times \R^d,
\end{equation}
where $\mathrm{ReLU6}(y) := \min\{\max\{0, y\}, 6\}$ is an activation function and $\psi_{\alpha}: [0, T] \times \R^d \to \R^m$ is a neural network with parameter $\alpha$.
Remark~\ref{rmk_generalU} provides a penalty method to deal with general control spaces.
To fulfill the terminal condition in \eqref{eq_defV},
the value network $v_{\theta}$ takes the form of
\begin{equation*}
    v_{\theta}(T, x) := g(x), \quad v_{\theta}(t, x) := \phi_{\theta}(t, x) \sptext{for} t \in [0, T), \;\; x\in \R^d,
\end{equation*}
where $\phi_{\theta}: [0, T] \times \R^d \to \R$ a neural network parameterized by $\theta$.

The adversarial network $\rho_{\eta}$ plays the role of test functions.
By our experiment results, $\rho_{\eta}$ is not necessarily to be very deep, but instead, it can be a shallow network with enough output dimensionality.
A typical example is that
\begin{equation}\label{eq_rho_sin}
    \rho_{\eta}(t, x) = \sin\br{W_1 t + W_2 x + b} \in \R^r, \quad \eta := (W_1, W_2, b) \in \R^r \times \R^{r\times d} \times \R^r
\end{equation}
for $(t, x) \in [0, T] \times \R^d$, where $\sin(\cdot)$ is the activation function applied on $W_1 t + W_2 x + b$ in an element-wise manner.

\medskip
\noindent {\bf SOC-MartNet.} Based on \eqref{eq_conopt} and \eqref{eq_defLag} with $(u_{\alpha}, v_{\theta}, \rho_{\eta})$ in place of $(u, v, \rho)$, the solution $(u, v)$ of \eqref{eq_conopt} can be approximated by $(u_{\alpha^*}, v_{\theta^*})$ given by
\begin{equation}\label{eq_argmintheta}
    (\alpha^*, \theta^*) = \lim_{\lambda \to +\infty} \argmin_{\alpha, \theta} \bbr{\max_{\eta} L(\alpha, \theta, \eta, \lambda)},
\end{equation}
where
\begin{equation}\label{eq_defLmart}
    L(\alpha, \theta, \eta, \lambda) := \int_0^T \E{H_t^{u_{\alpha}, v_{\theta}}}\di t + \lambda \abs{\int_0^{T-\Delta t} \E{\rho_{\eta}(t, X_t) \br{\M_{t+\Delta t}^{u_{\alpha}, v_\theta} - \M_t^{u_{\alpha}, v_{\theta}}}} \di t}^2
\end{equation}
with $H^{u_{\alpha}, v_{\theta}}$ and $M^{u_{\alpha}, v_{\theta}}$ given in \eqref{eq_defHt} and \eqref{eq_defMt}, respectively.
The proposed method will be named SOC-MartNet for SOCPs as it is based on the martingale condition of the cost process \eqref{eq_Mt_ut}, similar to the DeepMartNet \cite{cai2023deepmartnet,cai2023deepmartnet2}.

\begin{remark}\label{rmk_generalU}
    If the control space $U$ is general rather than an interval,
    the network structure in \eqref{eq_defualp} is no longer applicable.
    This issue can be addressed by appending a new penalty term on the right side of \eqref{eq_defLmart} to ensure $u_{\alpha}(t, X_t)$ remains within $U$.
    The following new loss function is an example:
    \begin{equation*}
        \bar{L}(\alpha, \theta, \eta, \lambda, \bar{\lambda}) := L(\alpha, \theta, \eta, \lambda) + \bar{\lambda} \int_0^T \E{\mathrm{dist}\br{u_{\alpha}(t, X_t), U}} \di t,
    \end{equation*}
    where $L(\alpha, \theta, \eta, \lambda)$ is given in \eqref{eq_defLmart}; $\bar{\lambda} \geq 0$ is a multiplier and $\mathrm{dist}\br{\kappa, U}$ denotes a certain distance between $\kappa \in \R^m$ and $U$.
\end{remark}

\subsection{Training algorithm}\label{sec_train}

To solve \eqref{eq_argmintheta} numerically, we introduce a time partition on the time interval $[0, T]$, i.e.,
\begin{equation*}
    \pi_N := \bbr{t_0, t_1, \cdots, t_N} \sptext{s.t.} 0 = t_0 < t_1 < t_2 < \cdots < t_n < t_{n+1} < \cdots < t_N = T.
\end{equation*}
For $n = 0, 1, \cdots, N-1$, denote $\Delta t_n := t_{n+1} - t_n$ and $\Delta B_{n+1} := B_{n+1} - B_n.$
Then we apply the following numerical approximations on the loss function in \eqref{eq_defLmart}:
\begin{enumerate}
    \item The process $(X_{t_n})_{n=0}^N$ can be approximated by $(X_n)_{n=0}^N$, which is obtained by applying the Euler scheme to the SDE \eqref{eq_SDE}, i.e., for $n = 0, 1, 2, \cdots, N-1$,
    \begin{equation}\label{eq_EulerSDE}
        X_{n+1} = X_n + \mu(t_n, X_n) \Delta t_n + \sigma(t_n, X_n) \Delta B_{n+1}.
    \end{equation}

    \item The integral in \eqref{eq_defMt} can be approximated by quadrature rules, e.g., the trapezoid formula, which leads to $\M_{t_{n+1}}^{u_{\alpha}, v_{\theta}} - \M_{t_n}^{u_{\alpha}, v_{\theta}} \approx \Delta \M_{n+1}^{\alpha, \theta}$ with
    \begin{align*}
        \Delta \M_{n+1}^{\alpha, \theta} &:= v_{\theta}(t_{n+1}, X_{n+1}) - v_{\theta}(t_n, X_n) - \frac{1}{2} \br{H_n^{\alpha, \theta} + H_{n+1}^{\alpha, \theta}} \Delta t_n,\\
        H_n^{\alpha, \theta} &:= H\br{t_n, X_{n}, u_{\alpha}(t_n, X_{n}), \partial_x v_{\theta}(t_n, X_{n}), \partial_{xx}^2 v_{\theta}(t_n, X_{n})}.
    \end{align*}

    \item The expectations in \eqref{eq_defLmart} can be approximated by the Monte-Carlo method based on the i.i.d. samples of $\{(X_n, H_n^{\alpha, \theta}, \Delta \M_{n}^{\alpha, \theta})\}_{n=0}^N$, i.e.,
    \begin{equation}\label{eq_samp_Xn}
        \{(X_n^{(m)}, H_n^{\alpha, \theta, (m)}, \Delta \M_{n}^{\alpha, \theta, (m)})\}_{n=0}^N, \quad m=1, 2, \cdots, M.
    \end{equation}
\end{enumerate}
Combining the above approximations, the loss function in \eqref{eq_defLmart} is replaced by its mini-batch version for stochastic gradient descent or ascent,
\begin{align}
     & L(\alpha, \theta, \eta, \lambda; A) := \frac{1}{\abs{A}} \sum_{(n, m) \in A} H_n^{\alpha, \theta, (m)} \Delta t_n + \lambda \abs{G(\alpha, \theta, \eta, \lambda; A)}^2,\label{eq_defhatL} \\
     & G(\alpha, \theta, \eta; A) := \frac{1}{\abs{A}} \sum_{(n, m) \in A} \rho_{\eta}(t_n, X_{n}^{(m)}) \Delta \M_{n+1}^{\alpha, \theta, (m)} \Delta t_n \label{eq_defG}
\end{align}
with $\Delta t_N := \Delta \M_{N+1}^{\alpha, \theta} := 0$ for convenience,
where $A$ is a index subset randomly taken from $\bbr{0, 1, \cdots, N} \times \bbr{1, 2, \cdots, M}$ and is updated at each optimization step.
The loss function in \eqref{eq_defhatL} can be optimized by alternating gradient descent and ascent of $L(\alpha, \theta, \eta, \lambda; A)$ over $(\alpha, \theta)$ and $(\lambda, \eta)$, respectively.
The details are presented in \cref{alg_amnet}.

\begin{remark}\label{rmk_eff}
    In the martingale formulation, the diffusion process $X_t$ given by \eqref{eq_SDE} is fixed and independent of the control and the value function,
    and thus its sample paths can be generated offline before optimizing the loss function in \eqref{eq_defhatL}.
    Moreover, in the SOC-MartNet, the gradient computation for the loss function and the training of neural networks are both free of recursive iterations along the time direction, which contributes to significant efficiency gains for the SOC-MartNet.
    This feature is different from many existing deep-learning probabilistic methods for PDEs, e.g., \cite{weinan2017deep,Zhang2022FBSDE,Zhou2021Actor,hure2020deep,Nakamura2021Adaptive,Hure2021Deep,Bachouch2022Deep,Ji2022Solving}.
    Our numerical experiments in \cb{subsection~\ref{sec_dis_eff}} confirm the high efficiency of the SOC-MartNet.
\end{remark}

\begin{algorithm}[t]
    \caption{SOC-MartNet for solving the HJB-type equation \eqref{eq_HJBPDE}}\label{alg_amnet}
    \begin{algorithmic}[1]
        \Require $I$: the maximum number of iterations of  stochastic gradient algorithm;
        $M$: the total number of sample paths of diffusion process from \eqref{eq_EulerSDE};
        $\delta_{1}$/$\delta_{2}$/$\delta_{3}$/$\delta_{4}$: learning rates for control network $u_{\alpha}$/value network $v_{\theta}$/adversarial network $\rho_{\eta}$/multiplier $\lambda$;
        $\bar{\lambda}$: upper bound of multiplier $\lambda$;
        $J$/$K$: number of $(\alpha, \theta)$/$(\lambda, \eta)$ updates per iteration.
        \State Initialize the networks $u_{\alpha}$, $v_{\theta}$, $\rho_{\eta}$ and the multiplier $\lambda$
        \State Generate the sample paths $\{X_n^{(m)}\}_{n=0}^N$ for $m=1, 2, \cdots, M$ by \eqref{eq_EulerSDE}
        \For{$i = 0, 1, \cdots, I-1$}
        \State Sample the index subset $A_i \subset \bbr{0, 1, \cdots, N-1} \times \bbr{1, 2, \cdots, M}$ per \eqref{indexsample}
        \For{$j = 0, 1, \cdots ,J-1$}
        \State $\alpha \leftarrow \alpha - \delta_{1} \nabla_{\alpha} L(\alpha, \theta, \eta, \lambda; A_i)$ \;\; \Comment{$L$ is computed by \eqref{eq_defhatL}}
        \State $\theta \leftarrow \theta - \delta_{2} \nabla_{\theta} L(\alpha, \theta, \eta, \lambda; A_i)$
        \EndFor
        \For{$k = 0, 1, \cdots, K-1$}
        \State $\eta \leftarrow \eta + \delta_{3} \nabla_{\eta} L(\alpha, \theta, \eta, \lambda; A_i)$
        \State $\lambda \leftarrow \min \bbr{\bar{\lambda},\; \lambda + \delta_{4} |G(\alpha, \theta, \eta; A_i)|^2}$ \;\;\Comment{$G$ is computed by \eqref{eq_defG}}
        \EndFor
        \EndFor
        \Ensure $u_{\alpha}$ and $v_{\theta}$
    \end{algorithmic}
\end{algorithm}


\subsection{Application to parabolic problems}\label{sec_app_para}

The SOC-MartNet proposed in the last subsection is applicable for the HJB-type equation~\eqref{eq_HJBPDE}.
In this section, we explore how SOC-MartNet can be tailored to parabolic equations, yielding enhanced efficiency and simplicity.
\cb{
    Specifically, we consider the parabolic equation in the form of
\begin{equation}\label{eq_parab}
    \partial_t v(t, x) + \mathcal{L} v(t, x)
    + f\br{t, x, v(t, x), \partial_x v(t, x), \partial_{xx}^2 v(t, x)}  = 0
\end{equation}
with $f$ being a given function.
Here $f$ could depend also on $v(t, x)$, compared with $\bar{H}$ in \eqref{eq_hjbbarH}, and this dependence does not pose any difficulties in applying the method to \eqref{eq_parab}}. 

By \eqref{eq_defMt} with $f$ in place of $H$, we obtain a new cost process $\tilde{\M}^{v}$ independent of $u$, i.e.,
\begin{equation*}
    \tilde{\M}_t^{v} := v(t, X_t) + \int_0^t f(s, X_t, v(t, X_t), \partial_x v(t, X_t), \partial_{xx}^2 v(t, X_t)) \di s, \quad t \in [0, T].
\end{equation*}
Under some regularity conditions, by following the deductions in section~\ref{sec_martformu}, we conclude that
$v$ satisfies the equation~\eqref{eq_HJBPDE} for $t \in [0, T]$ and $x \in \Gamma(X_t)$
if $\tilde{\M}^{v}$ is a $\mathbb{F}^B$-martingale.
Thus the value function $v_{\theta^*}$ can be learned through adversarial training to enforce the martingale property of $\tilde{\M}^{v}$, i.e.,
\begin{equation}\label{eq_defbarG}
    \theta^*  = \argmin_{\theta} \bbr{\max_{\eta}  \abs{\tilde{G}(\theta, \eta)}}
\end{equation}
with
\begin{equation*}
    \tilde{G}(\theta, \eta) := \int_0^{T-\Delta t} \E{\rho_{\eta}(t, X_t) \br{\tilde{\M}_{t+\Delta t}^{v_\theta} - \tilde{\M}_t^{v_{\theta}}}} \di t.
\end{equation*}

To learn the value function from \eqref{eq_defbarG}, at each iteration step, the loss function $\tilde{G}(\theta, \eta)$ is replaced by its mini-batch version defined as
\begin{equation}\label{eq_barG}
    \tilde{G}(\theta, \eta; A) := \frac{1}{\abs{A}} \sum_{(n, m) \in A} \rho_{\eta}(t_n, X_{n}^{(m)}) \Delta \tilde{\M}_{n+1}^{\theta, (m)} \Delta t_n,
\end{equation}
where $A$ is a index subset randomly taken from $\bbr{0, 1, \cdots, N} \times \bbr{1, 2, \cdots, M}$ and
\begin{align*}
    \Delta \tilde{\M}_{n+1}^{\theta, (m)} & := v_{\theta}(t_{n+1}, X_{n+1}^{(m)}) - v_{\theta}(t_n, X_n^{(m)}) - \frac{1}{2} \br{f_n^{\theta, (m)} + f_{n+1}^{\theta, (m)}} \Delta t_n,     \\
    f_n^{\theta, (m)}                     & := f\br{t_n, X_{n}^{(m)}, v_{\theta}(t_n, X_{n}^{(m)}), \partial_x v_{\theta}(t_n, X_{n}^{(m)}), \partial_{xx}^2 v_{\theta}(t_n, X_{n}^{(m)})},
\end{align*}
and $X_{n}^{(m)}$ is introduced in \eqref{eq_samp_Xn}.
\Cref{alg_amnet2} presents the detailed procedures of the SOC-MartNet for parabolic equations.

\begin{algorithm}[t]
    \caption{SOC-MartNet for solving the parabolic equation~\eqref{eq_parab}}\label{alg_amnet2}
    \begin{algorithmic}[1]
        \Require $I$: the maximum number of iterations of stochastic gradient algorithm;
        $M$: the total number of sample paths of diffusion process from \eqref{eq_EulerSDE};
        $\delta_{2}$/$\delta_{3}$: learning rates for value network $v_{\theta}$/adversarial network $\rho_{\eta}$;
        $J$/$K$: number of $\theta$/$\eta$ updates per iteration.
        \State Initialize the networks $v_{\theta}$ and $\rho_{\eta}$
        \State Generate the sample paths $\{X_n^{(m)}\}_{n=0}^N$ for $m=1, 2, \cdots, M$ by \eqref{eq_EulerSDE}
        \For{$i = 0, 1, \cdots, I-1$}
        \State Sample the index subset $A_i \subset \bbr{0, 1, \cdots, N-1} \times \bbr{1, 2, \cdots, M}$ per \eqref{indexsample}
        \For{$j = 0, 1, \cdots ,J-1$}
        \State $\theta \leftarrow \theta - \delta_{2} \nabla_{\theta} \abs{\tilde{G}(\theta, \eta; A_i)}^2$ \;\; \Comment{$\tilde{G}$ is computed by \eqref{eq_barG}}
        \EndFor
        \For{$k = 0, 1, \cdots, K-1$}
        \State $\eta \leftarrow \eta + \delta_{3} \abs{\nabla_{\eta} \tilde{G}(\theta, \eta; A_i)}^2$
        \EndFor
        \EndFor
        \Ensure $v_{\theta}$
    \end{algorithmic}
\end{algorithm}

\section{Numerical tests}\label{sec_tests}

In this section, we carry out numerical tests to show the behaviors of the SOC-MartNet in solving high-dimensional parabolic equations and HJB equations.

\medskip
\noindent {\bf Sampling of index set $A_i$.} For the SOC-MartNet given by \cref{alg_amnet,alg_amnet2}, the index subset $A_i$ on Line 4 is taken as
\begin{equation}\label{indexsample}
    A_i = \bbr{0, 1, 2, \cdots, N} \times M_i,
\end{equation}
where $M_i$ is a random subset of the index set $\bbr{1, 2, \cdots, M}$ and obtained as follows:
\begin{enumerate}
    \item At the beginning of each epoch, a random permutation of $\bbr{1, 2, \cdots, M}$ is created and equally split into some index subsets $P_1, P_2, \cdots, P_m$;
    \item During the epoch, $P_1, P_2, \cdots, P_m$ are used as $M_i$ one-by-one for each iteration;
    \item Once all the index subsets have been used, namely the epoch is completed, the process restarts at step 1 for the next epoch.
\end{enumerate}
At the $i$-th iteration step, we take the batch size $\abs{M_i}$ of sample paths as \cb{$\abs{M_i} = 256$ and $128$ for $d \leq 1000$ and $d > 1000$}, respectively.

\cb{The numerical method solves $v(0, x)$ for $x \in D_0$, where $D_0$ is a set of finite spatial points in $\R^d$ and will be specified in each benchmark problem. }
The start points $X_0^{(m)}$ of the sample paths in \eqref{eq_EulerSDE} are randomly taken from the set $D_0$ \cb{with replacement}, where the total number of sample paths is set to $M = 10^5$.

If no otherwise specified, we take $T = 1$ and $N = 100$ for all involved loss functions, and all the loss functions are minimized by the RMSProp algorithm. 
The learning rates on Lines 6, 7 and 10 of \cref{alg_amnet} are set to
\begin{equation*}
    \delta_{1} = \delta_{2} = \delta_0 \times 10^{-3} \times 0.01^{i/I}, \quad \delta_{3} = 10^{-2} \times 0.01^{i/I}, \quad i = 0, 1, \cdots, I-1
\end{equation*}
with $\delta_0 := 3d^{-0.5}$ for $d \leq 1000$ and $\delta_0 := 3 d^{-0.8}$ for $d > 1000$.
If not otherwise specified,
the initial value of $\lambda$ is $10$ with its learning rate $\delta_4 = 10$ and its upper bound $\bar{\lambda}$ set to $10^3$.
The inner iteration steps are $J = 2K = 2$.
The neural networks $u_{\alpha}$ and $v_{\theta}$ both consist of $6$ hidden layers with $d + 10$ ReLU units in each hidden layer.
The adversarial network $\rho_{\eta}$ is given by \eqref{eq_rho_sin} with the output dimensionality $r = 600$.

All the tests are implemented by Python 3.12 and PyTorch 2.51.
If no otherwise specified, the algorithm is accelerated by the strategy of Distributed Data Parallel (DDP) \footnote{\url{https://github.com/pytorch/tutorials/blob/main/intermediate_source/ddp_tutorial.rst}} on a compute node equipped with 8 GPUs \cb{(NVIDIA A100-SXM4-80GB)}.
When reporting the numerical results,  ``SD'', ``vs'' and ``Iter.''   are short for ``standard deviation'', ``versus'' and ``Iter.'', respectively.
The term ``Mart. Loss'' denotes the value of $|G(\alpha, \theta, \eta; A_i)|^2$ defined in \eqref{eq_defG}, and ``Hamilt.'' denotes the empirical Hamiltonian defined by $\text{Hamilt.} := \abs{A_i}^{-1} \sum_{(n, m) \in A_i} H_n^{\alpha, \theta, (m)} \Delta t_n$ recalling \eqref{eq_defhatL}. 
\cb{
The relative $L^1$-error and $L^{\infty}$-error are respectively defined by
\begin{align*}
    \text{Rel.}\; L^1\text{-error} &:= \vbr{\Big}{\sum_{x \in D_0} \abs{\widehat{v}\br{0, x} - v(0, x)}} \big/ \vbr{\Big}{\sum_{x \in D_0} \abs{v(0, x)}}, \\
    \text{Rel.}\; L^{\infty}\text{-error} &:= \vbr{\Big}{\max_{x \in D_0} \abs{\widehat{v}\br{0, x} - v(0, x)}} \big/ \vbr{\Big}{\max_{x \in D_0} \abs{v(0, x)}}.
\end{align*}
}

\subsection{Linear parabolic problem}

We consider the following problem:
\begin{equation}\label{eq_simp_proble}
    \left\{\begin{aligned}
         & (\partial_t + \frac{1}{2}\Delta_x) v(t, x) - f(t, x) = 0, \quad (t, x) \in [0, T) \times \R^d, \\
         & v(T, x) = g(x), \quad x \in \R^d,
    \end{aligned}\right.
\end{equation}
where $f$ and $g$ are chosen such that $v$ is given by
\begin{equation}\label{eq_linearsol}
    v(t, x) = 1 +  \frac{1}{d} \sum_{i=1}^d \sin(t + x_i), \quad (t, x) \in [0, T] \times \R^d.
\end{equation}
\cb{
We apply the SOC-MartNet (\cref{alg_amnet2}) to solve $v(0, x)$ for $x \in D_0$, where $D_0$ consits of $M$ uniformly spaced gird points on the two spatial line segments $S_1 \cup S_2$ defined by}
\cb{
\begin{align}
    &S_1 := \bbr{s \B{e}_1: s \in [-1, 1]}, \quad \B{e}_1 := (1, 0, 0, \cdots, 0)^{\top} \in \R^{d}, \label{eq_defSi}\\
    &S_2 := \bbr{s \B{1}_d: s \in [-1, 1]}, \quad \B{1}_d := (1, 1, \cdots, 1)^{\top} \in \R^{d}.\label{eq_defS2}
\end{align}
}
The relevant numerical results for $d = 100, 1000, 2000$ are presented in \cref{fig_linear_prab}, which demonstrate that the SOC-MartNet is effective for problems with dimensionality upto $2000$.

\begin{figure}[htbp]
    \centering
    \subfloat[$s \mapsto v(0, s \B{e}_1)$, $d=100$]{\includegraphics[width=0.25\textwidth]{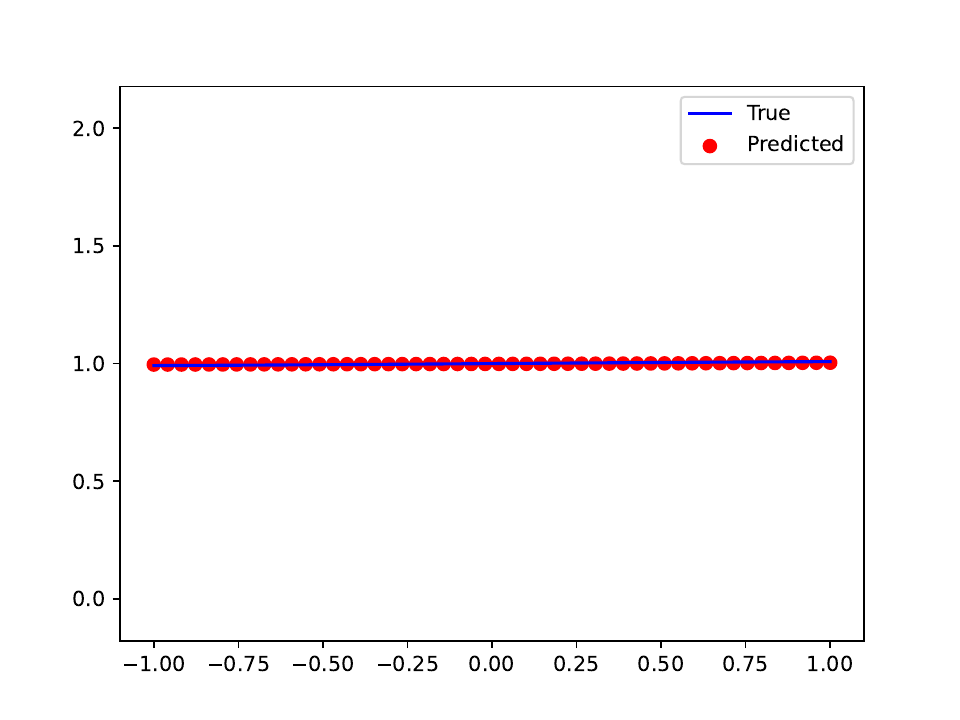}}
    \subfloat[$s \mapsto v(0, s \B{1}_d)$, $d=100$]{\includegraphics[width=0.25\textwidth]{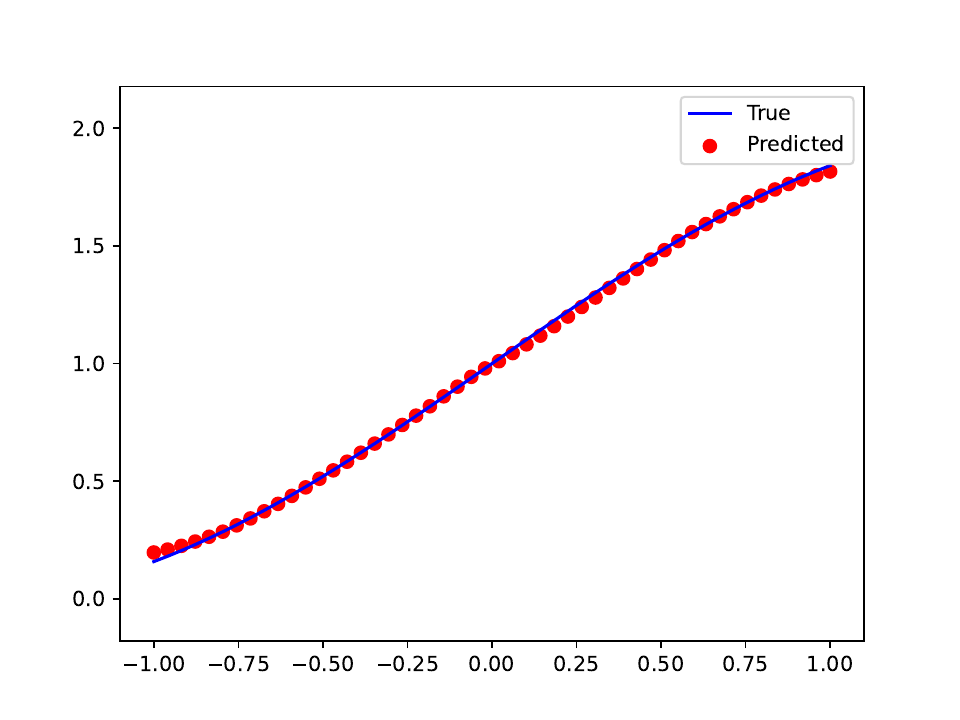}}
    \subfloat[Mart. Loss vs Iter., $d=100$]{\includegraphics[width=0.23\textwidth]{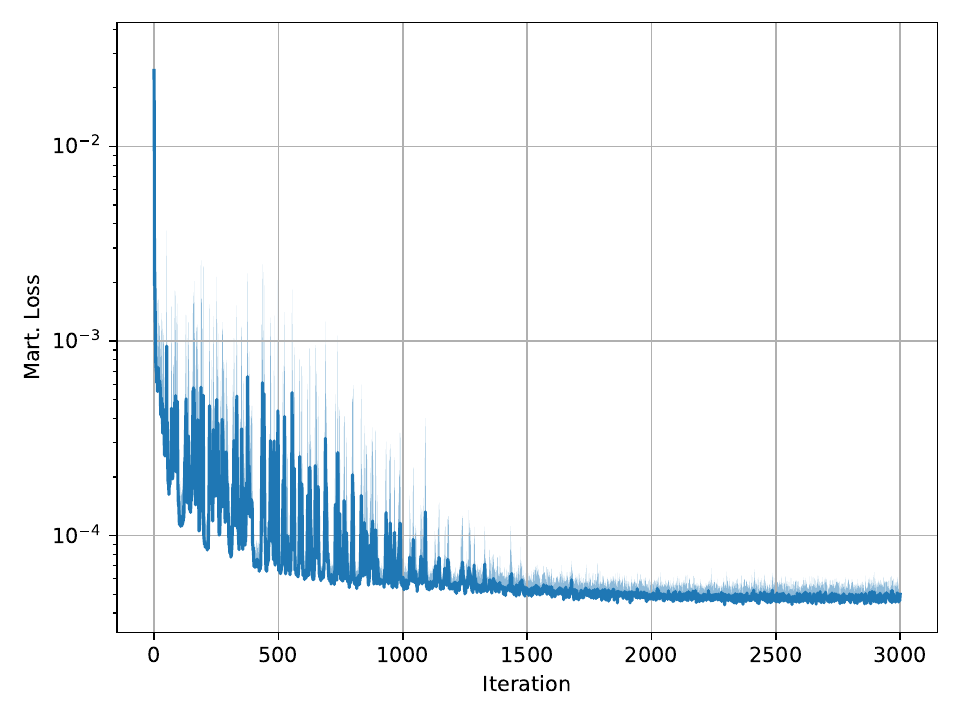}}
    \subfloat[RE vs Iter., $d=100$]{\includegraphics[width=0.23\textwidth]{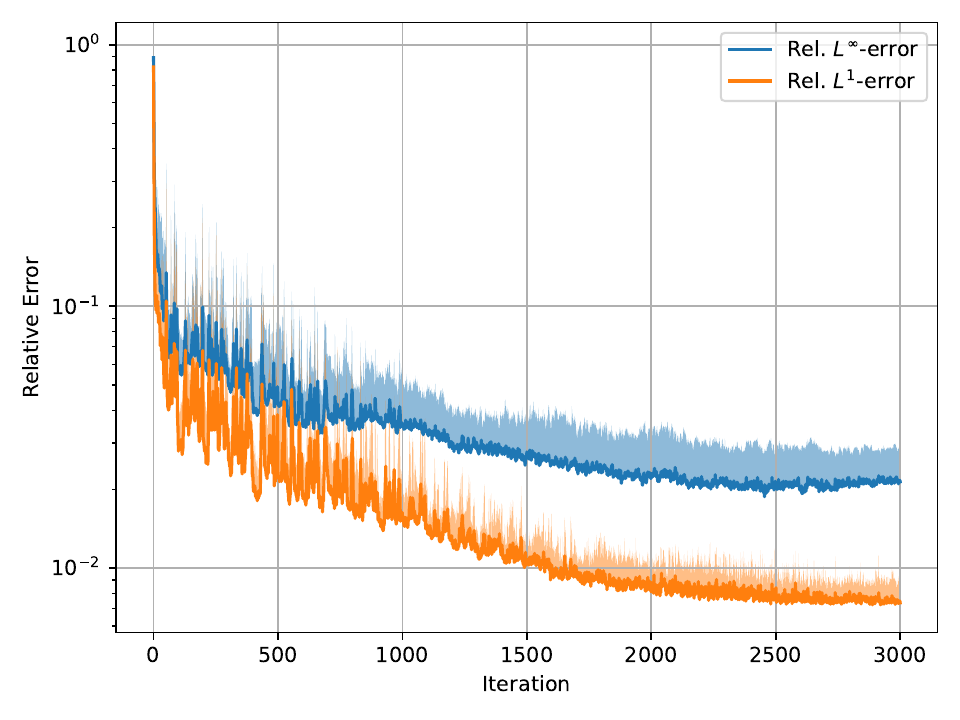}}
    \\
    \subfloat[$s \mapsto v(0, s \B{e}_1)$, $d=1000$]{\includegraphics[width=0.25\textwidth]{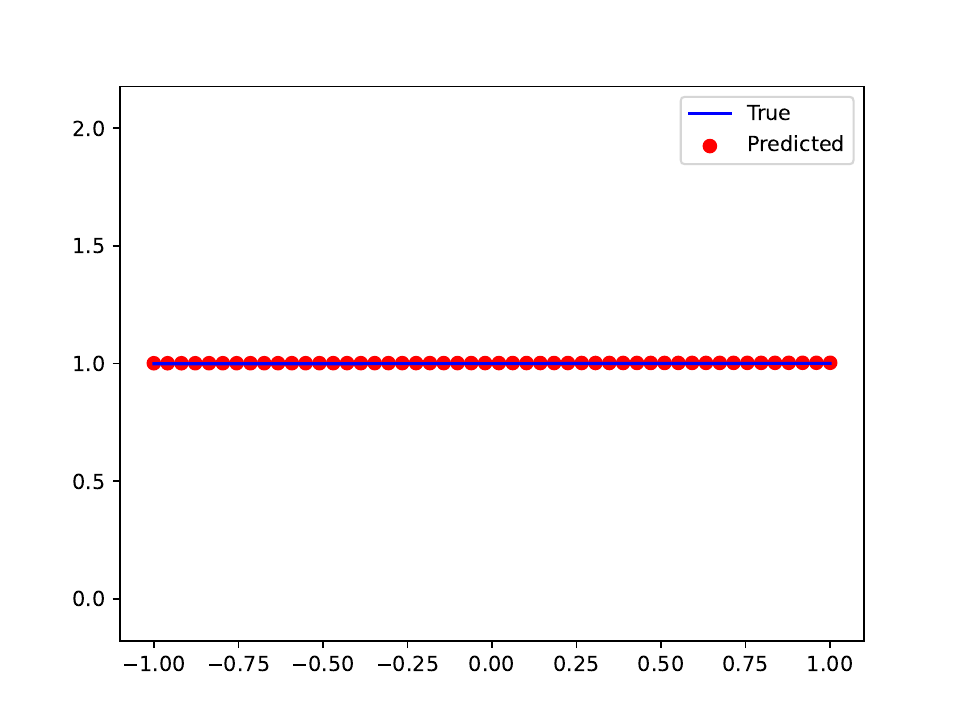}}
    \subfloat[$s \mapsto v(0, s \B{1}_d)$, $d=1000$]{\includegraphics[width=0.25\textwidth]{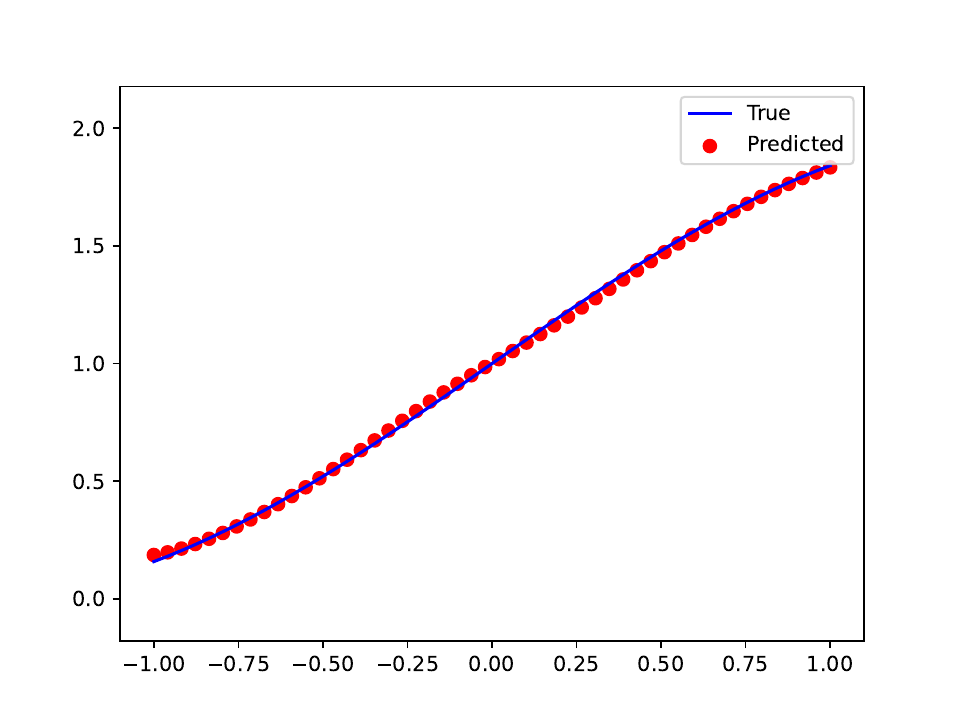}}
    \subfloat[Mart. Loss vs Iter., $d=1000$]{\includegraphics[width=0.23\textwidth]{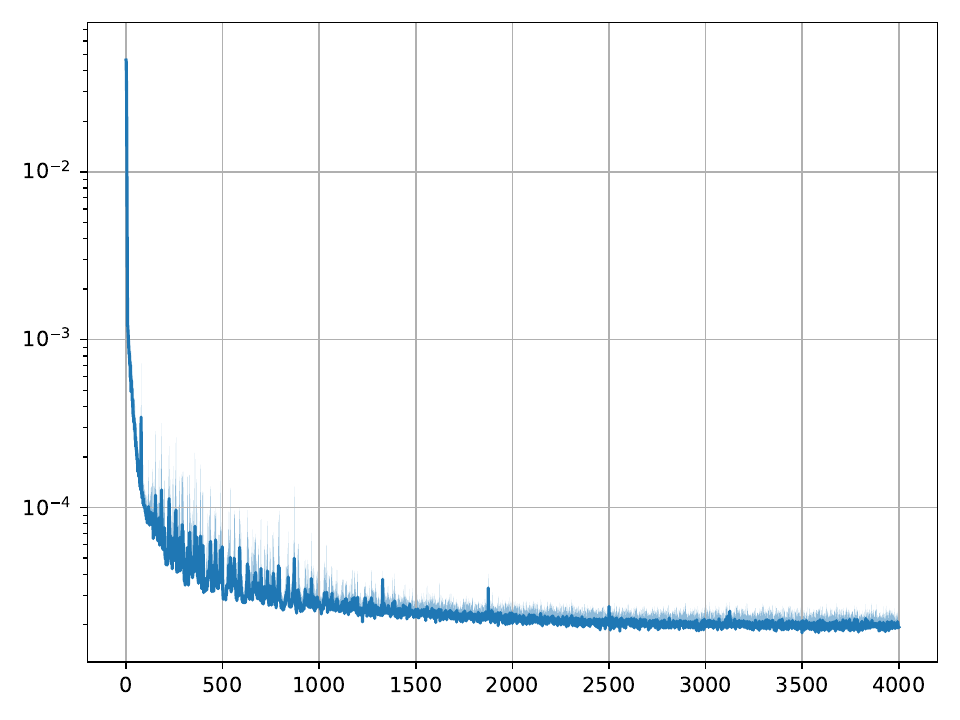}}
    \subfloat[RE vs Iter., $d=1000$]{\includegraphics[width=0.23\textwidth]{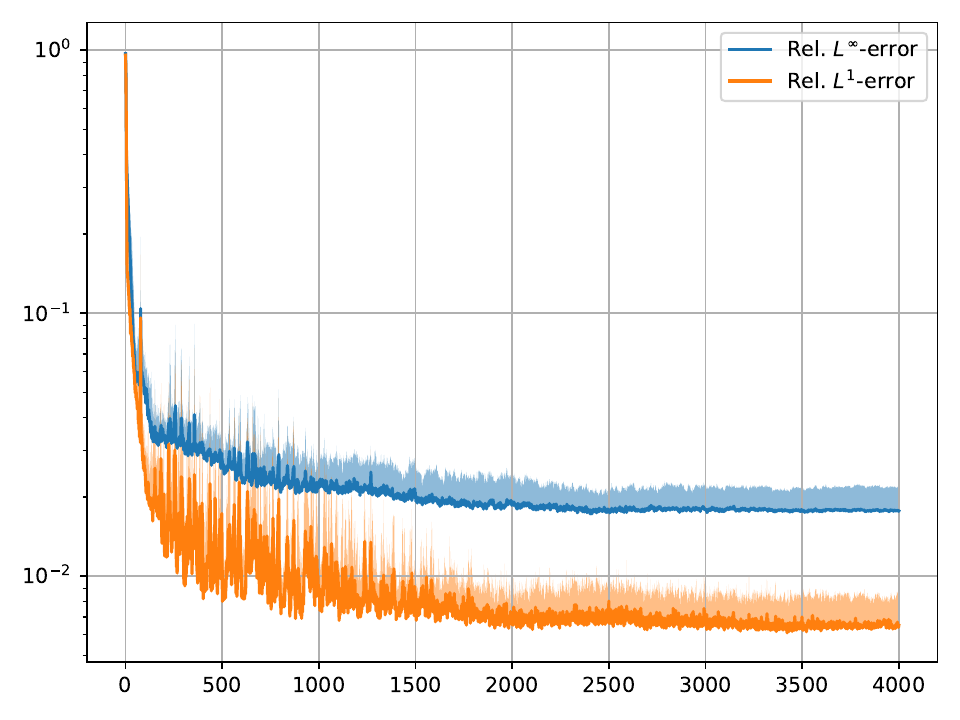}}
    \\
    \subfloat[$s \mapsto v(0, s \B{e}_1)$, $d=2000$]{\includegraphics[width=0.25\textwidth]{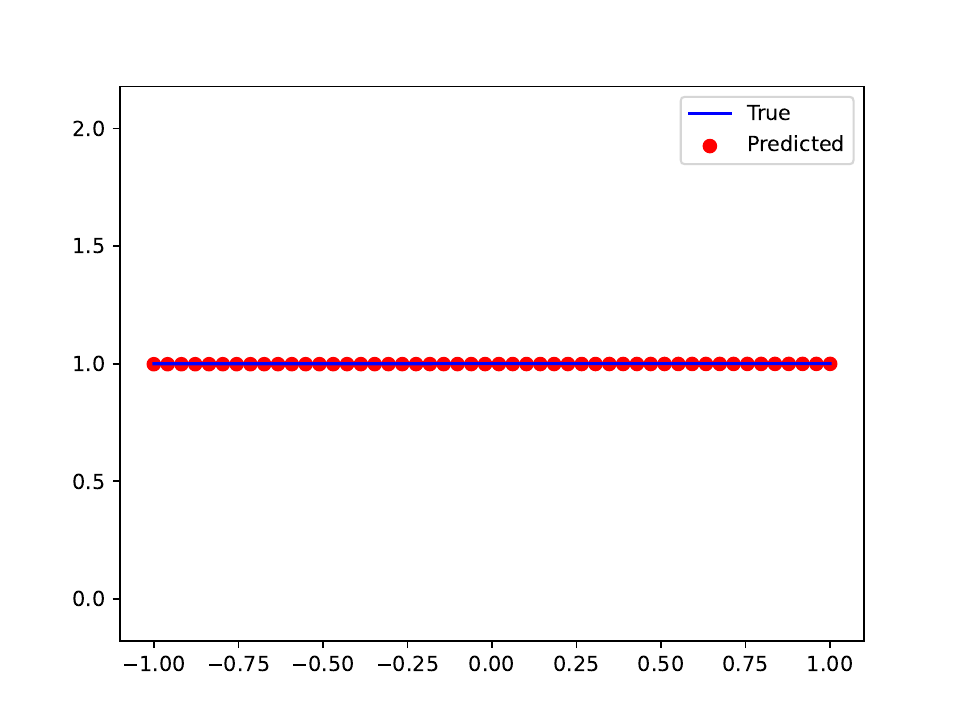}}
    \subfloat[$s \mapsto v(0, s \B{1}_d)$, $d=2000$]{\includegraphics[width=0.25\textwidth]{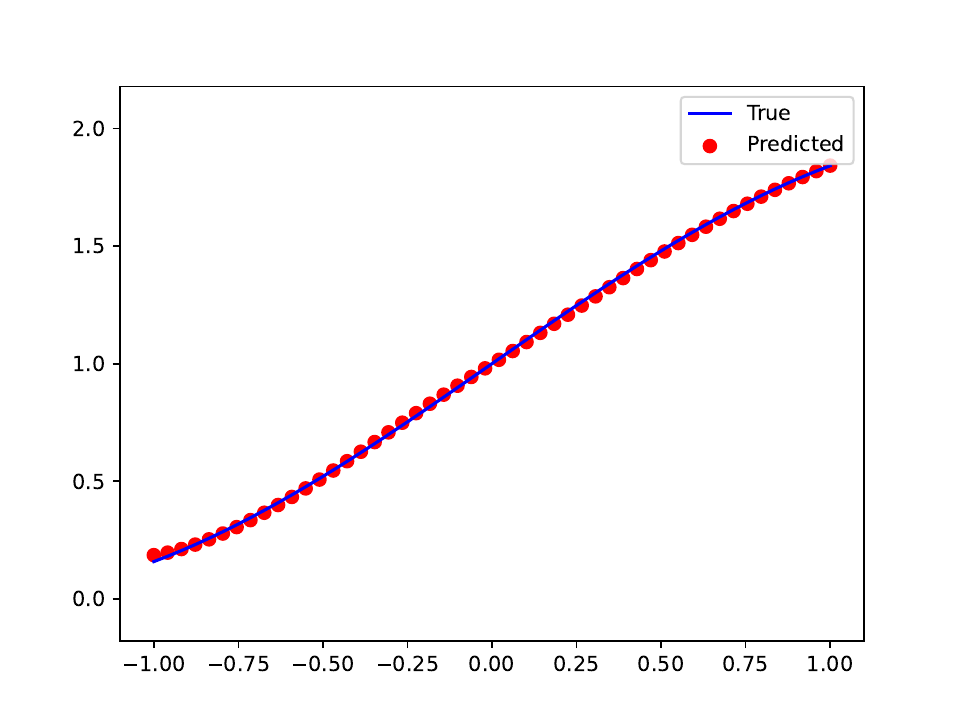}}
    \subfloat[Mart. Loss vs Iter., $d=2000$]{\includegraphics[width=0.23\textwidth]{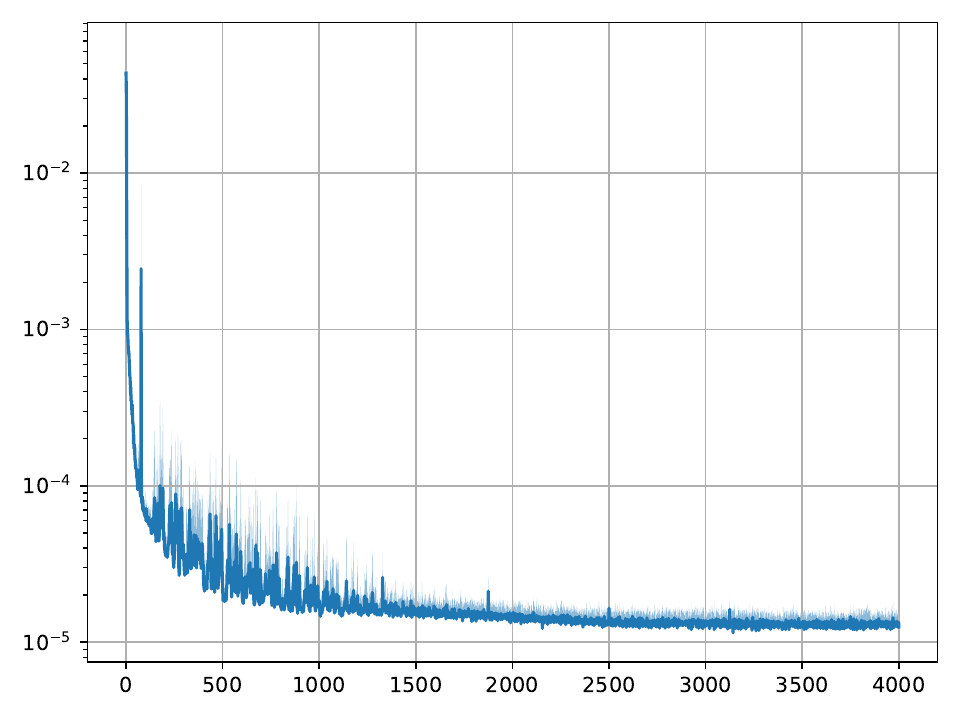}}
    \subfloat[RE vs Iter., $d=2000$]{\includegraphics[width=0.23\textwidth]{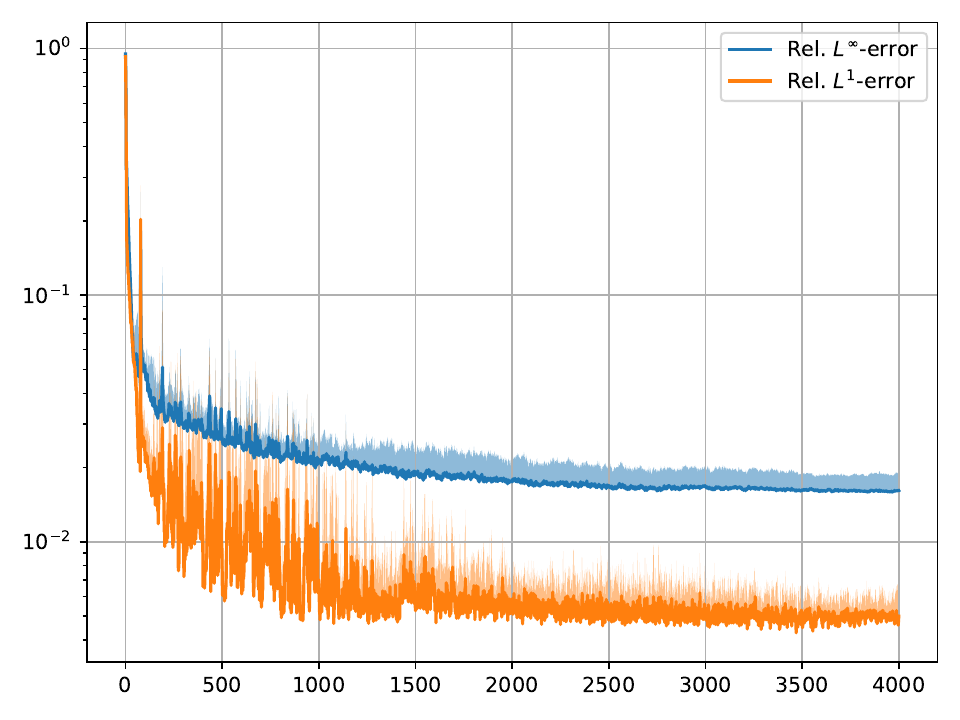}}
    \caption{Numerical results of SOC-MartNet (\Cref{alg_amnet2}) for the linear parabolic problem~\eqref{eq_simp_proble}.
    The shaded region represents the mean + $2 \times$ SD of the loss values and relative errors across 5 independent runs.
    The running times are 37, 112 and 363 seconds for $d=100$, $1000$, and $2000$, respectively.
    }\label{fig_linear_prab}
\end{figure}

\subsection{Semilinear parabolic equation}

We consider the semilinear parabolic equation from \cite[Section 4.3]{weinan2017deep}:
\begin{equation}\label{eq_semiparab}
    \left\{\begin{aligned}
         & \br{\partial_t + \Delta_x} v(t, x) - \abs{\partial_x v(t, x)}^2 = 0, \quad (t, x) \in [0, T) \times \R^d, \\
         & v(T, x) = 1 + g(x), \quad x \in \R^d
    \end{aligned}\right.
\end{equation}
for some given terminal function $g: \R^d \to \R$.
The problem \eqref{eq_semiparab} admits an analytic solution as
\begin{equation}\label{eq_vtx_hjb1}
    v(t, x) = 1 - \ln\left(\mathbb{E}\left[\exp \left(-g(x+\sqrt{2} B_{T-t})\right)\right]\right), \quad (t, x) \in [0, T] \times \R^d.
\end{equation}
To compute the absolute error of numerical solutions, the analytic solution in \eqref{eq_vtx_hjb1} is approximated by the Monte-Carlo method applied on the expectation using $10^6$ i.i.d. samples of $B_{T-t}$.

For this example, we consider an oscillatory terminal function as
\begin{equation}\label{eq_oscgx}
    g(x):= \frac{1}{d} \sum_{i=1}^d \bbr{\sin(x_i - \frac{\pi}{2}) + \sin\br{\br{\epsilon_0 + x_i^2}^{-1}}}, \quad x \in \R^d, \quad \epsilon_0 = \pi/10.
\end{equation}
Under \eqref{eq_oscgx}, the true solution along the diagonal of the unite cube, i.e., $s \mapsto v(t, s \B{1}_d)$, is independent of the spatial dimensionality $d$, whose graphs at $t = 0$ and $T$ are presented in Figures~\ref{fig1_hjb2_turet0te} (a) - (d).
As the graphs show, although $s \mapsto v(T, s \B{1}_d)$ is oscillatory around $s = 0$, the curve of $s \mapsto v(0, s \B{1}_d)$ is relatively smooth and depends on the terminal time $T$.

We apply the SOC-MartNet (\cref{alg_amnet2}) to solve $v(0, x)$ for $x \in D_0$ with $D_0$ the same as the last example. 
The relevant numerical results for $d=100$ are presented in Figures~\ref{fig1_hjb2_turet0te} (e) - (h), which show that the SOC-MartNet (\cref{alg_amnet2}) remains unaffected by the oscillations in $g(x)$, and captures the analytical solution well.

\begin{figure}[t]
    \centering
    \subfloat[$T=1$]{\includegraphics[width=0.24\textwidth]{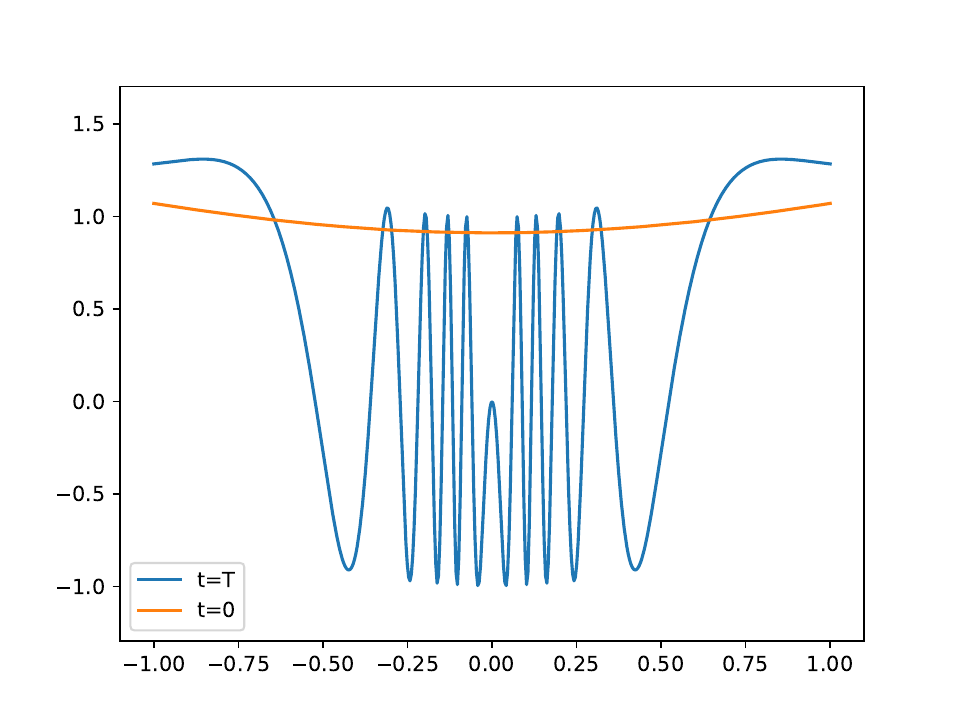}}
    \subfloat[$T=0.1$]{\includegraphics[width=0.24\textwidth]{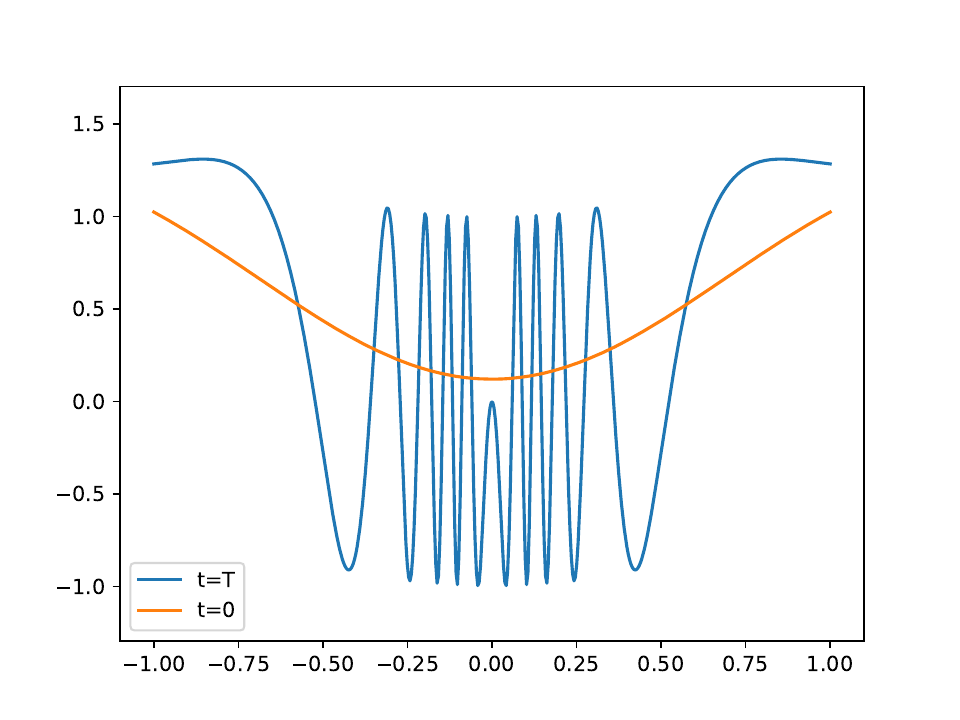}}
    \subfloat[$T=0.05$]{\includegraphics[width=0.24\textwidth]{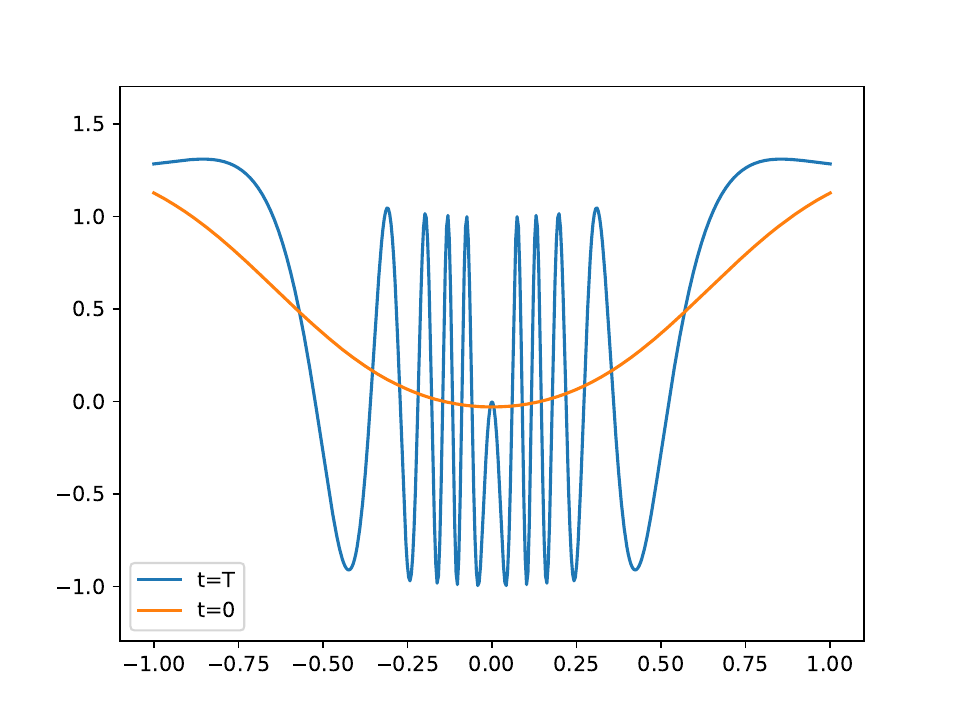}}
    \subfloat[$T=0.01$]{\includegraphics[width=0.24\textwidth]{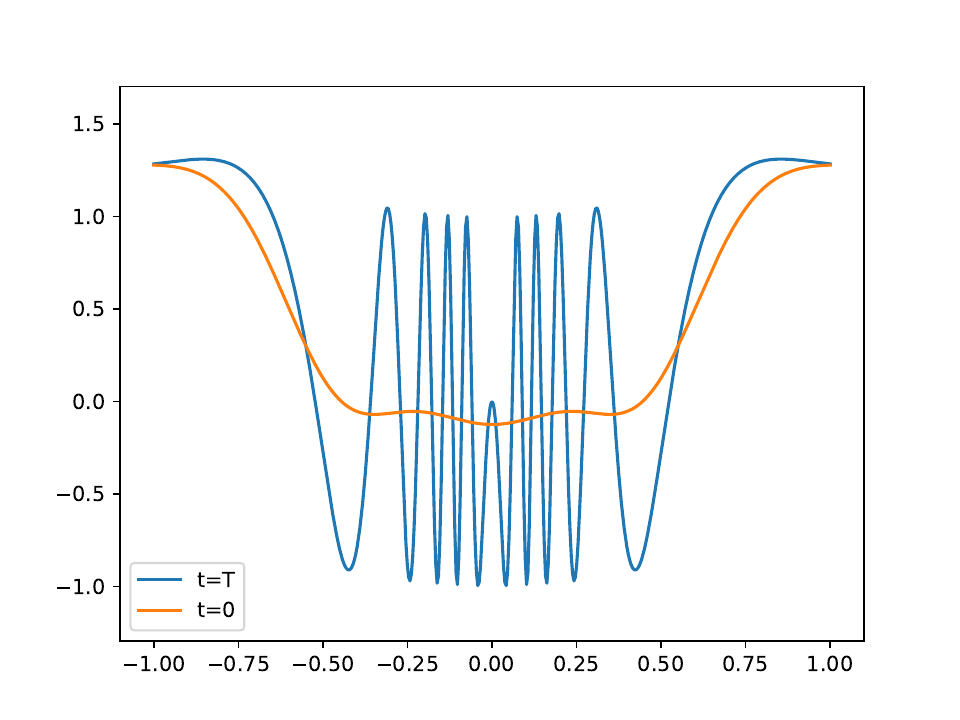}}\\
    \subfloat[$T=1$]{\includegraphics[width=0.24\textwidth]{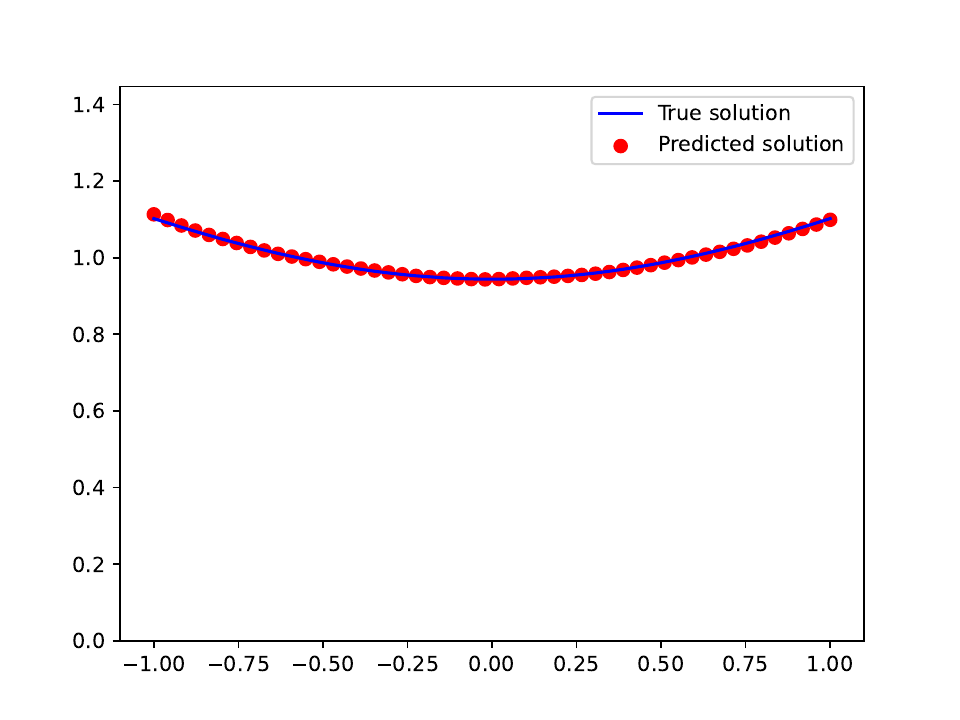}}
    \subfloat[$T=0.1$]{\includegraphics[width=0.24\textwidth]{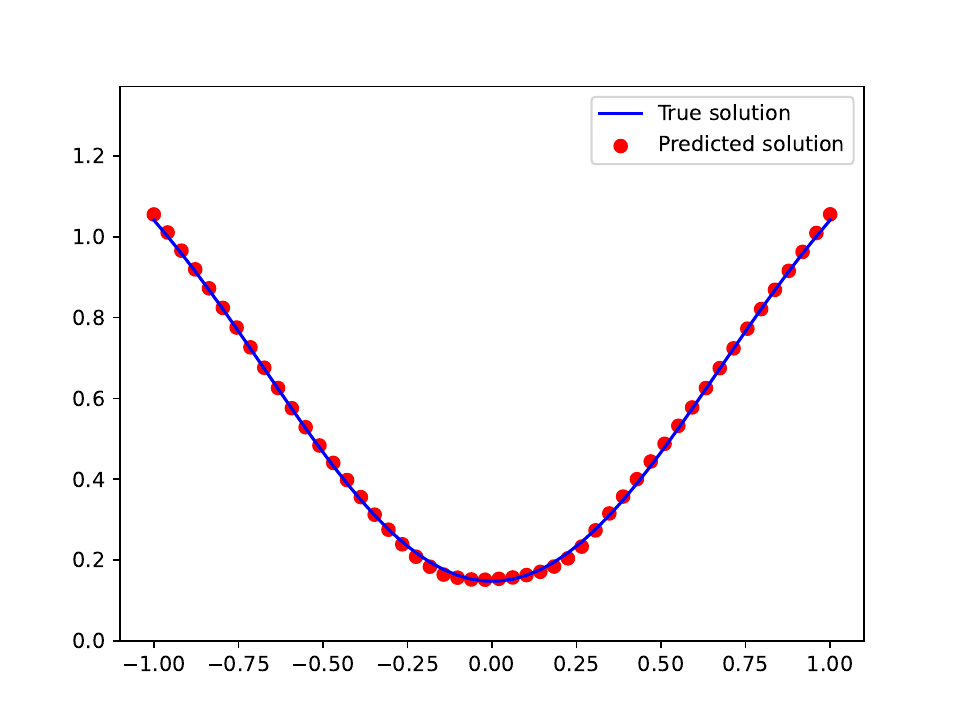}}
    \subfloat[$T=0.05$]{\includegraphics[width=0.24\textwidth]{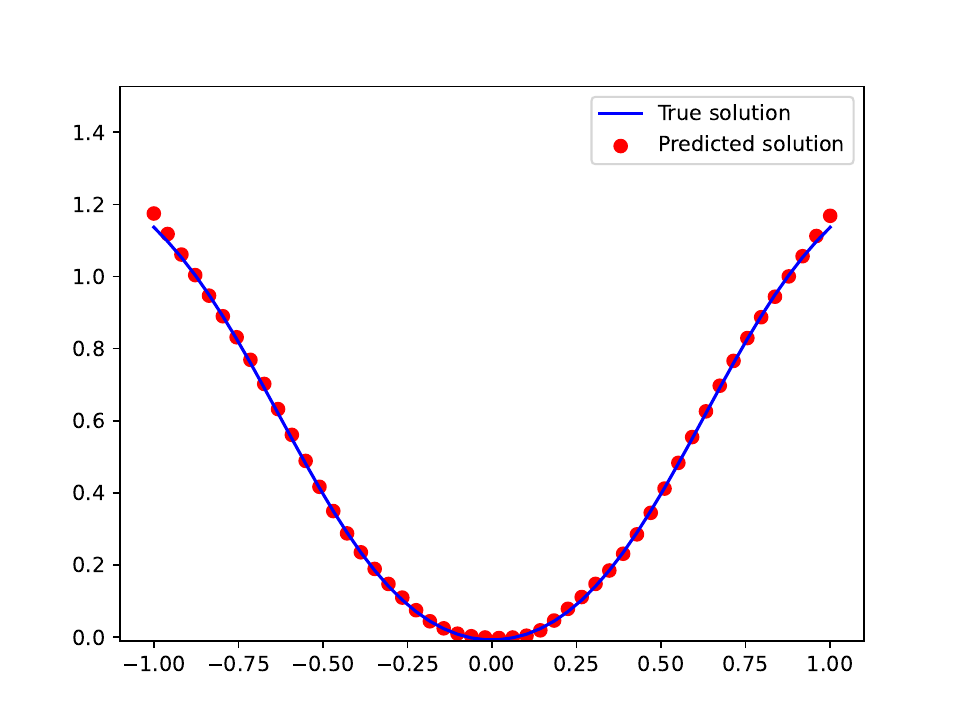}}
    \subfloat[$T=0.01$]{\includegraphics[width=0.24\textwidth]{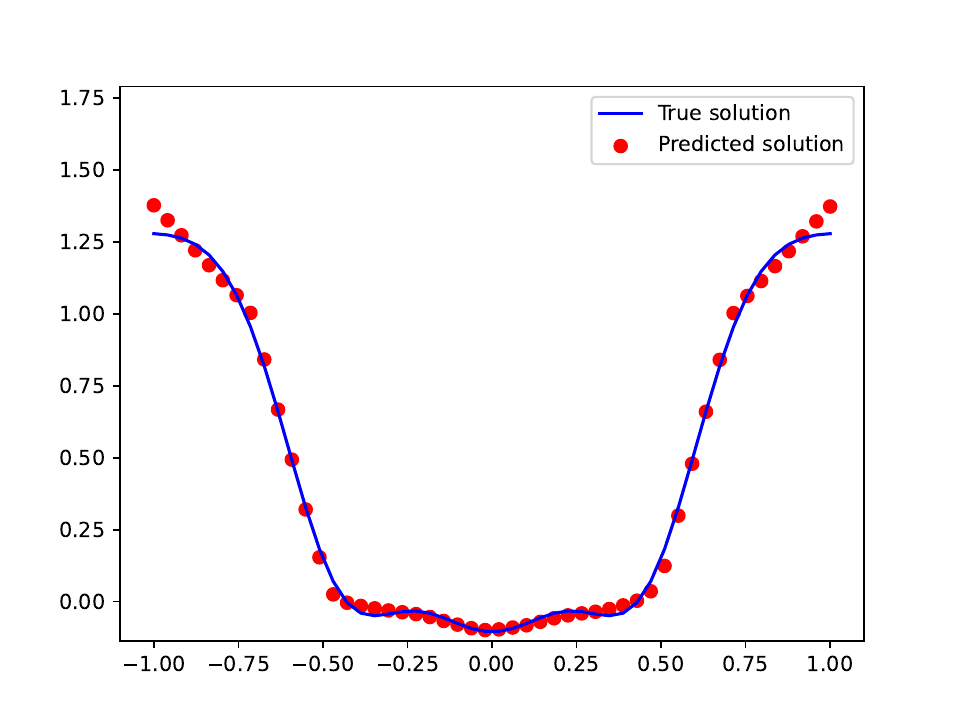}}
    \caption{Numerical results of SOC-MartNet (\Cref{alg_amnet2}) for the semilinear parabolic problem~\eqref{eq_semiparab} with $d=100$ and oscillatory terminal function~\eqref{eq_oscgx}.
    (a) - (d) Graphs of the true solutions $s \mapsto v(t, s \B{1}_d)$ at $t = 0, T$ with varying $T$.
    (e) - (h) Numerical solution of $s \mapsto v(0, s \B{1}_d)$ given by the SOC-MartNet.}\label{fig1_hjb2_turet0te}
\end{figure}

\subsection{Test on time convergence rate}\label{sec_testcr}
\cb{
We consider a specific parabolic equation \eqref{eq_parab} with a variable coefficient in $\mathcal{L}$ and an Allen-Cahn-type source term $f$, i.e.,
\begin{equation}\label{eq_linsin}
\begin{aligned}
    &\mathcal{L} = \sum_{i=1}^d \sin(2 x_i) \partial_{x_i} + \frac{1}{2} \sum_{i=1}^d \br{1 + 0.5 \sin(5 t + x_i)}^2 \partial_{x_i}^2, \\
    &f(t, x, v, \partial_x v, \partial_{xx}^2 v) = v - v^3 + \bar{f}(t, x), 
\end{aligned}
\end{equation}
where the function $\bar{f}$ and the terminal condition $v(T, x) = g(x)$ are chosen such that the true solution is given by \eqref{eq_linearsol}.  
The relevant numerical results of SOC-MartNet (\Cref{alg_amnet2}) are shown in \Cref{fig_cr}. 
The relative error $\mathrm{RE}_1$ closely aligns with the red reference line with slop $1.01$, indicating a first-order convergence rate of $O(N^{-1.01})$, and in line with our claim in \Cref{rmk_ord1}. 
}

\begin{figure}[htbp]
    \centering
    \begin{tikzpicture}
        \begin{loglogaxis}[
            width=0.55\textwidth,
            xlabel={$N$},
            ylabel={$\mathrm{RE}_1$},
            xtick={3, 4, 5, 6, 8, 12, 18, 26, 38, 56},
            xticklabels={3, 4, 5, 6, 8, 12, 18, 26, 38, 56},
            ytick={0.007, 0.01, 0.02, 0.03, 0.04, 0.05, 0.06, 0.07},
            log ticks with fixed point,
            transpose legend,
            legend pos=outer north east,
            ymajorgrids=true,
            xmajorgrids=true,
            grid style=dashed]
        \addplot[only marks, color=blue]
        coordinates{
            (3, 0.0643667 )
            (4, 0.04480899) 
            (5, 0.03325969) 
            (6, 0.02660629) 
            (8, 0.02002971) 
            (12, 0.01459832) 
            (18, 0.0103175 )
        };
        \addlegendentry{Numerical results}
        \addplot [domain=2.9:19, samples=100, color=red]{10^(-0.75) * x^(-1.01)};
        \addlegendentry{$\mathrm{RE}_1 = 10^{-0.75} N^{-1.01}$}
        \end{loglogaxis}
    \end{tikzpicture}
    \caption{
    \cb{
    Log-log plot of the relative $L^1$-error $\mathrm{RE}_1$ of SOC-MartNet (\Cref{alg_amnet2}) vs the number of time partitions $N$ for the parabolic equation \eqref{eq_parab} with parameter setting~\eqref{eq_linsin} and with $d=100$. 
    The red reference line indicates a first-order convergence rate of $O(N^{-1.01})$.}
    }\label{fig_cr}
\end{figure}


\subsection{Non-degenerated HJB equation without using explicit form of \texorpdfstring{$\inf_u H$}{inf\_u H}}\label{sec_nondegHjb}
\cb{
We consider the following HJB equation \cite[Section 3.1]{Bachouch2022Deep}:
\begin{equation}\label{eq_HjbLq}
\left\{\begin{aligned}
    &\br{\partial_t + b^{\top} \partial_x + \epsilon_1 \Delta_x} v(t, x) + \inf_{\kappa \in \R^d} \br{2 \kappa^{\top} \partial_x v(t, x) + c_1 \abs{\kappa}^2} = 0 \\
    &v(T, x) = g(x), \quad g(x + b) := \frac{1}{d} \sum_{i=1}^d \bbr{\sin(x_i - \frac{\pi}{2}) + \sin\br{\br{\epsilon_0 + x_i^2}^{-1}}}
\end{aligned} \right. 
\end{equation}
with $(t, x) \in [0, T) \times \R^d$, where $b \in \R^d$, $c_1, \epsilon_0$ and $\epsilon_1 > 0$ are parameters to be specified.
The HJB equation~\eqref{eq_HjbLq} is associated with the SOCP:
\begin{align}
    &u^* = \argmin_{u \in \mathcal{U}_{\mathrm{ad}}} J(u), \quad J(u) := 1 + \E{\int_0^T c_1 \abs{u_s}^2 \di s + g\br{X_T^u}}, \label{eq_socp_example} \\
    &\mathcal{U}_{\mathrm{ad}} = \bbr{u: [0, T] \times \Omega \to \R^d: \text{$u$ is $\mathbb{F}^B$-adapted}}, \notag\\
    &X_t^u = X_0 + \int_0^t (b + 2 u_s) \di s + \int_0^t \sqrt{2 \epsilon_1} \di B_s, \quad t \in [0, T],\label{eq_socp_sde} 
\end{align}
where $B$ is a $d$-dimensional standard Brownian motion.
For this example, we take $c_1 = \epsilon_1^{-2}$ such that the HJB \eqref{eq_HJBPDE} admits an analytic solution given by  
\begin{equation}\label{HJBsol}
    v(t, x) = - \ln\br{\E{\exp \br{-g(X_T^{t, x})}}}, \quad X_T^{t, x} := x + (T-t)b + \sqrt{2 \epsilon_1} B_{T-t}.
\end{equation}
}

\cb{
The specific case of the HJB equation \eqref{eq_HjbLq}, the associated SOCP~\eqref{eq_socp_example} and parabolic counterpart~\eqref{eq_semiparab}}, have been used often as benchmark problems in related literature, e.g., \cite{weinan2017deep,Ji2022Solving,Bachouch2022Deep,han2018solving,wang2022is,He2023Learning,raissi2018forwardbackward,hu2024sdgd}.
However, most studies have focused on the terminal function $g(x) = \ln(0.5 (1+\abs{x}^2))$, resulting in a smooth exact solution $v$ with inherent low-dimensional structure.
In this work, we consider a more complicated $g(x)$ given in \eqref{eq_HjbLq}, under which the exact solution $x \mapsto v(t, x)$ is anisotropic along spatial dimensions and highly oscillatory around $t = T$; see \cref{fig1_hjb2_turet0te}.
Thus this example can effectively demonstrate the ability of numerical methods to address complex high-dimensional problems.
    
\cb{
In the numerical tests, specific settings of \eqref{eq_HjbLq} are considered.
By selecting different values of $b$, $\epsilon_0$, and $\epsilon_1$,  
three particular instances of \eqref{eq_HjbLq} are obtained:
\begin{itemize}
    \item HJB-1: $b = \br{0, 0, \cdots, 0}^{\top} \in \R^d$, $\epsilon_0 = 0.1 \pi$, $\epsilon_1 = 1$. 
    \item HJB-2: $b = \br{1, 1, \cdots, 1}^{\top} \in \R^d$, $\epsilon_0 = 0.3 \pi$, $\epsilon_1 = 0.2$;
    \item HJB-3: $b = \br{1, 1, \cdots, 1}^{\top} \in \R^d$, $\epsilon_0 = 0.3 \pi$, $\epsilon_1 = 0.1$.
\end{itemize}
With the optimal control in \eqref{eq_HjbLq} explicitly solved, HJB-1 reduces to the semilinear parabolic equation \eqref{eq_semiparab}. 
HJB-2 and HJB-3 are more intricate variants of HJB-1, featuring a non-zero drift coefficient $b$ and a smaller $\epsilon_1$. 
As shown in \Cref{fig_hjb2}, the solutions $x \mapsto v(t, x)$ of HJB-2 and -3 remains oscillatory at $t = 0$, even for $T=1$.
}

\cb{
The SOC-MartNet (\cref{alg_amnet}) is applied to HJB-1 to -3, where no explicit form is used for $\inf_{\kappa \in U} H$, and the optimal control $u^*$ is approximated by  the neural network $u_{\alpha}$.
The numerical method solves $v(0, x)$ for $x \in D_0$ with $D_0$ consits of $M$ uniformly spaced gird points on $S_2 \cup S_3$, where $S_2$ is the line segment defined in \eqref{eq_defS2} and $S_3 := \bbr{\B{l}(s): s \in [-1, 1]}$ is a space curve in $\mathbb{R}^d$ with
\begin{equation}\label{eq_defls}
    \B{l}(s) := \br{l_1, l_2, \cdots, l_d}^{\top} \in \R^d, \quad l_i := s \times \mathrm{sgn}\br{\sin(i)} + \cos(i + \pi s), \quad s \in \R,
\end{equation}
and $\mathrm{sgn}(z) := -1, 0, 1$ for $z < 0$, $= 0$ and $> 0$, respectively. 
For $d \geq 1000$, the upper bound $\bar{\lambda}$ on Line 11 of \cref{alg_amnet} is specially set to $100$ for $T = 0.01$ and set to $10^3$ for $T = 0.1, 0.5, 1$.
The relevant numerical results and some of the convergence histories are presented in \cref{fig2_hjb2,fig_hjb2}.}

As the numerical results show, even without using the explicit form of $\inf_{\kappa \in U} H$, the SOC-MartNet (\cref{alg_amnet}) still works well for the HJB equation with oscillatory and anisotropic terminal function, and with $d$ up to \cb{$10^4$}.

\begin{figure}[t]
    \centering
    \subfloat[$d=100$, $T=1$]{\includegraphics[width=0.24\textwidth]{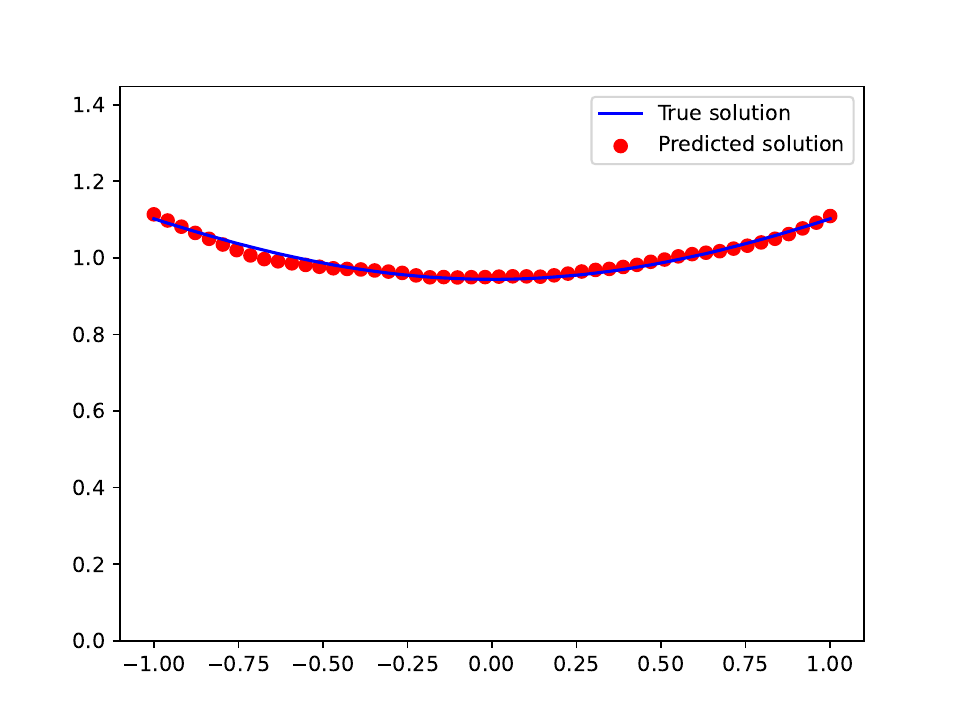}}
    \subfloat[$d=100$, $T=0.5$]{\includegraphics[width=0.24\textwidth]{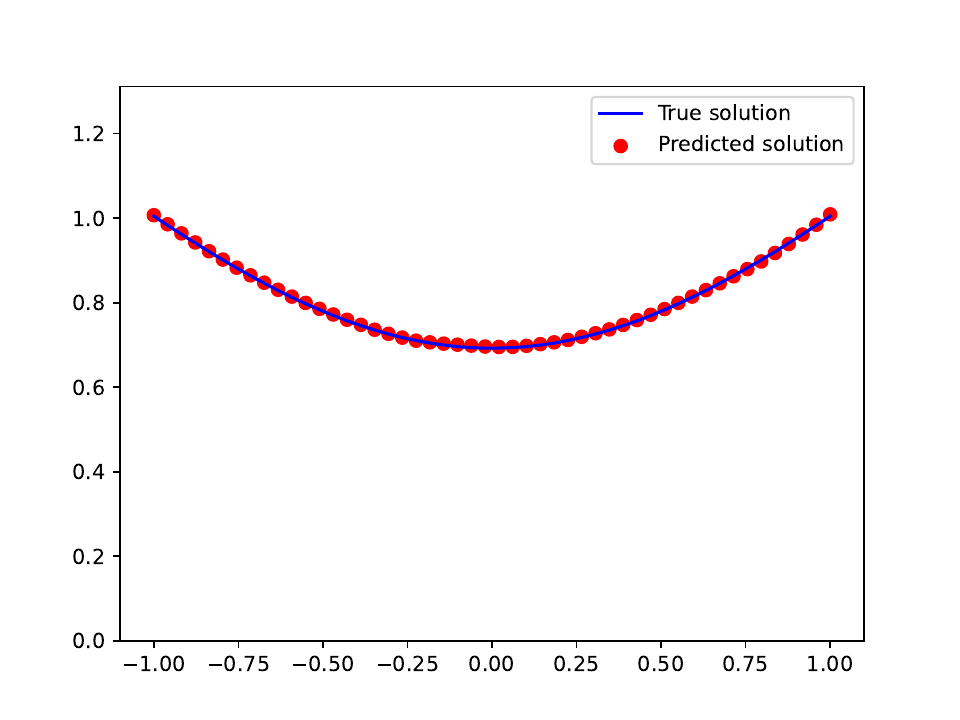}}
    \subfloat[$d=100$, $T=0.1$]{\includegraphics[width=0.24\textwidth]{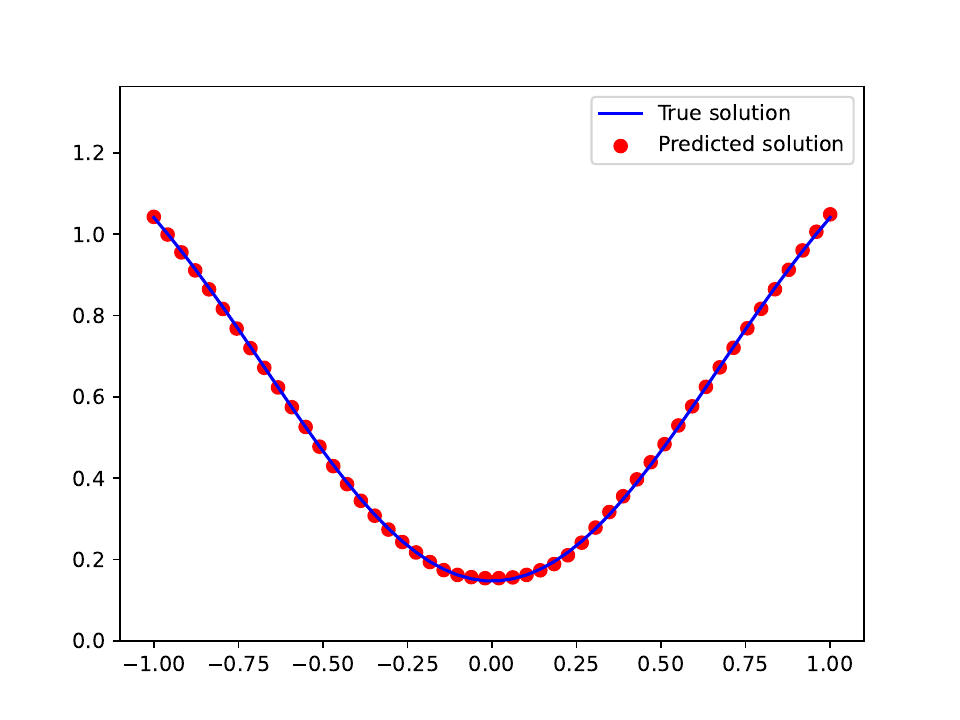}}
    \subfloat[$d=100$, $T=0.01$]{\includegraphics[width=0.24\textwidth]{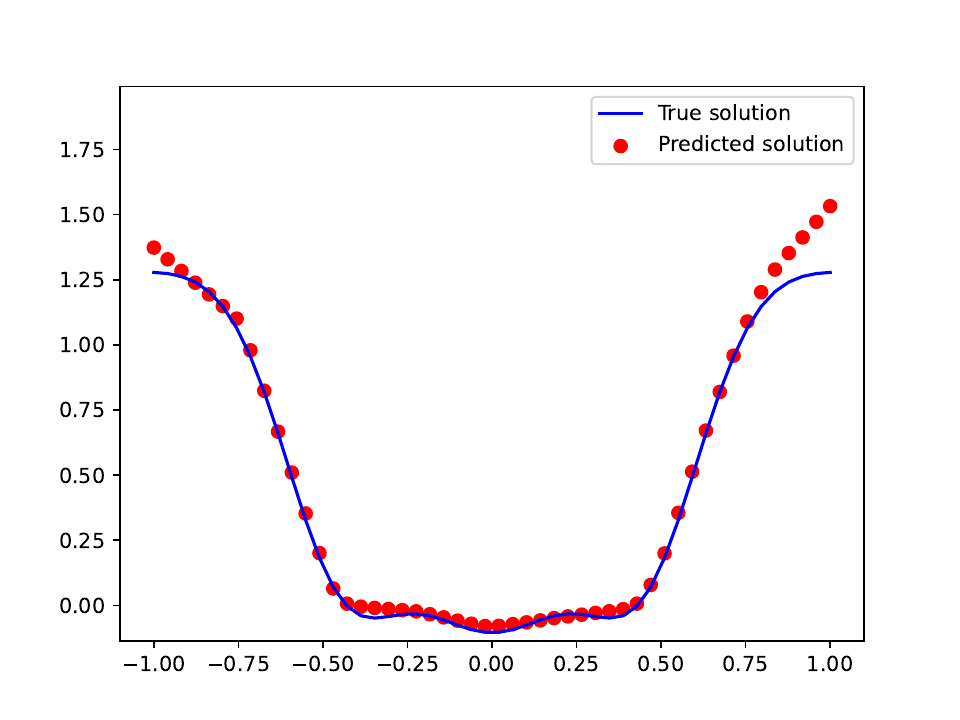}}
    \\
    \subfloat[$d=1000$, $T=1$]{\includegraphics[width=0.24\textwidth]{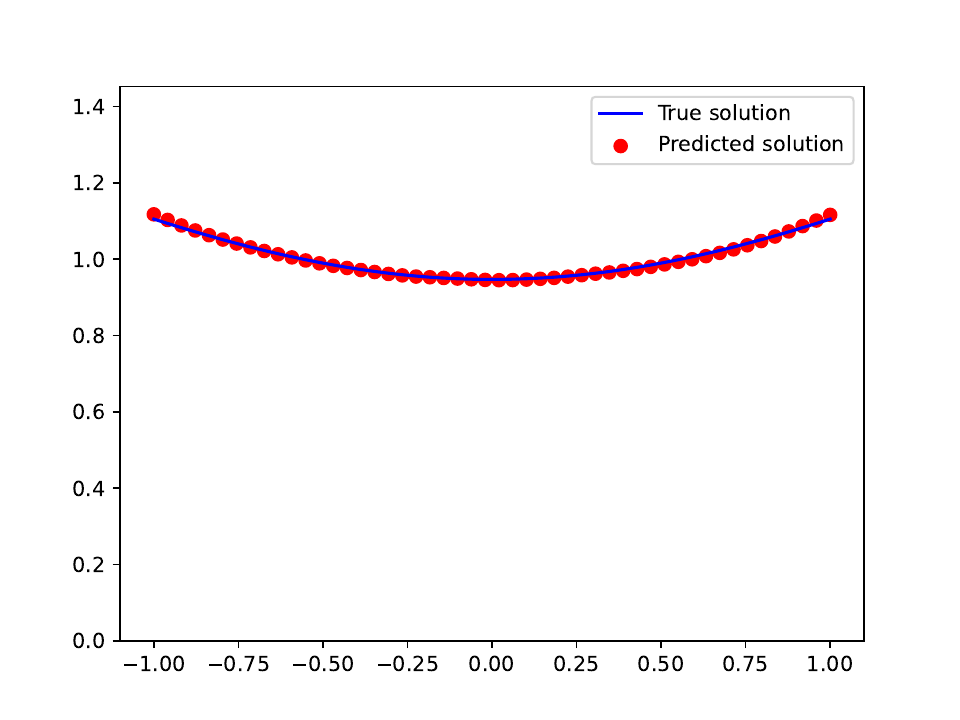}\label{fig2_hjb2_d1e3T1}}
    \subfloat[$d=1000$, $T=0.5$]{\includegraphics[width=0.24\textwidth]{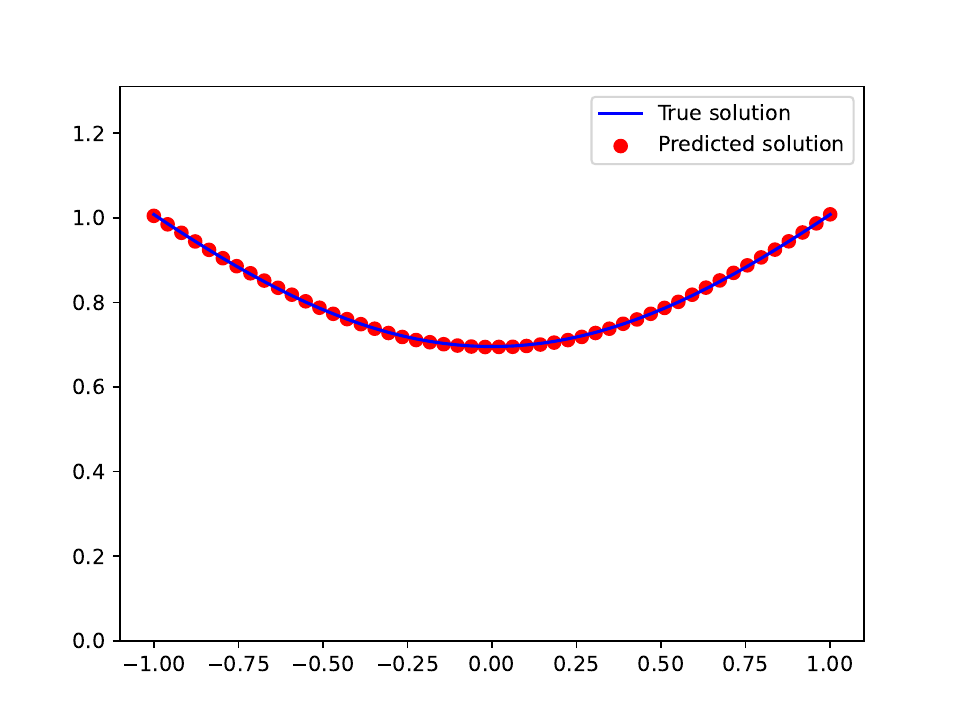}\label{fig2_hjb2_d1e3T5en1}}
    \subfloat[$d=1000$, $T=0.1$]{\includegraphics[width=0.24\textwidth]{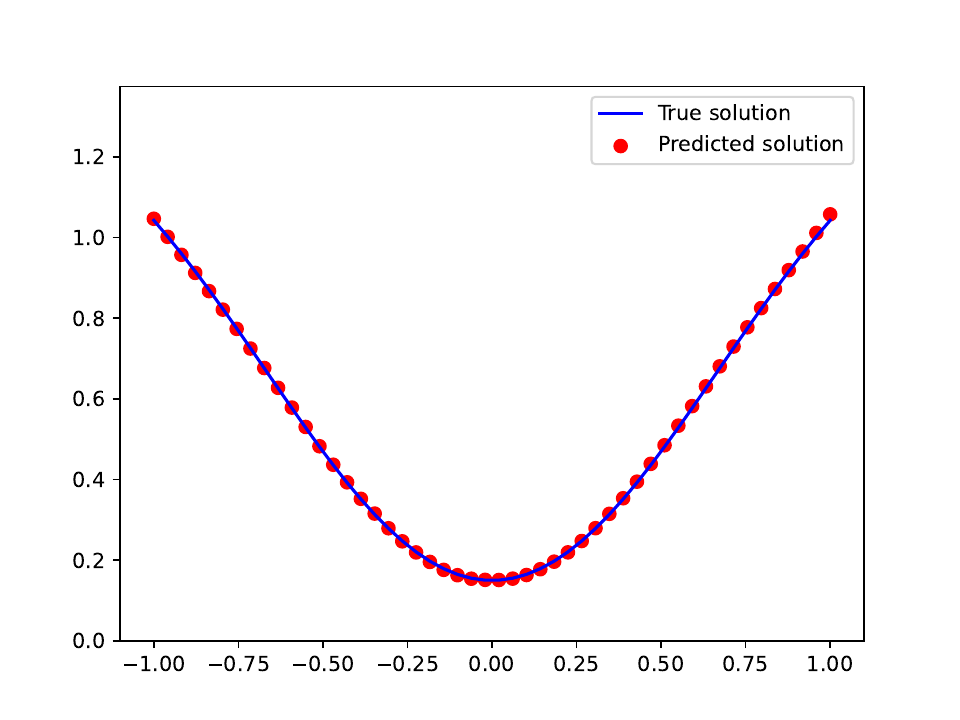}\label{fig2_hjb2_d1e3T1en1}}
    \subfloat[$d=1000$, $T=0.01$]{\includegraphics[width=0.24\textwidth]{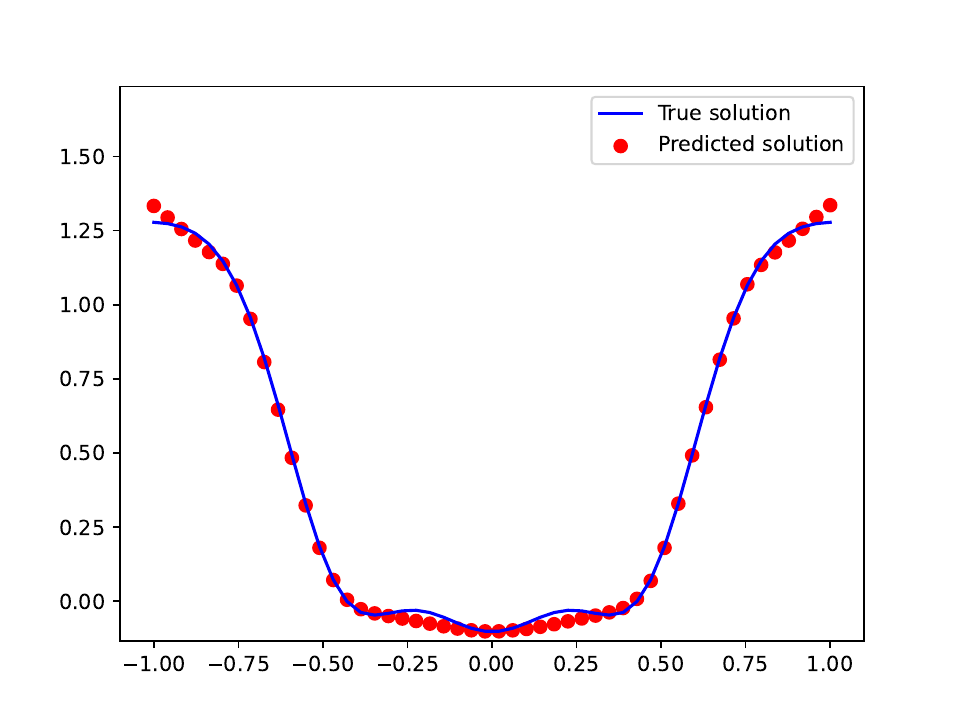}\label{fig2_hjb2_d1e3T1en2}}
    \\
    \subfloat[$d=2000$, $T=1$]{\includegraphics[width=0.24\textwidth]{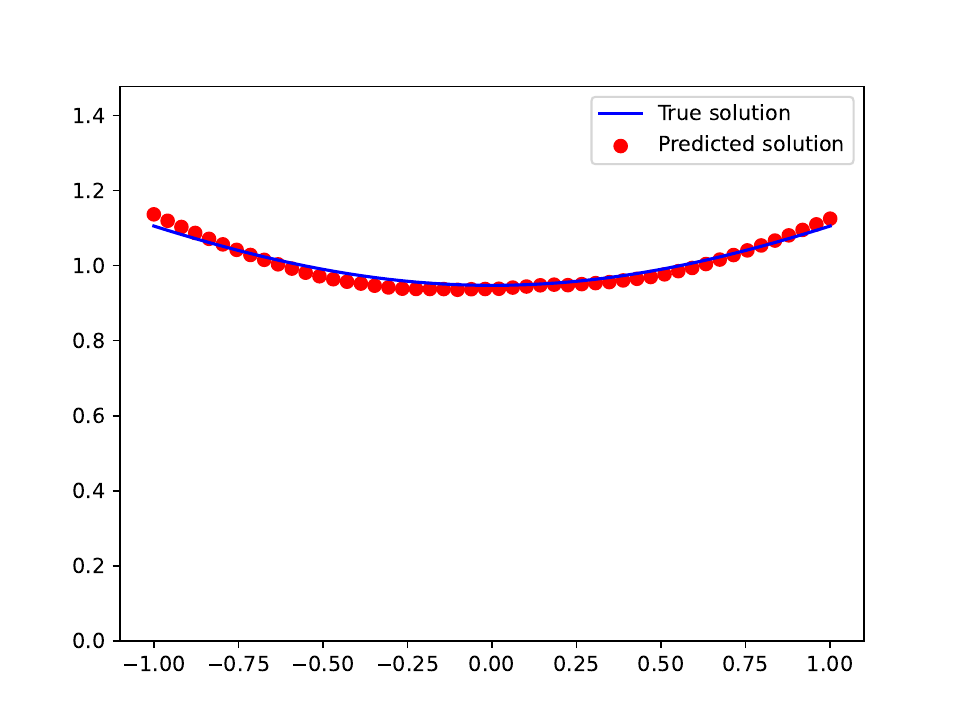}\label{fig2_hjb2_d2e3T1}}
    \subfloat[$d=2000$, $T=0.5$]{\includegraphics[width=0.24\textwidth]{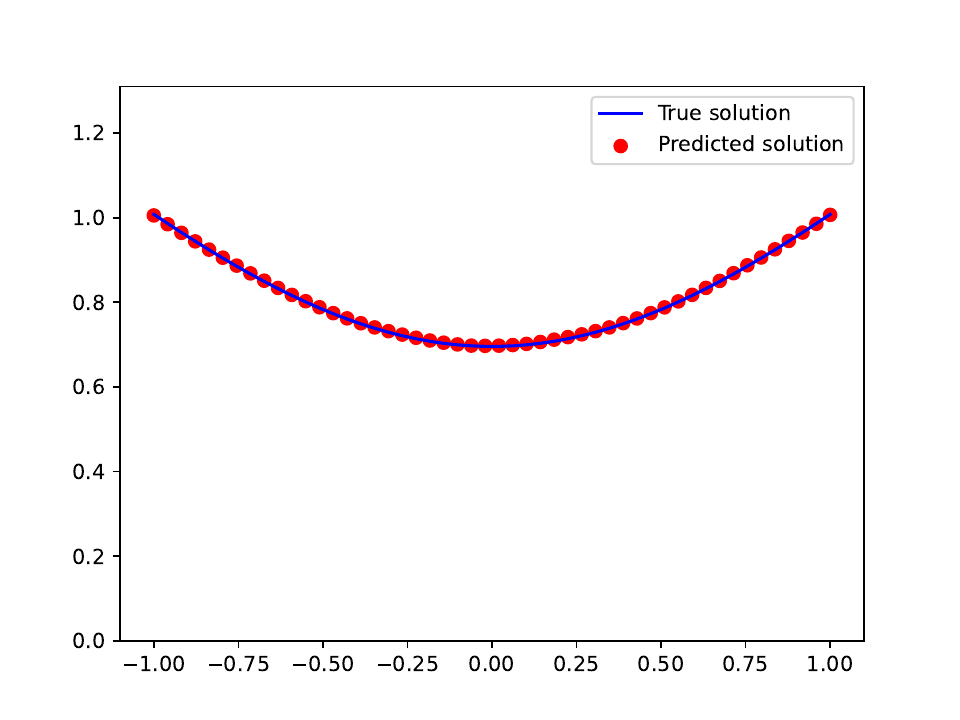}\label{fig2_hjb2_d2e3T5en1}}
    \subfloat[$d=2000$, $T=0.1$]{\includegraphics[width=0.24\textwidth]{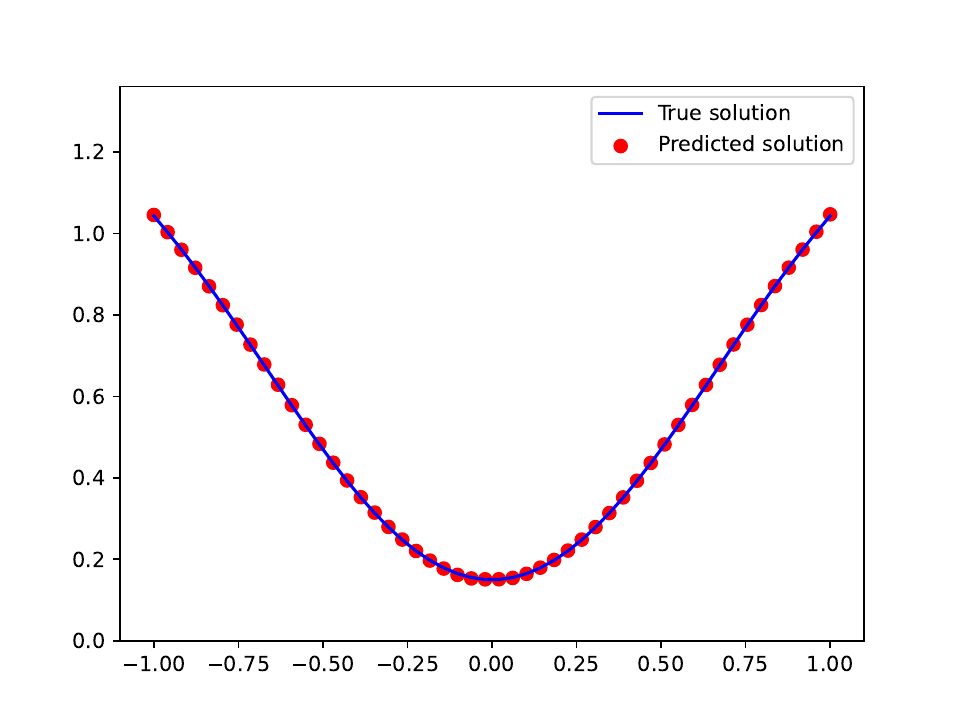}\label{fig2_hjb2_d2e3T1en1}}
    \subfloat[$d=2000$, $T=0.01$]{\includegraphics[width=0.24\textwidth]{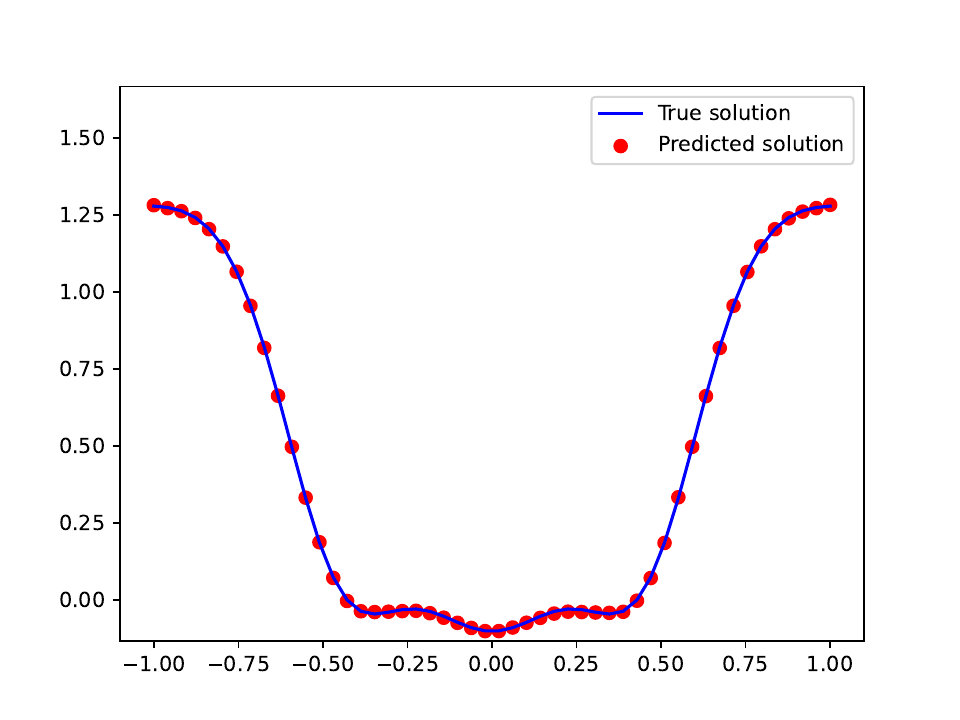}\label{fig2_hjb2_d2e3T1en2}}
    \caption{Graphs of the true solution and the numerical solution of SOC-MartNet for $s \mapsto v(t, s \B{1}_d)$ at $t = 0$ given by HJB-1.
    }\label{fig2_hjb2}
\end{figure}

\begin{figure}[htbp]
    \centering
    \subfloat[$s \mapsto v(0, s \B{1}_d)$, \\ HJB-2, $d=2000$]{\includegraphics[width=0.25\textwidth]{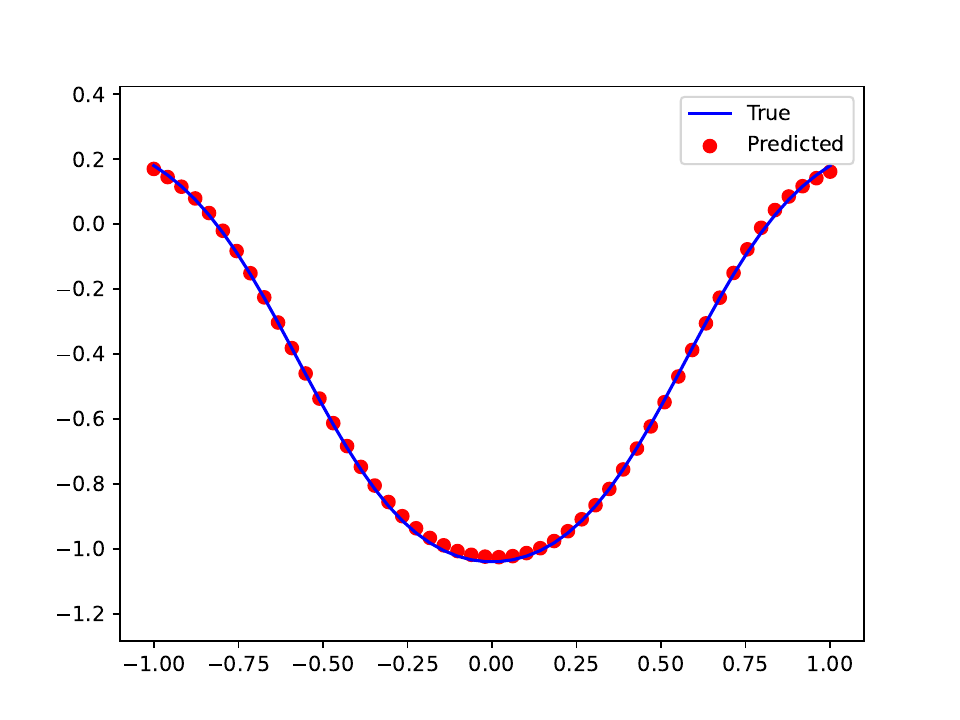}}
    \subfloat[$s \mapsto v(0, \B{l}(s))$, \\ HJB-2, $d=2000$]{\includegraphics[width=0.25\textwidth]{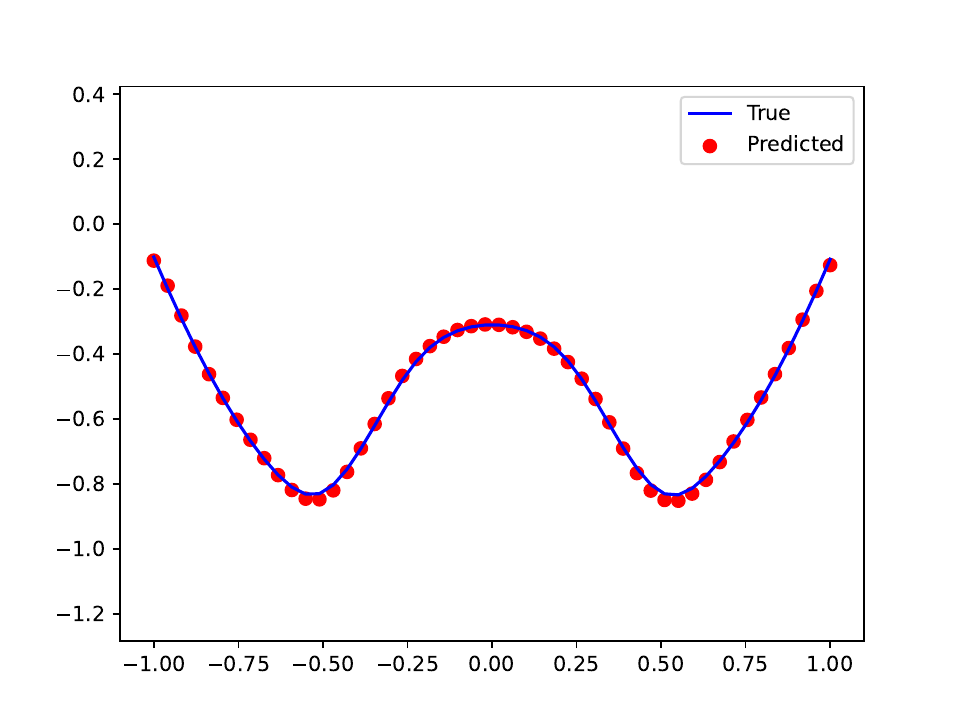}}
    \subfloat[$s \mapsto v(0, s \B{1}_d)$, \\ HJB-3, $d=2000$]{\includegraphics[width=0.25\textwidth]{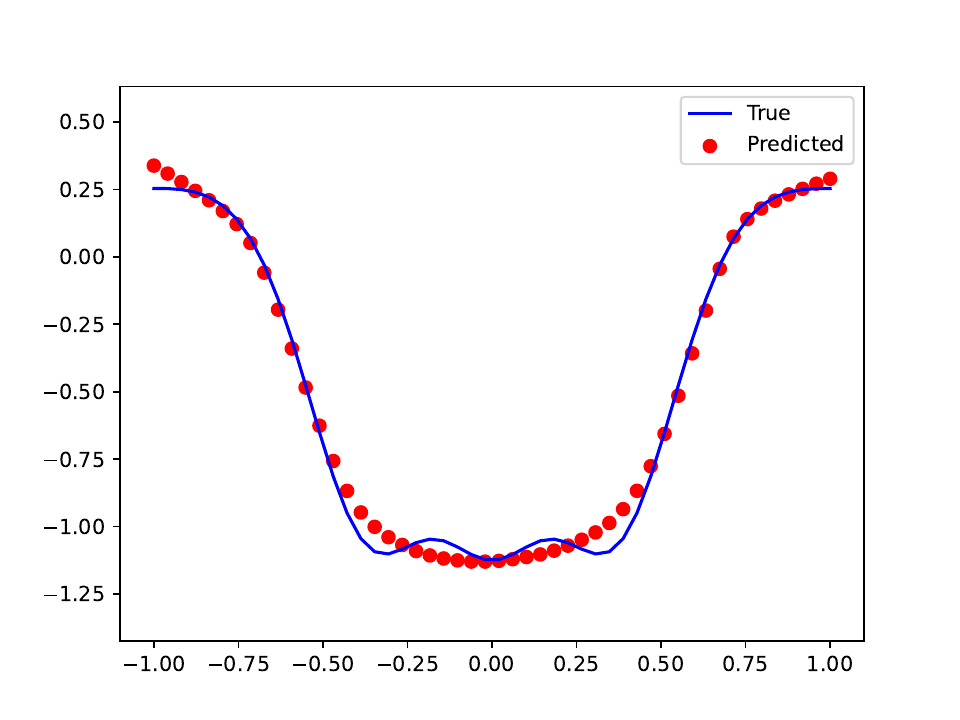}}
    \subfloat[$s \mapsto v(0, \B{l}(s))$, \\ HJB-3, $d=2000$]{\includegraphics[width=0.25\textwidth]{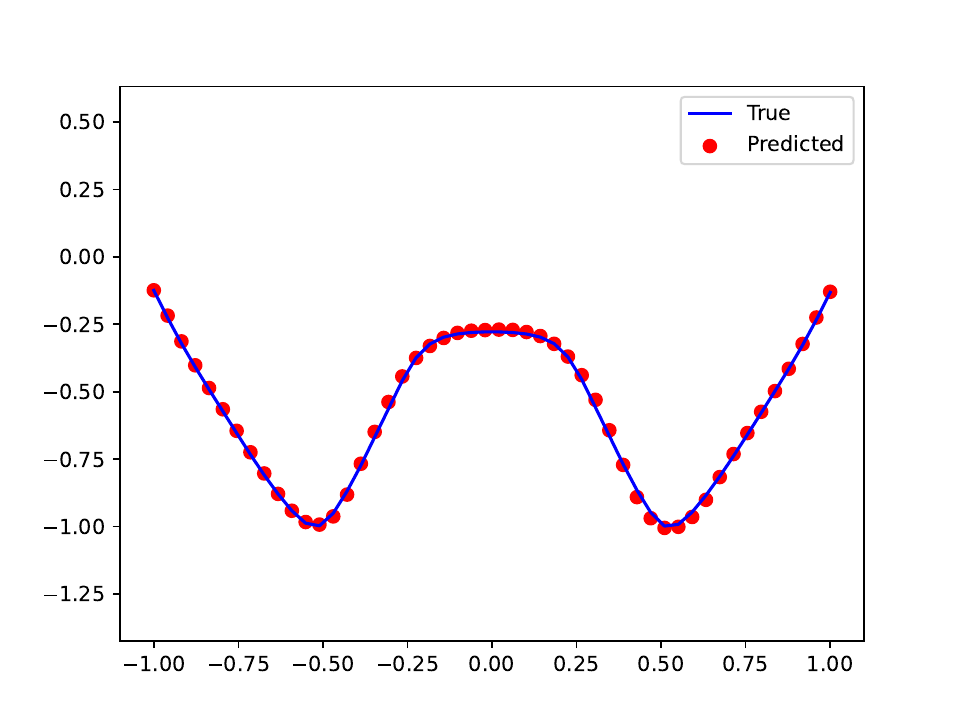}}
    \\
    \subfloat[$s \mapsto v(0, \B{l}(s))$, \\ HJB-2, $d=10^4$]{\includegraphics[width=0.25\textwidth]{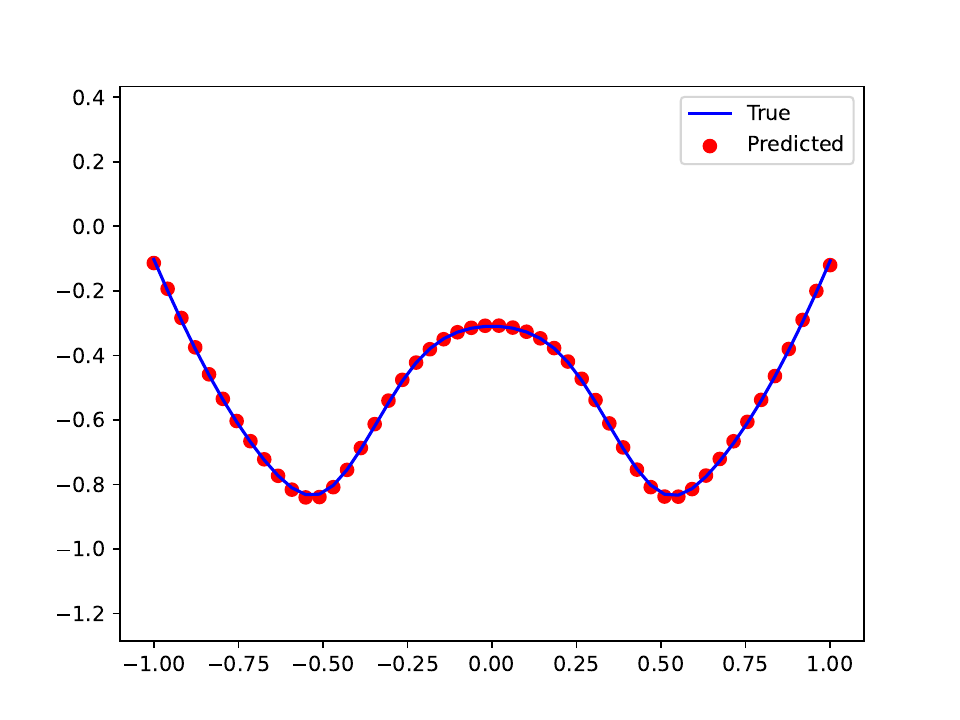}}
    \subfloat[Mart. Loss vs Iter., HJB-2, $d=10^4$]{\includegraphics[width=0.25\textwidth]{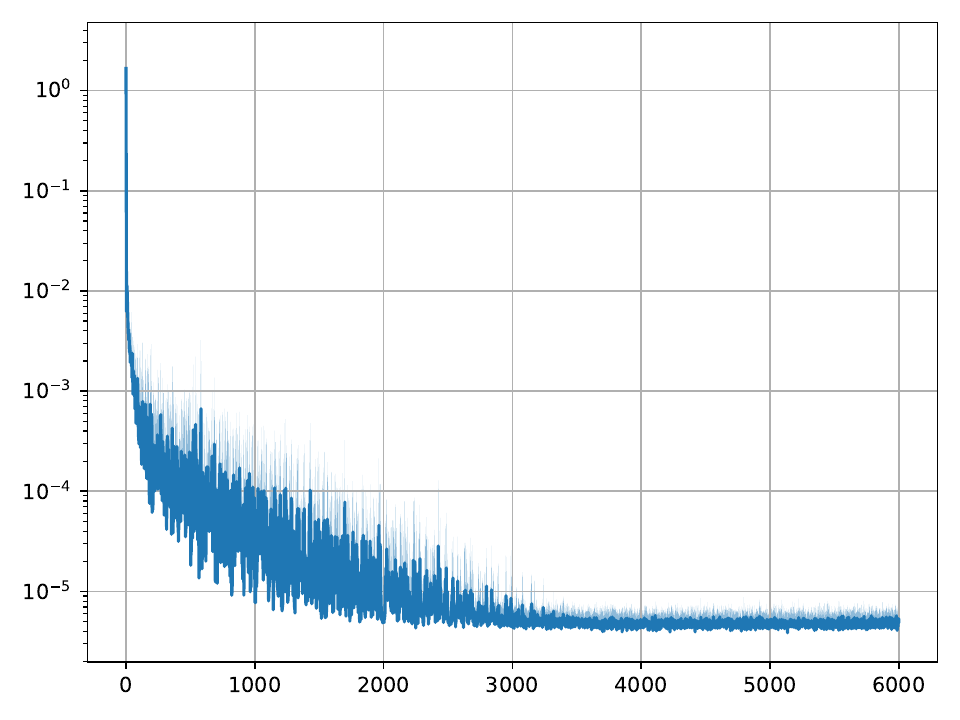}}
    \subfloat[Hamilt. vs Iter., HJB-2, $d=10^4$]{\includegraphics[width=0.25\textwidth]{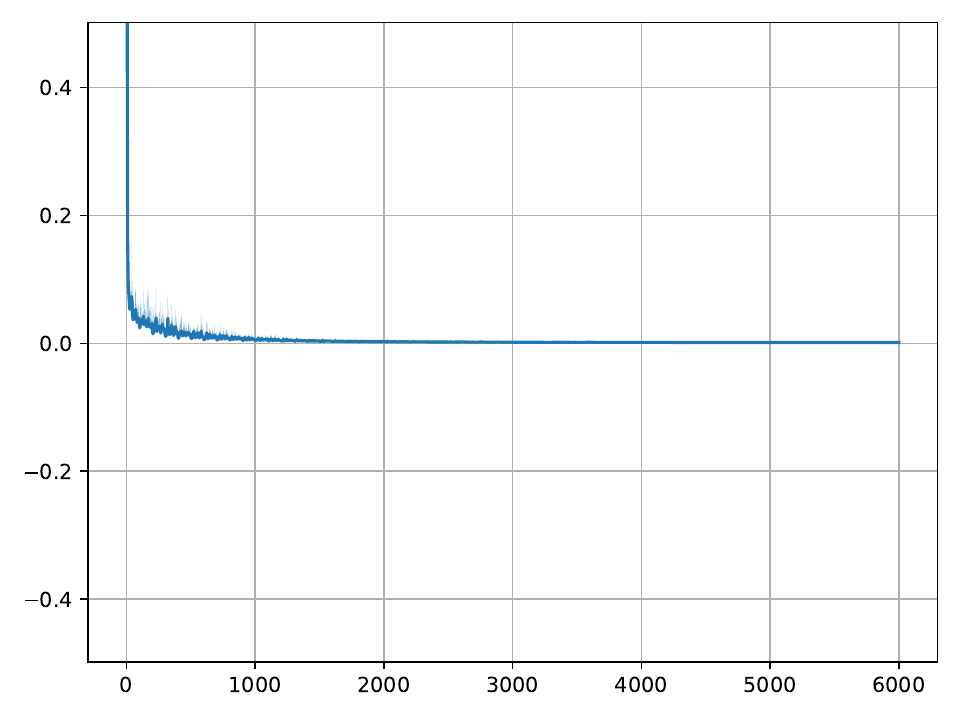}}
    \subfloat[RE vs Iter., HJB-2, $d=10^4$]{\includegraphics[width=0.25\textwidth]{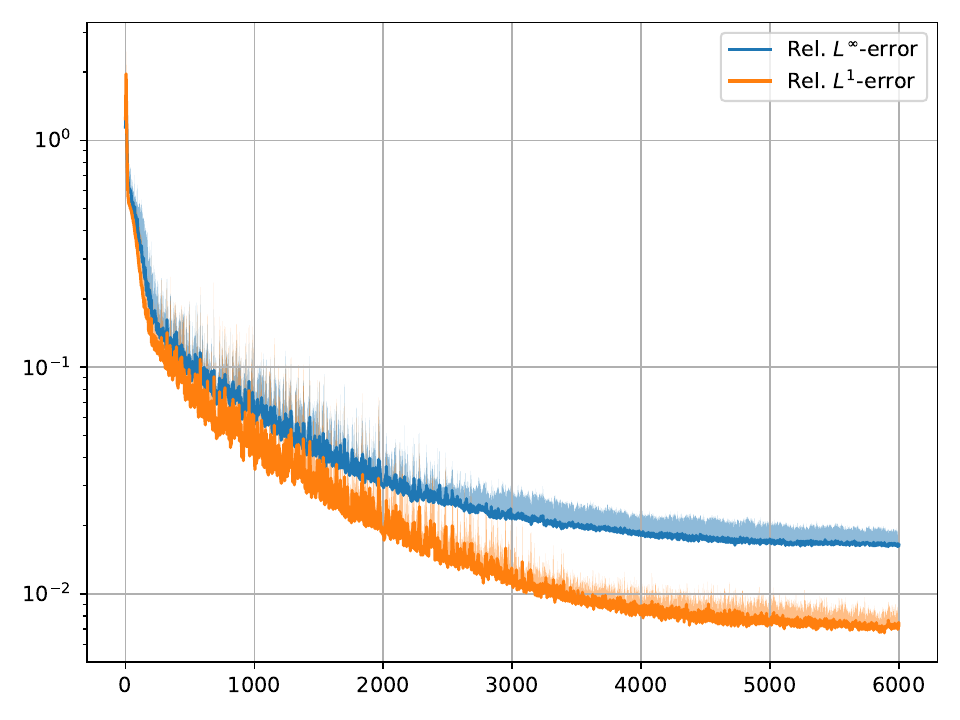}}
    \\
    \subfloat[$s \mapsto v(0, \B{l}(s))$, \\ HJB-3, $d=10^4$]{\includegraphics[width=0.25\textwidth]{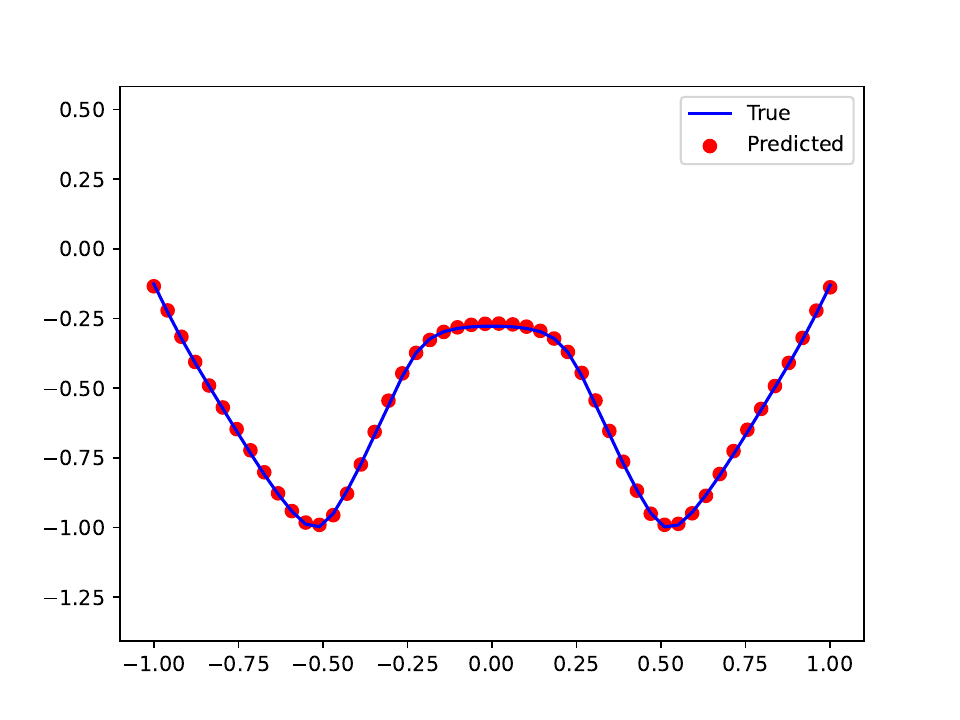}}
    \subfloat[Mart. Loss vs Iter., HJB-3, $d=10^4$]{\includegraphics[width=0.25\textwidth]{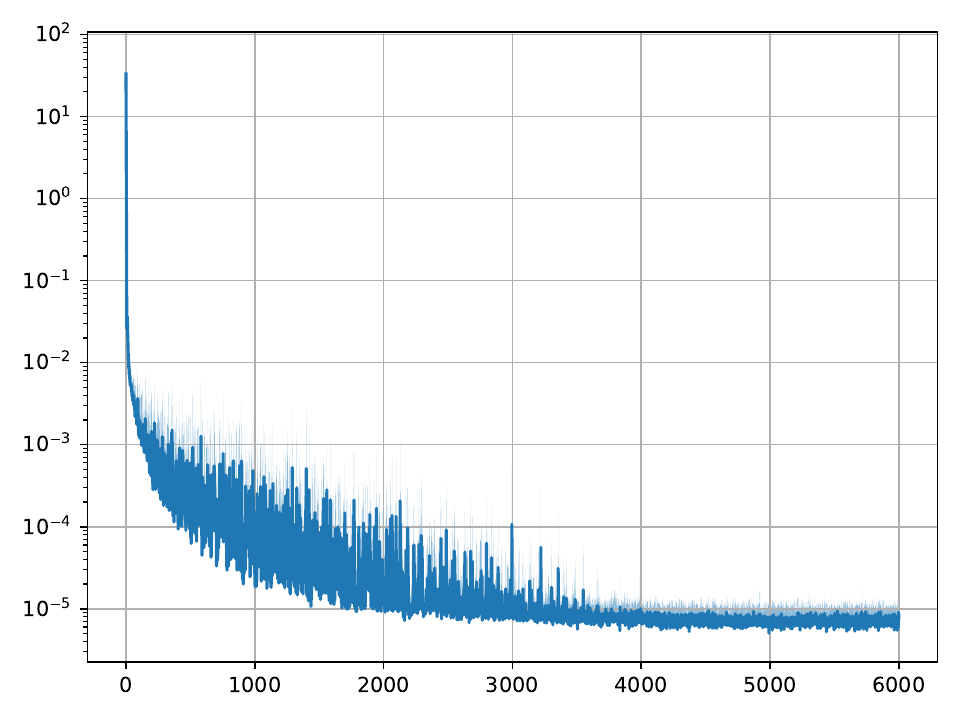}}
    \subfloat[Hamilt. vs Iter., HJB-3, $d=10^4$]{\includegraphics[width=0.25\textwidth]{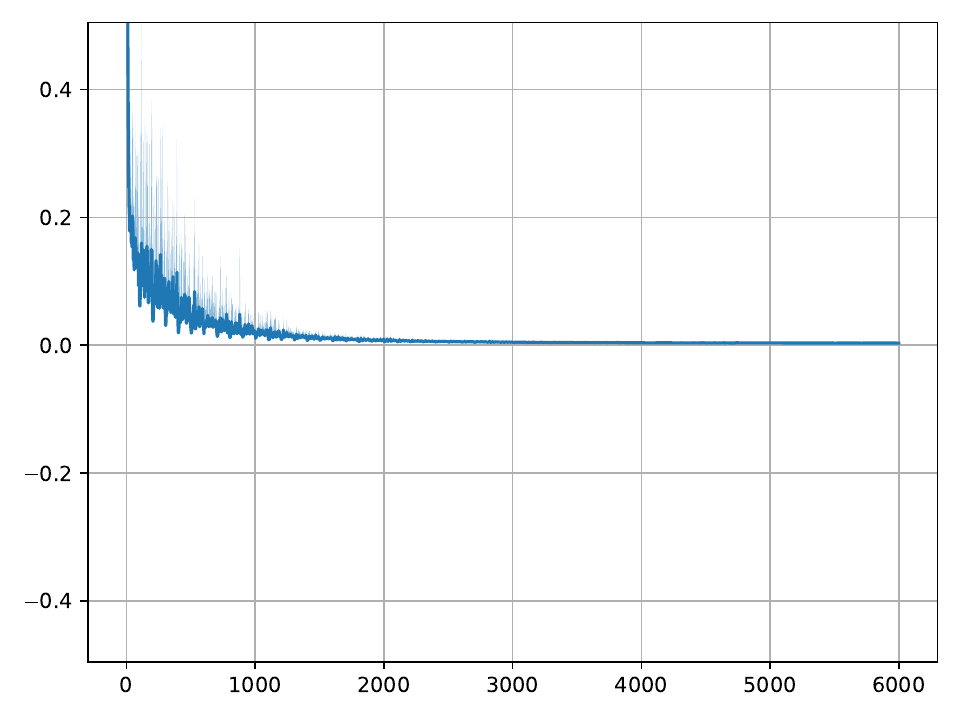}}
    \subfloat[RE vs Iter., HJB-3, $d=10^4$]{\includegraphics[width=0.25\textwidth]{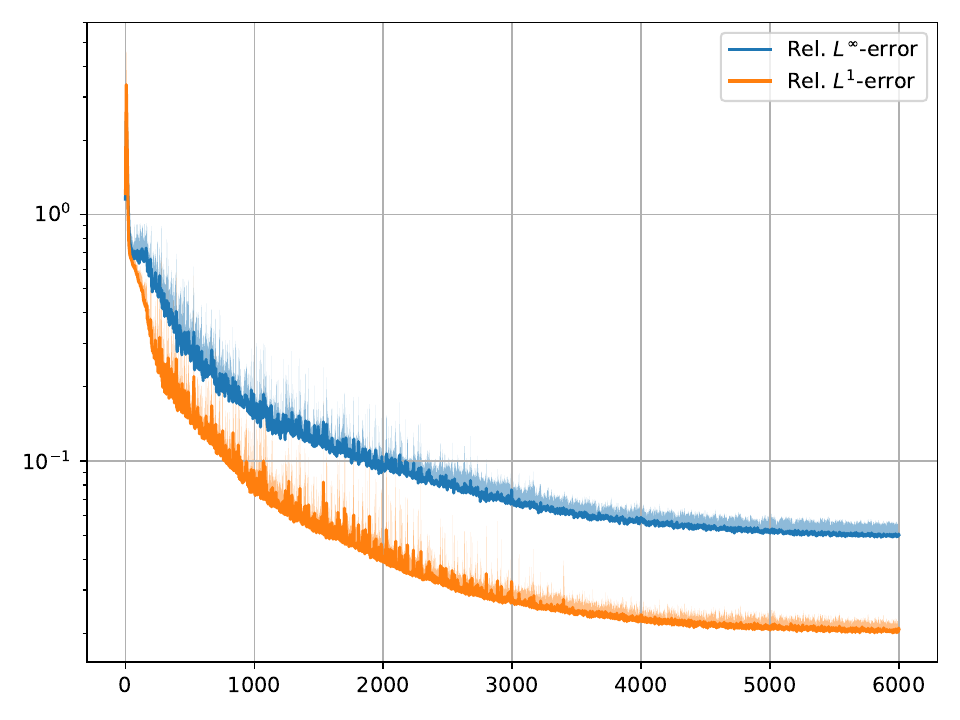}}
    \caption{
    \cb{Numerical results of SOC-MartNet (\Cref{alg_amnet}) for HJB-2 and HJB-3 with $T=1$ and $d = 2000$, $10000$.
    The shaded region represents the mean + $2 \times$ SD of the plotted values across 5 independent runs.}
    }\label{fig_hjb2}
\end{figure}

\cb{
\subsection{Validity of SOC-MartNet solution in space-time region}\label{sec_path_valid}

To illustrate the generalization of our method, we present \Cref{fig_socp_domain} for HJB-2 with $d=1000$. 
The SOC-MartNet solutions remain close to the exact solution for $(t, x)$ outside the time-space region given by $t = 0$ and $x \in D_0$. 
The relative error on the sample paths $\{X_n\}_{n=0}^N$ grows slightly with $t$ as expected given that the variance of $X_t$ increases over time, causing sparser path distributions.
Despite this, the relative errors remain satisfactory across the entire time interval, indicating that SOC-MartNet performs well in the region explored by $X_t$.
}
\begin{figure}[t]
    \centering
    \subfloat[$s \mapsto v(r, s\B{1}_d + r \B{1}_d)$ with $r = 0.125$]{\includegraphics[width=0.33\textwidth]{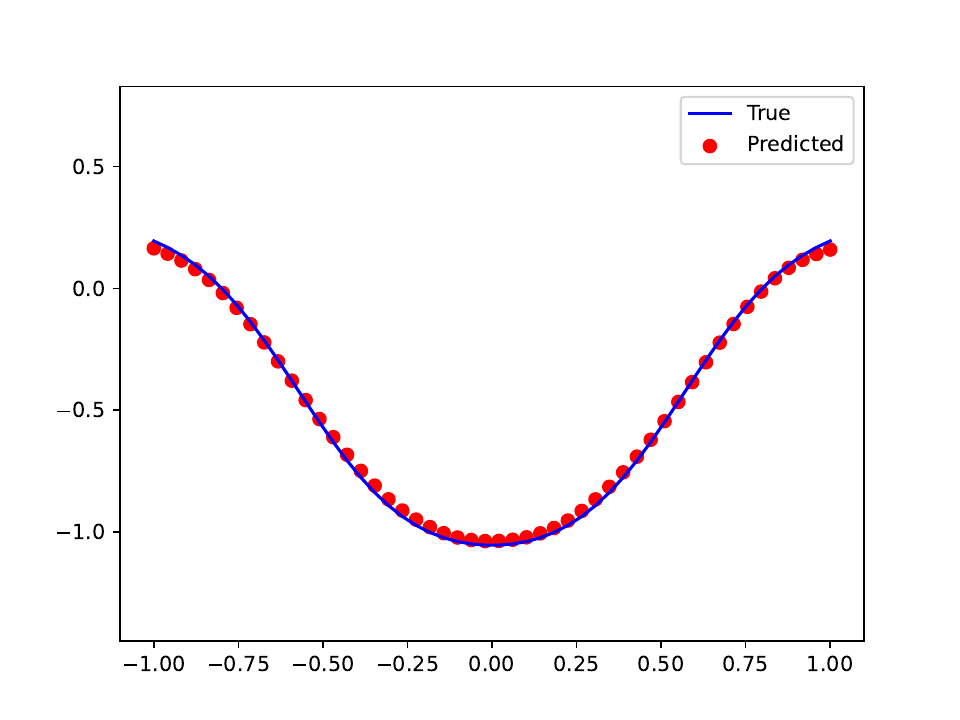}}
    \subfloat[$s \mapsto v(r, s\B{1}_d + r \B{1}_d)$ with $r = 0.25$]{\includegraphics[width=0.33\textwidth]{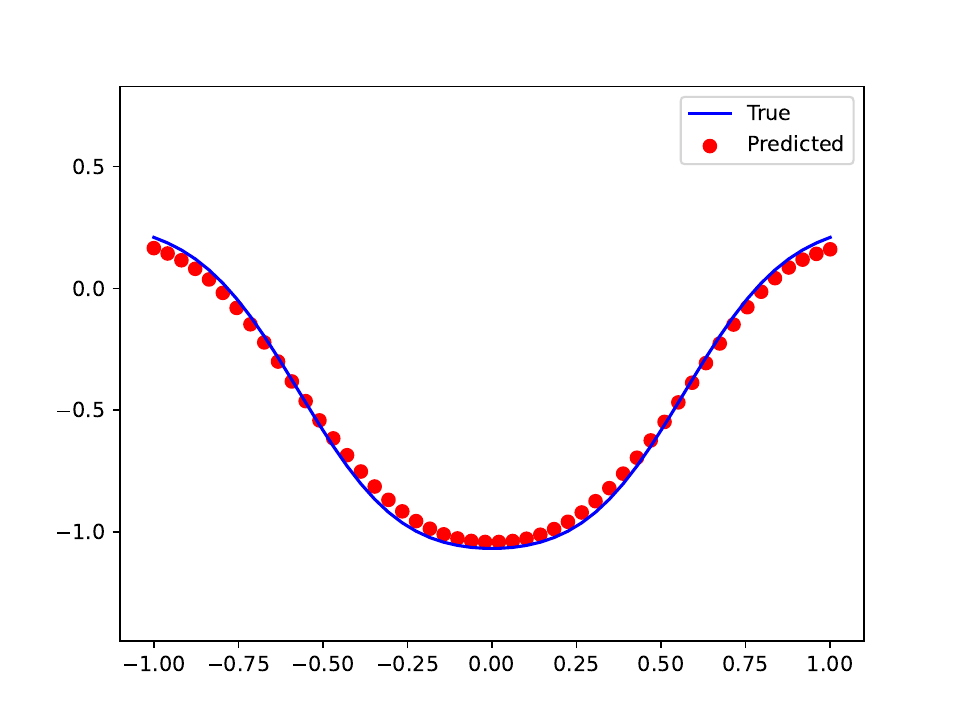}}
    \subfloat[RE vs $t$]{\includegraphics[width=0.33\textwidth]{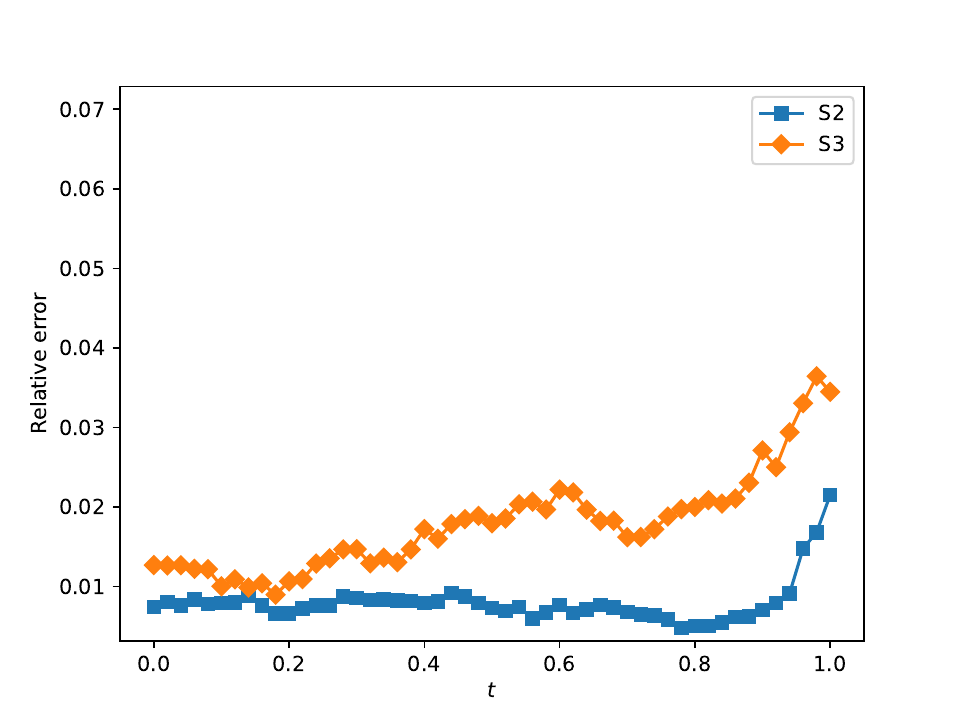}}
    \caption{
    \cb{(a) (b) Numerical solutions of SOC-MartNet (\Cref{alg_amnet}) for HJB-2 with $d=1000$ and $s \in [-1,1]$.
    (c) The RE of the solution along on 8 paths of the process $\bbr{X_n}_{n=0}^N$ given by \eqref{eq_EulerSDE} with $X_0 = 0$ for S2, and $X_0 = \B{l}(0.75)$ for S3.}
    }\label{fig_socp_domain}
\end{figure}

\cb{
\subsection{SOCP with a shifted target}\label{sec_scop_st}

We consider a SOCP with a shifted target function in the terminal condition whose minimum is away from the origin as considered in \cite{Li2024neural}. 
Specifically, the considered SOCP is given by \eqref{eq_socp_example} with $T=1$, $b = 0$, and with the terminal function $g$ replaced by 
\begin{equation}\label{eq_defg}
    g(x) = C_g \ln\br{0.5 \vbr{\big}{1 + \sum_{i=1}^d \br{x_i - 3}^2}}, 
\end{equation}
where we set $C_g=10$ so the exact solution expression \eqref{HJBsol} will not have an overflow issue.
The terminal function encourages the control $u$ to drive the state process $X_t^u$ to get close to the target point $\br{3, 3, \cdots, 3}^{\top} \in \R^d$ at the terminal time.
To avoid trivial solutions caused by noise-dominated state equations, we take a small diffusion coefficient as $\epsilon_1 = 0.1$ in \eqref{eq_socp_sde}.
}

\cb{
We apply the SOC-MartNet (\cref{alg_amnet}) to solve $v(0, x)$ for $x \in D_0$, where $D_0$ consists of $M$ uniformly spaced grid points on the line segment $S_2$ defined in \eqref{eq_defS2}.
At each iteration, we record the empirical values of $J(u_{\theta})$, which are estimated via Monte Carlo sampling based on 256 sample paths of the state process in \eqref{eq_stateq} approximated by the Euler-Maruyama method with $x_0 = 0$.
The relevant numerical results for $d=100$ and $500$ are presented in \cref{fig_socp_shifted}, which shows that the SOC-MartNet provides accurate solutions for the value function.
Moreover, the cost value of $u_{\theta}$ quickly converges to the theoretical lower bound $J(u^*) := v(0, x_0)$.
}

\cb{
}

\begin{figure}[t]
    \centering
    \subfloat[$s \mapsto v(0, s\B{1}_d)$, $d=100$]{\includegraphics[width=0.3\textwidth]{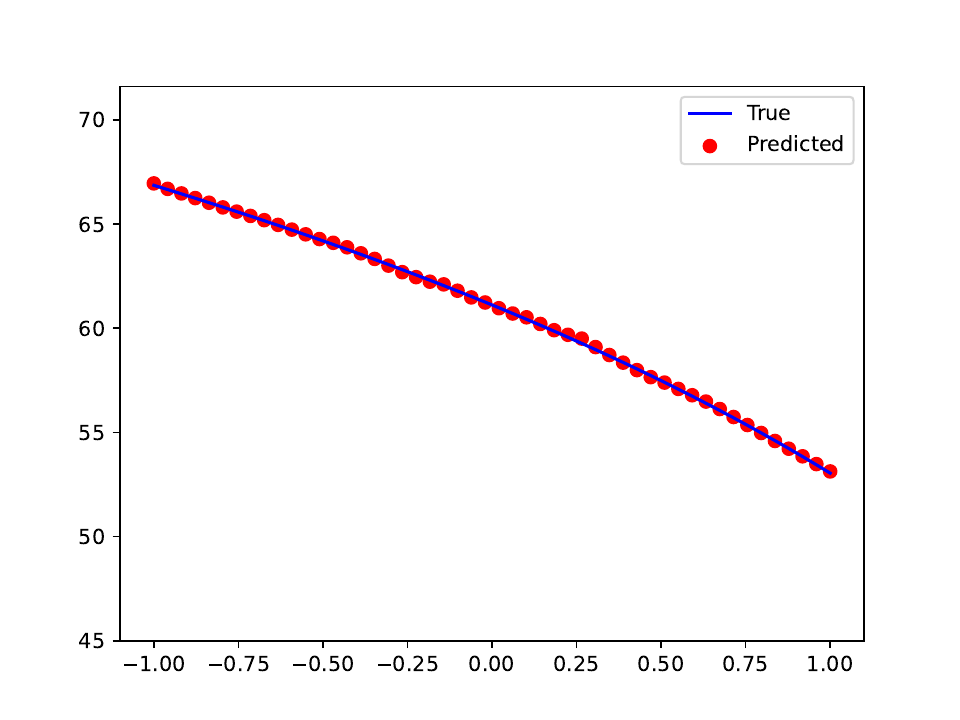}}
    \subfloat[RE vs Iter., $d=100$]{\includegraphics[width=0.28\textwidth]{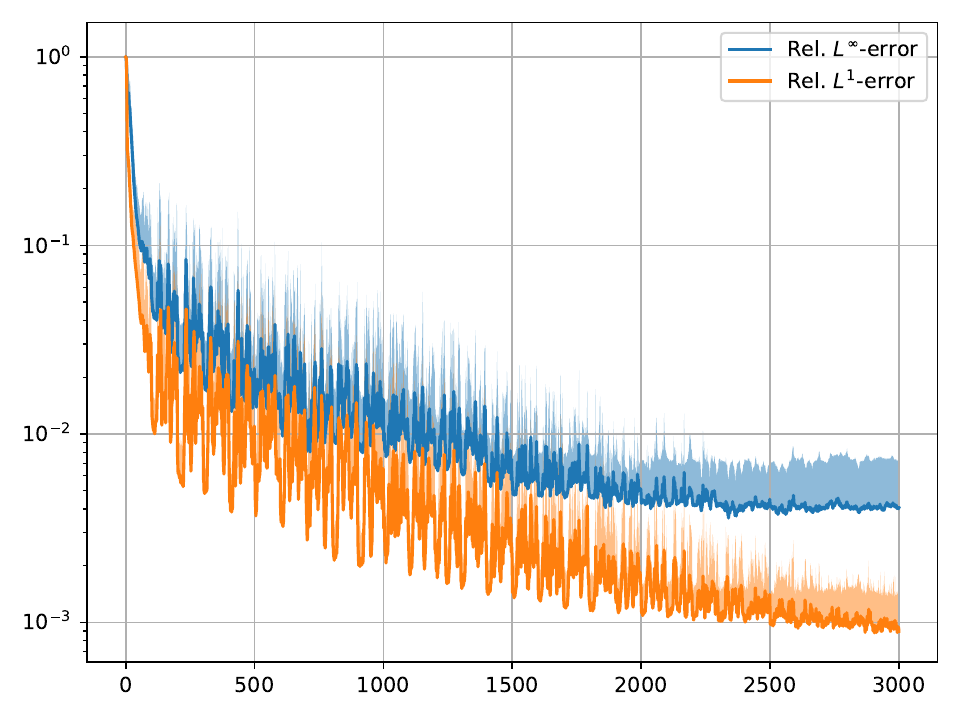}}
    \subfloat[$J(u)$ vs Iter., $d=100$]{\includegraphics[width=0.28\textwidth]{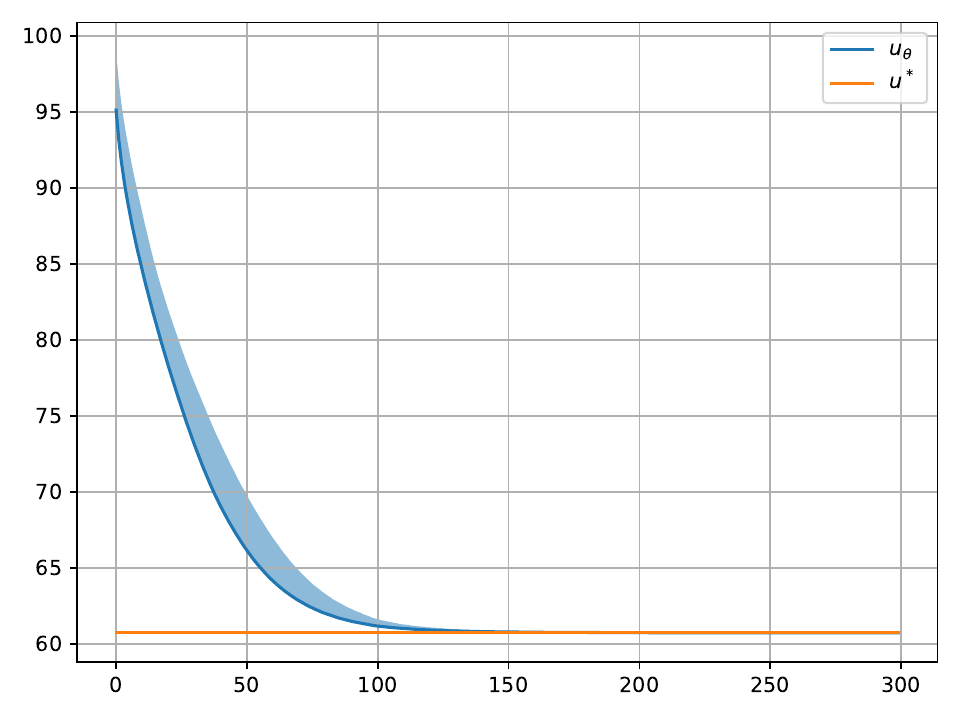}}
    \\
    \subfloat[$s \mapsto v(0, s\B{1}_d)$, $d=500$]{\includegraphics[width=0.3\textwidth]{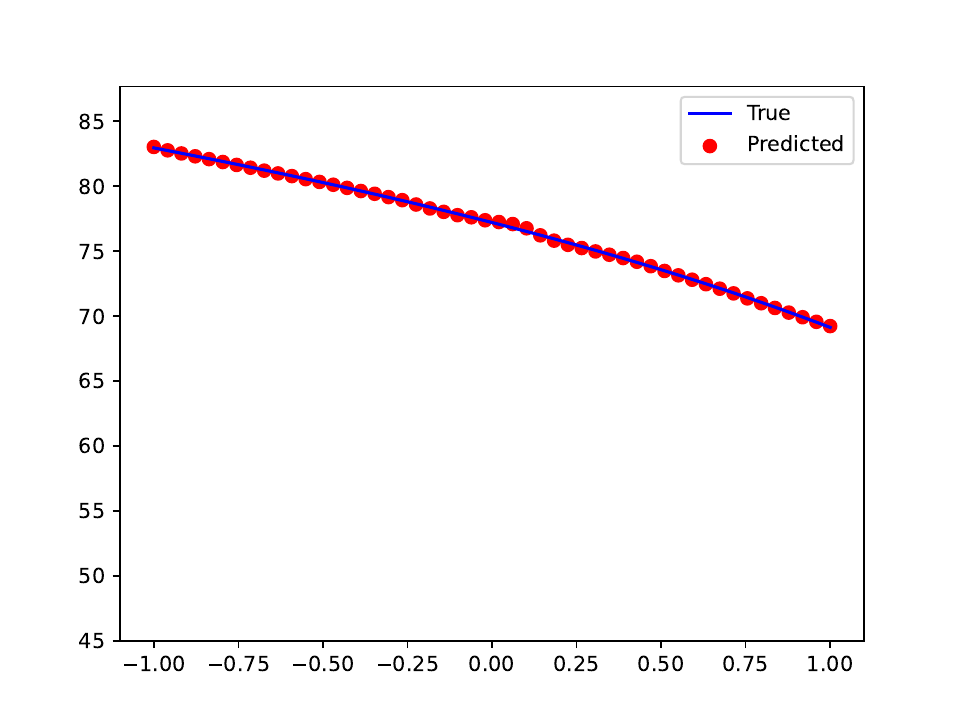}}
    \subfloat[RE vs Iter., $d=500$]{\includegraphics[width=0.28\textwidth]{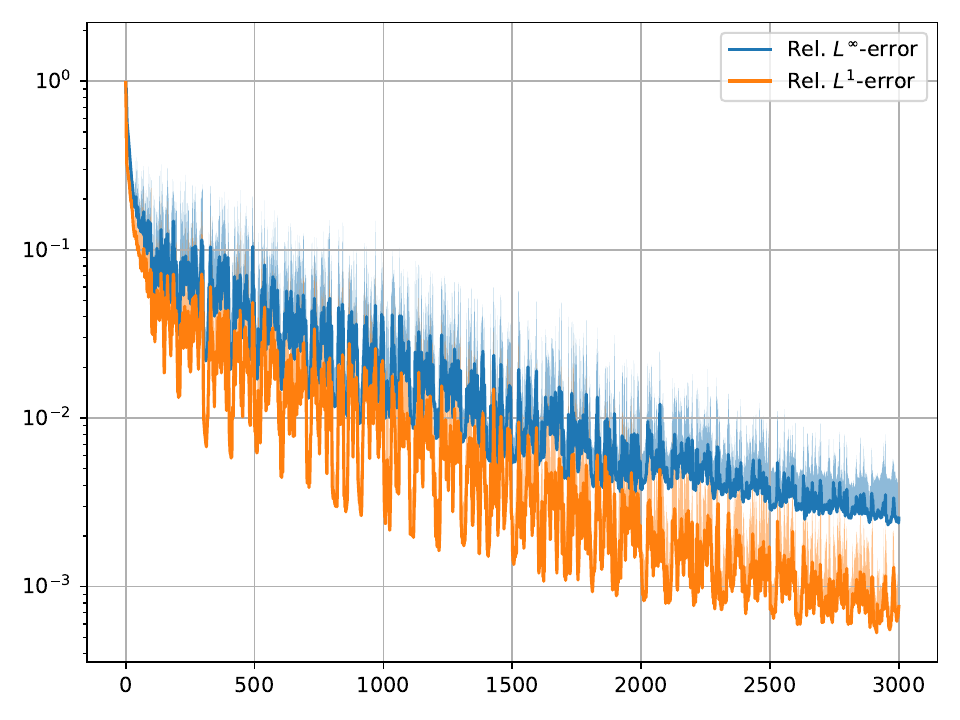}}
    \subfloat[$J(u)$ vs Iter., $d=500$]{\includegraphics[width=0.28\textwidth]{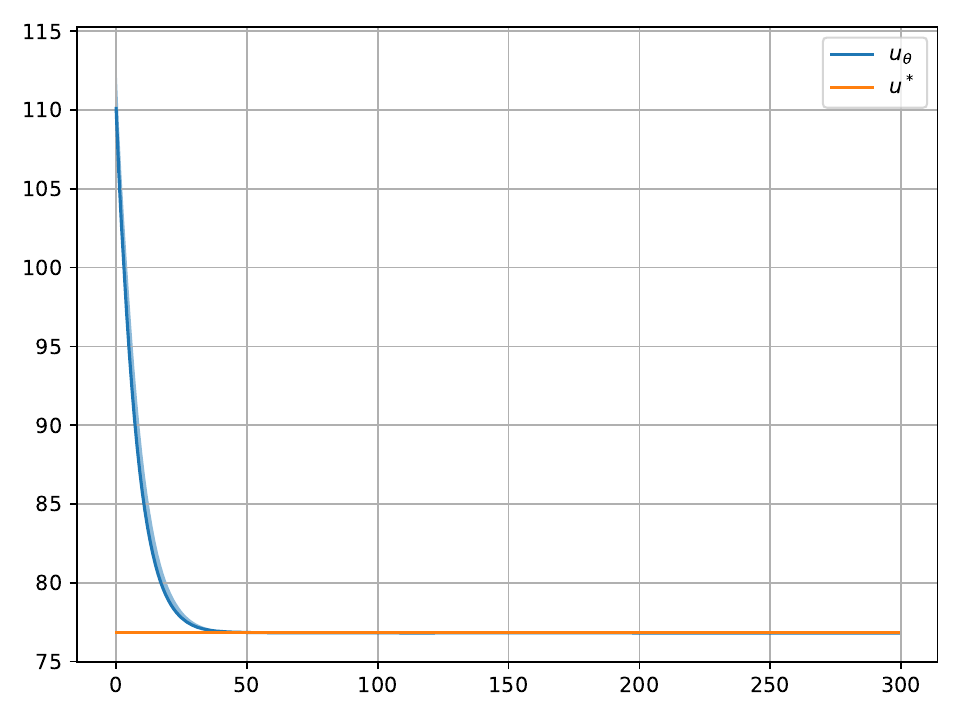}}
    \caption{
    \cb{Numerical results of SOC-MartNet (\Cref{alg_amnet}) for the SOCP with a shifted target in \cref{sec_scop_st}.
    The shaded region represents the mean + $2 \times$ SD of the plotted values across 5 independent runs.}
    }\label{fig_socp_shifted}
\end{figure}

\cb{
\subsection{Non-degenerated HJB equation without explicit \texorpdfstring{$\inf_u H$}{inf\_u H}}\label{sec_eps}
The $\inf_u H$ in the HJB equation \cref{eq_HjbLq} in fact has an explicit expression, so to obtain a HJB equation without an explicit form of $\inf_u H$, we modify it 
by an perturbation of a nonlinear term $\varepsilon \sin(\B{1}_d^{\top} \kappa)$ to the Hamiltonian, i.e.,
\begin{equation}\label{eq_hjbeps}
    \br{\partial_t + b^{\top} \partial_x + \epsilon_1 \Delta_x} v(t, x) + \inf_{\kappa \in \R^d} \br{2 \kappa^{\top} \partial_x v(t, x) + c_1 \abs{\kappa}^2 + \varepsilon \sin(\B{1}_d^{\top} \kappa)} = 0
\end{equation}
where $\B{1}_d$ is given in \eqref{eq_defS2};  $\varepsilon$ is a scalar parameter to be specified, and other parameter settings are the same as HJB-2 introduced in \cref{sec_nondegHjb}.
}

\cb{
Compared with \eqref{eq_HjbLq}, the Hamiltonian in \eqref{eq_hjbeps} has no closed-form minimization due to the added nonlinear term. 
Consequently, this example can be used to test the effectiveness of SOC-MartNet for general problems. 
In the tests, we use a decreasing sequence of the perturbation parameter $\varepsilon = 1, 1/2, 1/4, 1/8$, and the solution of \eqref{eq_hjbeps} is expected to approach HJB-2 as $\varepsilon$ decreases. 
Hence, the solution of HJB-2 (i.e., $\varepsilon =0$ ) is taken as a reference to evaluate the numerical results of SOC-MartNet for \eqref{eq_hjbeps}.
The relevant numerical results are presented in \Cref{fig_socp_eps}, where we can see that the solution provided by SOC-MartNet indeed approaches to the reference solution as $\varepsilon$ decreases.
}

\begin{figure}[t]
    \centering
    \subfloat{\includegraphics[width=0.45\textwidth]{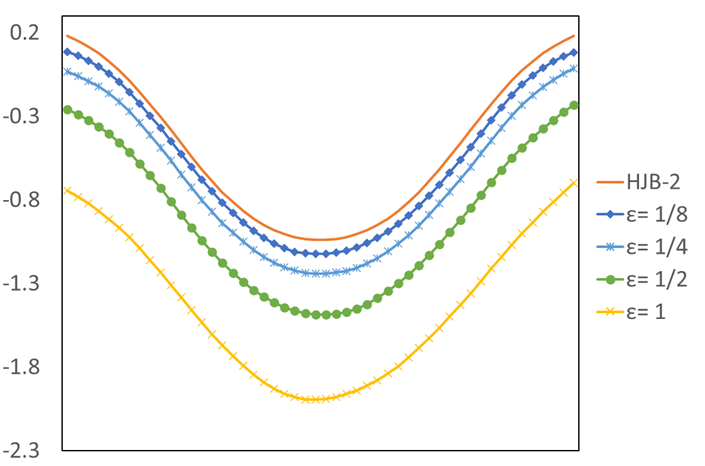}}
    \caption{
    \cb{Numerical results of SOC-MartNet (\Cref{alg_amnet}) for $s \mapsto v(0, s\B{1}_d), s \in [-1,1]$ from \eqref{eq_hjbeps} with $d = 1000$ and $\varepsilon = 1, 1/2, 1/4, 1/8$ (from bottom up). 
    The top orange curve without markers is the reference solution of HJB-2 (i.e. $\varepsilon=0$).}
    }\label{fig_socp_eps}
\end{figure}

\cb{
\subsection{Test on efficiency}\label{sec_dis_eff}

Finally, we investigate the efficiency of our method in comparison with the
deep BSDE proposed by \cite{han2018solving,weinan2017deep} as a benchmark method.
We apply the SOC-MartNet (\cref{alg_amnet}) and the deep BSDE to solve $v(0, x)$ for one point $x = \br{0, 0, \cdots, 0}^{\top} \in \R^d$ as the deep BSDE method is only designed for finding solution at a single point while our method can in fact give solution in a region at once where the $X_0=x$ is sampled. 
The deep BSDE is implemented with the code publicly available from the website\footnote{\url{https://github.com/frankhan91/DeepBSDE}} to solve HJB-2 introduced in \cref{sec_nondegHjb}, where the infimum Hamiltonian $\inf_u H$ is computed by its explicit form.
At each time step, the deep BSDE employs a fully connected neural network with two hidden layers, each containing $W_h$ units, to approximate the policy function. 
Other hyperparameters align with those used in example (13) of \cite{han2018solving}.
For the SOC-MartNet, \cref{alg_amnet} is applied to HJB-2 without using the explicit $\inf_u H$.
Both networks $u_{\alpha}$ and $v_{\theta}$ in SOC-MartNet are structured with four hidden layers, each comprising $W_h$ units. 
The values of $W_h$ used are reported in the numerical results.
Both methods are implemented on a single GPU (NVIDIA A100-SXM4-80GB) with a needed number $I$ of iterations to reach a comparable accuracy,  with $I=1000$ iterations for the SOC-MartNet and $I=2000$ for the deep BSDE.
Table~\ref{tab_db} reports the errors and running times.
}

\cb{
As shown in Table~\ref{tab_db}, although SOC-MartNet does not rely on the explicit $\inf_u H$, it still achieves accuracy, using shorter running times, comparable to the deep BSDE  and
its efficiency stems from the parallel nature in the computation of the loss functions in \eqref{eq_defhatL} and \eqref{eq_defG} across time and space.

\begin{table}[t]
    \caption{\cb{The REs and running times (RTs) of SOC-MartNet (\Cref{alg_amnet}) and deep BSDE for solving HJB-2 introduced in \cref{sec_nondegHjb}.
    The number of iteration steps is $I=1000$ for SOC-MartNet and $I=2000$ for deep BSDE.}
    }\label{tab_db}
    \resizebox{\textwidth}{!}{%
    \begin{tabular}{@{}ll|ll|ll|ll@{}}
    \toprule
                 &      & \multicolumn{2}{c|}{RE} & \multicolumn{2}{|c|}{SD} & \multicolumn{2}{|c}{RT} \\
    Network width & $d$  & Deep BSDE & SOC-MartNet & Deep BSDE& SOC-MartNet & Deep BSDE & SOC-MartNet \\ \midrule
    $W_h = 256$  & 100  & 3.23E-03 & 1.24E-03    & 8.90E-04 & 2.43E-03    & 73       & 48          \\
                 & 300  & 1.18E-03 & 1.14E-03    & 8.19E-04 & 1.20E-03    & 90       & 53          \\
                 & 500  & 1.05E-03 & 2.89E-03    & 5.71E-04 & 1.22E-03    & 116      & 58          \\
                 & 800  & 2.54E-03 & 4.35E-03    & 1.90E-03 & 1.41E-03    & 145      & 73          \\
                 & 1000 & 1.86E-03 & 5.83E-03    & 2.08E-03 & 2.57E-03    & 170      & 118         \\ 
                 &      &          &             &          &             &          &             \\
    $W_h = d+10$ & 100  & 2.86E-03 & 2.94E-03    & 1.09E-03 & 1.67E-03    & 53       & 20          \\
                 & 300  & 3.27E-04 & 8.99E-04    & 1.27E-04 & 1.02E-03    & 103      & 51          \\
                 & 500  & 6.41E-04 & 6.92E-04    & 3.58E-04 & 4.99E-04    & 184      & 103         \\
                 & 800  & 1.28E-03 & 4.00E-03    & 1.14E-03 & 3.78E-03    & 386      & 255         \\
                 & 1000 & 3.77E-03 & 1.11E-03    & 5.87E-03 & 4.26E-04    & 615      & 360         \\
                 \bottomrule
    \end{tabular}
    }
\end{table}
}

\cb{
We further highlight that the efficiency of our method can be significantly enhanced through the use of multiple GPU parallel computing.
Specifically, we apply the SOC-MartNet (\cref{alg_amnet}) to HJB-3, and record the running times under different numbers of GPUs used by DDP.}
The relevant results are presented in \Cref{tab_rt}.
Notably, the SOC-MartNet can be accelerated significantly by the DDP for the problem with dimensionality \cb{$d \geq 500$}.

\begin{table}[t]
    \centering
    \caption{\cb{The running times (unit: second) of SOC-MartNet (\Cref{alg_amnet}) for solving HJB-3 introduced in \cref{sec_nondegHjb}. 
    The algorithm is accelerated by DDP on multiple GPUs. 
    The number of iteration steps and the network widths are set to $I=6000$ and $W_h = d+10$, respectively.}
    }\label{tab_rt}
    \resizebox{!}{!}{%
        \begin{tabular}{@{}llllll@{}}
            \toprule
            GPUs                & $d=100$  & $d=500$ & $d=800$  & $d=1000$  & $d=2000$  \\ \midrule
            1 $\times$ A100     &  153     &  775    & 1350     & 1909     & 5032  \\
            2 $\times$ A100     &  151     &  430    & 721      & 1001     & 2582  \\
            4 $\times$ A100     &  142     &  233    & 393      & 536      & 1387  \\
            8 $\times$ A100     &  148     &  153    & 231      & 302      & 773     \\ \bottomrule
        \end{tabular}%
    }
\end{table}

\section{Conclusions and future work}\label{sec_conclu}

In this paper, we propose a martingale-based DNN method for stochastic optimal controls, SOC-MartNet,  based on the original DeepMartNet \cite{cai2023deepmartnet,cai2023deepmartnet2} combined with adversarial learning.
A Hamiltonian process and a cost process are introduced using a control network and a value network. A loss function for the networks' training is used to ensure the minimum principle for the optimal feedback control as well as the fulfillment of the HJB equation by the value function, and the latter is implemented through a martingale formulation and a training that drives the value function network to satisfy its martingale properties at convergence.
In SOC-MartNet, the martingale property of the cost process is however enforced by an adversarial learning,
whose loss function is built upon the projection property of conditional expectations.
Numerical results show that the proposed SOC-MartNet is effective and efficient for solving HJB-type equations with dimension up to \cb{$10,000$ in a small number of stochastic gradient method iterations (less than $6000$)} for the training, and particularly, for SOCPs where the $\inf H$ has no explicit expressions.

For future work, we will carry out theoretical analysis to study the convergence of the proposed SOC-MartNet and provide some mathematical justification of the observed computational performance of the method.

\section*{Acknowledgments} The authors thank Alain Bensoussan for helpful discussion on the HJB equation and stochastic optimal controls.

\bibliographystyle{siamplain}

\end{document}